\documentclass[review]{elsarticle}

\usepackage{lipsum}
\usepackage{amsfonts}
\usepackage{graphicx}
\usepackage{epstopdf}
\usepackage{algorithmic}
\ifpdf
  \DeclareGraphicsExtensions{.eps,.pdf,.png,.jpg}
\else
  \DeclareGraphicsExtensions{.eps}
\fi
\usepackage{amssymb}
\usepackage{amsthm}
\usepackage{isomath}
\usepackage{mathtools}
\usepackage{tikz}
\usepackage{pgfplots}
\usetikzlibrary{arrows.meta}
  \pgfplotsset{compat=newest}
  \usetikzlibrary{plotmarks}
  \usepackage{grffile}

\pgfplotsset{every axis/.append style={
        scaled ticks = false, 
        tick label style={/pgf/number format/fixed}
    }
}

\usetikzlibrary{external}
\newlength\figureheight 
\newlength\figurewidth 
\usetikzlibrary{patterns}
\usetikzlibrary{plotmarks}
\usetikzlibrary{patterns}
\usepackage{tikz-3dplot}
\usepackage{multirow}
\usepackage{subfigure}
\usepackage{esint} 
\usepackage{multirow}
\usepackage{booktabs}
\usepackage{ulem}
\usepackage{cancel}
\usepackage[autostyle]{csquotes} 
\usepackage{stackengine} 
\usepackage{float}
\usepackage{setspace}
\usepackage{placeins}
\usepackage{lipsum}
\usepackage{amsfonts}
\usepackage{stmaryrd} 
\usepackage{algorithm2e}

\newtheorem{theorem}{Theorem}
\numberwithin{theorem}{section}

\newtheorem{proposition}[theorem]{Proposition}

\usepackage{amsopn}

\usepackage{adjustbox}

\journal{...}

\begin{document}

\begin{frontmatter}

\title{FDTD schemes for Maxwell's equations with embedded perfect electric conductors based on the correction function method}

\author{Y.-M. Law\fnref{myfootnote1}} \author{J.C. Nave\fnref{myfootnote2}}
\address{Department of Mathematics and Statistics, McGill University, Montr\'{e}al, QC H3A 0B9, Canada.}
\fntext[myfootnote1]{yann-meing.law-kamcio@mcgill.ca}
\fntext[myfootnote2]{jean-christophe.nave@mcgill.ca}

%


\begin{abstract}
In this work, 
	we propose staggered FDTD schemes based on the Correction Function Method (CFM) to discretize 
	Maxwell's equations with embedded perfect electric conductor (PEC) boundary conditions.
The CFM uses a minimization procedure to compute a correction to a given FD 
	scheme in the vicinity of the embedded boundary to retain its order. 
The minimization problem associated with CFM approaches is analyzed in the context of Maxwell's equations with 
	embedded boundaries. 
In order to obtain a well-posed minimization problem,
	we propose fictitious interface conditions to fulfill the lack of information, 
	namely the surface current and charge density,
	on the embedded boundary.	
Fictitious interfaces can induce some issues for long time simulations and 
	therefore the penalization coefficient associated with fictitious interface conditions 
	must be chosen small enough.
We introduce CFM-FDTD schemes based on the well-known Yee scheme and 
	a fourth-order staggered FDTD scheme. 
Long time simulations and convergence studies are performed in 2-D for various geometries of the embedded boundary.
CFM-FDTD schemes have shown high-order convergence.
\end{abstract}


\end{frontmatter}


\section{Introduction}
In electromagnetic dynamics, 
	the perfect electric conductor is an important idealized material, 
	that allows surface charges and currents. 
Since the surface charge and current density are often unknown, 	
	PEC walls are modeled through the imposition of boundary conditions 
	that require the continuity of the normal component of the magnetic field and the tangential component 
	of the electric field on the boundary of a domain.

In free-space simulations involving PECs, 
	the interface between a PEC and its surrounding medium is treated 
	as an embedded boundary.	
Embedded boundary conditions can be difficult to treat, 
	particularly in a finite-difference context. 
Indeed, 
	challenges include the development of numerical methods 
	that can handle various complex geometries of the PEC
	without increasing the complexity of a numerical method while retaining high-order accuracy.
It is also worth mentioning that high-order schemes are important 
	to diminish the phase error for long time simulations \cite{Hesthaven2003}. 
	
Many numerical strategies are proposed to achieve high-order accuracy for problems involving 
	embedded PEC walls, 
	such as discontinuous Galerkin (DG) approaches \cite{Hesthaven2002}, 
	pseudospectral time-domain methods (PSTD) \cite{Yang1997,Fan2002,Galagusz2016} or 
	finite-difference time-domain (FDTD) schemes \cite{Yee1992,Jurgens1992,Ditkowski2001,Zhao2010,Wang2013}.
A discontinuous Galerkin approach can treat complex geometries of embedded PEC walls 
	by an appropriate mesh.
However, 
	a large number of unknowns for high-order accuracy is needed for these approaches.
This is due to the use of piecewise polynomial spaces that do not require continuity between 
	two elements of a mesh. 
Various strategies have been proposed to reduce the computational cost 
	of DG based methods, 
	such as parallel computing strategies or particular choices of basis functions \cite{Cockburn2004,Cohen2006}.
It is also worth mentioning that finite-element approaches with non-body-fitted grids 
	have been developed for electromagnetic problems,
	but low-order basis functions have been used \cite{Chen2007,Chen2009,Wang2019}. 
On the other hand, 
	finite-difference time-domain approaches use a simple Cartesian grid and have 
	low computational costs.
However, 	
	the imposition of embedded boundary conditions is far from trivial.
A naive approach is to use the Yee scheme \cite{Yee1966} with a staircased approximation of the 
	embedded boundary. 
Unfortunately, 
	this approach leads to a first-order scheme at best
	and sometimes to non-convergent approximations \cite{Ditkowski2001}. 
To overcome this issue, 
	many FDTD approaches have been proposed, 
	such as overlapping grids \cite{Yee1992}, 
	contour path methods \cite{Jurgens1992,Wang2013} and 
	locally modifications of FD schemes \cite{Ditkowski2001,Zhao2010}.  
A staircase-free FDTD scheme has been proposed in \cite{Ditkowski2001} to recover 
	a second-order scheme without significantly compromising the simplicity of the used FD scheme.
They explicitly impose PEC boundary conditions by locally modifying a finite-difference scheme in 
	the vicinity of the boundary.
Following the same idea,
	a fourth-order finite-difference scheme based on 
	the Matched Interface and Boundary (MIB) method \cite{Zhao2004} 
	has been proposed to handle embedded PEC walls 
	using the vector Helmholtz equation \cite{Zhao2010}.
This FD scheme is obtained 
	by deriving and explicitly imposing jump conditions for PEC walls on
	the embedded boundary.  
However,
	the complexity of a MIB based scheme increases with its order or the complexity of 
	the geometry of the embedded boundary because of the imposition 
	of high-order jump conditions \cite{Yu2007,Zhang2016}.
Finally, 
	pseudospectral methods have the advantage to require less 
	grid points per wavelength than FDTD approaches.
The Fourier and Chebyshev collocation methods with a multidomain strategy 
	have been used to achieve high-order accuracy \cite{Yang1997,Fan2002}.  
These approaches need a multidomain decomposition with an appropriate mesh grid
	for embedded boundary conditions, 
	and therefore
	an additional treatment of interfaces between subdomains is needed.
An alternative approach is to use a Fourier penalty method \cite{Galagusz2016}. 
Although this approach does not need a multidomain strategy, 
	there is some stability issues that limit the order of the 
	method in  two and three dimensions.
 
Another avenue to handle embedded boundary conditions is 
	FD schemes based on the Correction Function Method (CFM).
The CFM, 
	which was inspired by the Ghost Fluid Method (GFM), 
	has been originally developed to treat Poisson's equations with interface jump conditions with 
	arbitrarily complex interfaces. 
FD schemes based on the CFM achieve high-order by means of a minimization problem.
This numerical strategy has been successfully applied to Poisson problems with 
	constant and piecewise constant coefficients, 
	and interface jumps \cite{Marques2011,Marques2017}.
It is also worth mentioning that a CFM based strategy have also been used to treat 3-D Poisson equation with interface 
	jump conditions \cite{Marques2019}.
Afterward, 
	extensions of this method have been used to handle the wave equation and Maxwell's equations with 
	constant coefficients and interface jump conditions \cite{Abraham2018,LawMarquesNave2020}. 
Briefly, 
	the underline assumption of the CFM 
	is that jumps on the interface can be smoothly extended in the vicinity of the interface.
Therefore, 
	a system of partial differential equations (PDEs) coming from the original system of PDEs is derived 
	to model jumps around the interface.
The solution of this system of PDEs is called the correction function.
A square measure of the error associated with the correction function's system of PDEs 
	is then minimized to compute approximations of the correction function. 
Afterward, 
	these approximations are used to correct a given FD 
	scheme that involves nodes in different subdomains. 
It is worth mentioning that the additional computational cost associated with 
	minimization problems of the CFM is not negligible.
Hence, 
	this makes unavoidable the use of parallel computing 
	strategies \cite{Abraham2017}.
 
As mentioned before, 
	a FDTD strategy based on the CFM has been developed to handle Maxwell's equations with constant 
	coefficients and interface jump conditions \cite{LawMarquesNave2020}.
Even though this numerical scheme could also be used to enforce embedded boundary conditions,
	one needs to impose all information on the boundary, 
	that is both normal and tangential components of each electromagnetic field.
This is a major drawback when embedded PEC wall boundary conditions are considered.
As shown in this work, 
	a direct application of the numerical scheme proposed in \cite{LawMarquesNave2020} 
	for PEC wall boundary conditions leads to an ill-posed minimization problem because 
	of the lack of information on the embedded boundary.
Hence, 
	the main goal of this work is to overcome this issue and proposes high-order 
	FDTD schemes based on the CFM, 
	which are referred as CFM-FDTD schemes,
	to handle embedded PEC wall boundary conditions.
The algorithm presented in this paper also provides an important 
	stepping stone towards a CFM-FDTD approach to handle Maxwell's interface problems.
			
We first extend CFM-FDTD schemes for PEC problems 
	for which the surface current and charge density are unknown.
Afterward, 
	we describe a construction of appropriate local patches 
	that are needed for minimization problems.
This construction reduces the computation cost of the CFM while 
	guarantees the uniqueness of the correction function for a given node to 
	be corrected.
We then introduce CFM-FDTD schemes based on the well-known Yee scheme and 
	a fourth-order staggered FDTD scheme. 
Finally, 
	numerical examples based on the transverse magnetic (TM$_z$) mode are performed to verify 
 	the proposed CFM-FDTD schemes.

The paper is organized as follows.
In Section~\ref{sec:defPblm}, 
	we introduce Maxwell's equations with embedded PEC wall conditions. 
The CFM applied to Maxwell's equations is presented in Section~\ref{sec:CFM}.
In this section, 
	we propose an extension of the CFM for problems involving a PEC for which the surface current 
	and charge density are unknown. 
Minimization problems coming from the CFM are analyzed. 
The impact of the CFM on a FDTD scheme is also investigated and 
	the construction of local patches is described. 
We then introduce CFM-FDTD schemes based on the Yee scheme in Section~\ref{sec:Yee} and 
	on a fourth-order staggered FDTD scheme in Section~\ref{sec:fourthOrderScheme}.
Finally, 
	we perform numerical examples to verify the proposed CFM-FDTD schemes in Section~\ref{sec:numEx}.

\section{Definition of the Problem} \label{sec:defPblm}

Assume a domain $\Omega$ subdivided into two subdomains $\Omega^+$ and $\Omega^-$.
The interface $\Gamma$ between subdomains is independent of time 
	and allows solutions, 
	that is the magnetic field $\mathbold{H}$ and the electric field $\mathbold{E}$ in this work, 
	to be discontinuous along it. 
We define $\mathbold{H}^+$ and $\mathbold{E}^+$ as the solutions in $\Omega^+$, 
	and $\mathbold{H}^-$ and $\mathbold{E}^-$ as the solutions in $\Omega^-$.
The jumps are denoted as 
\begin{equation*}
\begin{aligned}
\llbracket \mathbold{H} \rrbracket =&\,\, \mathbold{H}^+ - \mathbold{H}^-, \\
\llbracket \mathbold{E} \rrbracket =&\,\, \mathbold{E}^+ - \mathbold{E}^-.
\end{aligned}
\end{equation*}
We consider the boundary $\partial \Omega$ of $\Omega$ and a given time interval $I = [0,T]$.
Assuming linear media and a periodic domain, 
	Maxwell's equations with interface conditions are given by
\begin{subequations} \label{eq:pblmDefinition}
\begin{align}
\partial_t (\mu\,\mathbold{H}) + \nabla\times \mathbold{E} =&\,\, 0 \quad \text{in } \Omega \times I, \label{eq:Faraday} \\
\partial_t (\epsilon\,\mathbold{E}) - \nabla\times\mathbold{H} =&\,\, 0 \quad \text{in } \Omega \times I , \label{eq:AmpereMaxwell}\\
\nabla\cdot(\epsilon\,\mathbold{E}) =&\,\, \rho \quad \text{in } \Omega \times I , \label{eq:divE}\\
\nabla\cdot(\mu\,\mathbold{H}) =&\,\, 0 \quad \text{in } \Omega \times I , \label{eq:divH}\\
\hat{\mathbold{n}}\times\llbracket \mathbold{E} \rrbracket =&\,\, 0 \quad \text{on } \Gamma \times I ,\label{eq:tangentEInterf}\\
\hat{\mathbold{n}}\times\llbracket \mathbold{H} \rrbracket =&\,\, \mathbold{J}_s(\mathbold{x},t)  \quad \text{on } \Gamma \times I ,\label{eq:tangentHInterf}\\
\hat{\mathbold{n}}\cdot\llbracket \epsilon\,\mathbold{E} \rrbracket =&\,\, \rho_s(\mathbold{x},t) \quad \text{on } \Gamma \times I ,\label{eq:normalEInterf}\\
\hat{\mathbold{n}}\cdot\llbracket \mu\,\mathbold{H} \rrbracket =&\,\, 0 \quad \text{on } \Gamma \times I ,\label{eq:normalHInterf}\\
\mathbold{H}(\mathbold{x},0) =&\,\, \mathbold{H}_0(\mathbold{x})	\quad \text{in } \Omega, \label{eq:InitialCdnH}\\
\mathbold{E}(\mathbold{x},0) =&\,\, \mathbold{E}_0(\mathbold{x})	\quad \text{in } \Omega, \label{eq:InitialCdnE}
\end{align}
\end{subequations}
	where $\mu$ is the magnetic permeability, 
	$\epsilon$ is the electric permittivity,
	$\rho$ is the electric charge density,
	$\mathbold{J}_s$ is the surface current density, 
	$\rho_s$ is the surface charge density and 
	$\hat{\mathbold{n}}$ is the unit normal to the interface $\Gamma$ pointing toward $\Omega^+$.
	Fig.~\ref{fig:typicalDomain} illustrates a typical geometry of a domain $\Omega$.
	Without loss of generality, 
		we assume constant coefficients that are such that $\epsilon, \mu >0$ and $\rho = 0$. 
\begin{figure}[htbp]
 	\centering
		\setlength\figureheight{0.3\linewidth} 
		\setlength\figurewidth{0.35\linewidth} 
		\tikzset{external/export next=false}
%
%
\definecolor{mycolor1}{rgb}{0.00000,0.44700,0.74100}%
\begin{tikzpicture}
\begin{axis}[%
width=0.951\figurewidth,
height=\figureheight,
at={(0\figurewidth,0\figureheight)},
scale only axis,
xmin=0,
xmax=1,
ymin=0,
ymax=1,
axis line style={draw=none},
yticklabels=\empty,
xticklabels=\empty,
tick style={draw=none},
legend style={legend cell align=left,align=left,draw=white!15!black},
ylabel style={yshift=-5pt},xlabel style={yshift=2.5pt}
]
\addplot [color=blue,line width=0.8pt,solid]
  table[row sep=crcr]{%
0.75	0.55\\
0.768383691155712	0.567056285485863\\
0.784855939371776	0.586353079609766\\
0.798569247672799	0.607544505835653\\
0.808783283651588	0.630118192268234\\
0.814908287960172	0.653425855534761\\
0.816539300787274	0.676723220029148\\
0.813479080096982	0.699216694779779\\
0.805748304479489	0.720113788035931\\
0.793582461693234	0.738673991711258\\
0.777415669648107	0.754256830101566\\
0.757852505596638	0.766363942344277\\
0.735629678650627	0.774672444790369\\
0.711570021888213	0.779057373755364\\
0.686531762311539	0.779601705541283\\
0.661356319017503	0.776593244583852\\
0.636817963590651	0.770508511213973\\
0.61357854671968	0.761984593953246\\
0.592150159874138	0.75178070405503\\
0.57286808220072	0.740731832587438\\
0.555875694011952	0.729697420209933\\
0.541122262361211	0.719508274169579\\
0.528373671098689	0.710915085057808\\
0.517235331010121	0.704541799856712\\
0.507185718903124	0.700846803803962\\
0.497618308349223	0.700094370975715\\
0.48788911327206	0.702338193667119\\
0.477366703462387	0.707418035147949\\
0.465481391646248	0.714969718104235\\
0.4517703451712	0.724447815239603\\
0.435915638275791	0.735159603166056\\
0.417772716391469	0.74630812733257\\
0.397387362750012	0.757041649900452\\
0.375	0.76650635094611\\
0.351036977897187	0.773898951759353\\
0.326089339866257	0.778515940109955\\
0.300880372261704	0.779796300355625\\
0.276223968364629	0.777355071772135\\
0.252976437718672	0.771005649468394\\
0.231984821811756	0.760769465763369\\
0.214035011603472	0.746872499529074\\
0.199802985787138	0.729728904825291\\
0.189812299298019	0.709912874076284\\
0.184400561411509	0.68812060227235\\
0.183697076668585	0.66512484910402\\
0.18761311595587	0.641725065201822\\
0.195845484459602	0.61869632675674\\
0.207893209093245	0.596740392003832\\
0.223086334355692	0.576442049038377\\
0.240625045742688	0.558233576556572\\
0.259626683065719	0.542369609807538\\
0.279177704357766	0.528914027386286\\
0.298387346568558	0.517739694017157\\
0.316439622097695	0.50854106036347\\
0.332640397236881	0.500858786781107\\
0.34645661293846	0.494114777441185\\
0.357545209872456	0.487655334569043\\
0.365769976374519	0.480799613212578\\
0.371205307595373	0.472890209038588\\
0.374126697681612	0.463342567597689\\
0.374988630816851	0.451689972653066\\
0.374391337443451	0.437621149210248\\
0.37303858778221	0.421007986082703\\
0.371689260765654	0.401921513178496\\
0.37110581656491	0.380635019844539\\
0.372002990169758	0.357614024441967\\
0.375	0.33349364905389\\
0.380579330947101	0.309044762754784\\
0.389054720761967	0.285130980280279\\
0.400550380065498	0.262659193808722\\
0.414992747983782	0.242526735959632\\
0.432115274321156	0.225568494996845\\
0.45147587740097	0.212507314207483\\
0.472485908299546	0.203910805691148\\
0.494448709733373	0.200157307138778\\
0.516605239008747	0.201413134212458\\
0.538183768940384	0.207622567626085\\
0.558450417734776	0.218511208551704\\
0.576757205393503	0.233602490007808\\
0.592584493652186	0.252246299487896\\
0.605575028595216	0.273657902454885\\
0.615557346626805	0.29696470637777\\
0.622556990666661	0.321257912229455\\
0.6267947702146	0.345645796239216\\
0.628672135768096	0.369305268558684\\
0.628744571230722	0.391528473395405\\
0.627684683890353	0.411761519426597\\
0.626237340401908	0.429632939049313\\
0.625169715962851	0.444970137501007\\
0.625219459117422	0.457802865574245\\
0.627044304722356	0.468353582983459\\
0.631176384055404	0.477015419985697\\
0.637984189046328	0.484319238735193\\
0.647644665720762	0.490891992198984\\
0.660127270910301	0.497409132685517\\
0.67519106704659	0.504544198677694\\
0.692395100958555	0.512918883655448\\
0.711121467043621	0.523056852822891\\
0.730609647080229	0.535344325657581\\
0.75	0.55\\
};
\addplot [color=black,line width=1.5pt,solid]
  table[row sep=crcr]{%
1	1\\
1	0\\
0	0\\
0	1\\
1	1\\
1	1\\
};
\addplot [color=blue,line width=0.8pt,solid,->]
  table[row sep=crcr]{%
0.5926	0.2522\\
0.675925	0.159025\\
};
\end{axis}
\draw (2.5,3.0) node {\color{blue}$\Gamma$};	
\draw (2.9,0.9) node {\color{blue}$\hat{n}$};
\draw (2.15,2.1) node {\color{black}$\Omega^-$};
\draw (0.7,3.2) node {\color{black}$\Omega^+$};
\draw (-0.35,0.25) node {\color{black}$\partial\Omega$};
\end{tikzpicture}%
		\caption{Geometry of a domain $\Omega$ with an interface $\Gamma$.}
\label{fig:typicalDomain}
\end{figure}
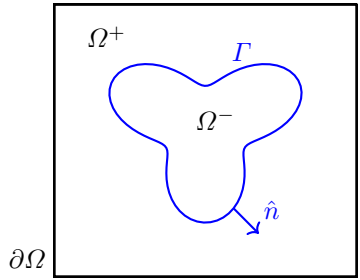

For problems involving a PEC, 
	we often do not know the surface current density $\mathbold{J}_s$ and the surface charge density $\rho_s$. 
Interface conditions \eqref{eq:tangentEInterf}-\eqref{eq:normalHInterf} are then reduced to 
\begin{equation*}
\begin{aligned}
	\hat{\mathbold{n}}\times\llbracket \mathbold{E} \rrbracket =&\,\, 0 \quad \text{on } \Gamma \times I ,\\
	\hat{\mathbold{n}}\cdot\llbracket \mu\,\mathbold{H} \rrbracket =&\,\, 0 \quad \text{on } \Gamma \times I.
\end{aligned}
\end{equation*}
Let us assume that subdomain $\Omega^-$ is a PEC,
	we then have $\mathbold{E}^-(\mathbold{x},t) = 0, \forall (\mathbold{x},t)\in\Omega^-\times I$.
Considering that the initial condition of the magnetic field is given by $\mathbold{H}^-(\mathbold{x},0) = 0$, 
	we also have $\mathbold{H}^-(\mathbold{x},t) = 0, \forall (\mathbold{x},t)\in\Omega^-\times I$.
Thus, 
	interface conditions on $\Gamma$ can be considered as embedded boundary conditions, 
	given by
\begin{equation*}
\begin{aligned}
	\hat{\mathbold{n}}\times\mathbold{E}^+=&\,\, 0 \quad \text{on } \Gamma \times I ,\\
	\hat{\mathbold{n}}\cdot (\mu\,\mathbold{H}^+) =&\,\, 0 \quad \text{on } \Gamma \times I,
\end{aligned}
\end{equation*}
 for $\Omega^+$.
In this work, 
 	we focus on problems involving PECs and therefore assume $\Gamma$ to be an embedded boundary of $\Omega^+$.

\section{Correction Function Method} \label{sec:CFM}

The smoothness of solutions is important when one wants to use FD schemes. 
Realizing that problem \eqref{eq:pblmDefinition} can have discontinuous solutions, 
	standard FD schemes cannot {\it a priori} be used around the embedded boundary $\Gamma$. 
The correction function method allows one to circumvent this issue.
The purpose of the CFM is to find a correction for a finite difference approximation that 
	involves nodes that belong to different subdomains. 
To find such a correction, 
	the CFM assumes that solutions in $\Omega^+\times  I$ and $\Omega^- \times I$ can be 
	smoothly extended in a small region $\Omega_\Gamma \times I$, 
	where $\Omega_\Gamma \subset \Omega$ encloses the embedded boundary $\Gamma$,
	in such a way that the original PDE is still satisfied.
A functional that is a square measure of the error of a PDE that describes the behaviour of jumps 
	or correction functions 
	in the vicinity of the embedded boundary is derived.
This functional is then minimized in a discrete functional space 
	to obtain approximations of the correction function in $\Omega_\Gamma \times I$.
In practice,
	we define a local patch $\Omega_\Gamma^h \subset \Omega_\Gamma$ for which  
	the correction function needs to be computed at a node $\mathbold{x}_c\in\Omega_\Gamma^h$
	and a time interval $I_\Gamma^h = [t_n - \Delta t_{\Gamma}, t_n]$. 
The additional computational cost associated with the CFM is not negligible 
	when compared with the original FDTD scheme. 
In fact, 
	the CFM consumes most of the computational time.
However, 
	a parallel implementation of the computation of correction functions 
	can be performed since minimization problems associated with local patches 
	are independent. 
We refer to \cite{Abraham2017} for more details about the benefits of a parallel implementation of the CFM.

In the following,
	we summarize the procedure for the CFM
	applied on Maxwell's equations when electromagnetic fields are known on the embedded boundary \cite{LawMarquesNave2020}.
Afterward, 
	we present an analysis of the minimization problem that is needed for the CFM.
The functional to be minimized
	is then modified and analyzed 
	for embedded PEC walls for which the surface current and charge density are 
	unknown.
We also investigate the impact of such a modification on a FDTD scheme.
We then describe a construction of local patches that reduces the computation cost of the CFM and 
	ensures an appropriate representation of the embedded boundary within the local patch.

Let us first introduce some notations.	
The inner product in $L^2\big(\Omega_{\Gamma}^{h}\times I_\Gamma^h\big)$ 
	is defined by
$$\langle\mathbold{v},\mathbold{w}\rangle = \int\limits_{I_\Gamma^h}\!\int\limits_{\Omega_{\Gamma}^{h}}\!\!\mathbold{v}\cdot\mathbold{w}\,\mathrm{d}V\,\mathrm{d}t$$
and we also use the notation
	$$\langle \mathbold{v},\mathbold{w} \rangle_{\Gamma} = \int\limits_{I_\Gamma^h}\!\int\limits_{\Omega_{\Gamma}^{h}\cap\Gamma}\!\!\mathbold{v}\cdot\mathbold{w} \, \mathrm{d}S\,\mathrm{d}t$$
	for legibility. 
The correction functions are defined as
\begin{equation*}
\begin{aligned}
\mathbold{D}_H =&\,\, \mathbold{H}^+ - \mathbold{H}^-, \\
\mathbold{D}_E =&\,\, \mathbold{E}^+ - \mathbold{E}^-.
\end{aligned}
\end{equation*}
	
Let us first assume that the surface current density and charge density are known. 
Following the same procedure as in \cite{LawMarquesNave2020}, 
	one can obtain the following quadratic functional to minimize
\begin{equation*}
\begin{aligned}
J(\mathbold{D}_H,&\mathbold{D}_E) = \frac{\ell_h}{2} \, \big\langle\mu\,\partial_t \mathbold{D}_H + \nabla\times\mathbold{D}_E,\mu\,\partial_t \mathbold{D}_H + \nabla\times\mathbold{D}_E\big\rangle \\
+&\,\, 
\frac{\ell_h}{2} \, \langle \epsilon\,\partial_t \mathbold{D}_E - \nabla\times\mathbold{D}_H,\epsilon\,\partial_t \mathbold{D}_E - \nabla\times\mathbold{D}_H \big\rangle\\
+&\,\, \frac{c_p}{2} \,\big\langle \hat{\mathbold{n}}\times\mathbold{D}_H - \mathbold{J}_s,\hat{\mathbold{n}}\times\mathbold{D}_H - \mathbold{J}_s\big\rangle_{\Gamma} + \frac{c_p}{2}\,\big\langle\hat{\mathbold{n}}\cdot \mathbold{D}_H, \hat{\mathbold{n}}\cdot \mathbold{D}_H \big\rangle_{\Gamma}\\
+&\,\, \frac{c_p}{2} \, \big\langle\hat{\mathbold{n}}\times\mathbold{D}_E,\hat{\mathbold{n}}\times\mathbold{D}_E \big\rangle_{\Gamma}+\frac{c_p}{2} \, \big\langle\hat{\mathbold{n}}\cdot \mathbold{D}_E - \tfrac{\rho_s}{\epsilon},\hat{\mathbold{n}}\cdot \mathbold{D}_E - \tfrac{\rho_s}{\epsilon} \big\rangle_{\Gamma} ,
\end{aligned}
\normalsize
\end{equation*}	
	where $c_p>0$ is a penalization coefficient and 
	$\ell_h$ is the length in space of the patch.
We scale the integral over the domain by $\ell_h$ to ensure that all terms 
	in the functional $J$ behave in a similar way when the 
	computational grid is refined \cite{LawMarquesNave2020}.
To guarantee the divergence-free constraint \eqref{eq:divE} and \eqref{eq:divH}, 
	we minimize the functional $J$ in a divergence-free space-time polynomial spaces,
	namely
\begin{equation*}
	V = \big\{ \mathbold{v} \in \big[P^k\big(\Omega_{\Gamma}^{h}\times I_\Gamma^h\big)\big]^3 : \nabla\cdot\mathbold{v} = 0 \big\},
\end{equation*}
	where $P^k$ denotes the space of polynomials of degree $k$.
It is worth mentioning that basis functions of $V$ are based on high-degree divergence-free basis functions proposed in \cite{Cockburn2004} for discontinuous Galerkin approaches.
The problem statement is then 
\begin{equation} \label{eq:minPblm}
\text{Find } (\mathbold{D}_H,\mathbold{D}_E) \in V \times W \text{ such that } (\mathbold{D}_H,\mathbold{D}_E) \in  \underset{\mathbold{v}\in V, \mathbold{w} \in W}{\arg\min}J(\mathbold{v},\mathbold{w}),
\end{equation}
	where $W=V$.
The following proposition shows that the quadratic functional $J$ has a global minimizer
 	when finite-dimensional functional spaces are used and for an appropriate choice 
	of the penalization coefficient $c_p$. 
\begin{proposition} \label{lem:minAnalysis} 
Let us consider problem \eqref{eq:minPblm}, 
	and assume $\mathbold{v}$ and $\mathbold{w}$ to be basis functions of $V$ and $W$.
There exists a positive constant $\tilde{c} = \frac{c_p}{\ell_h}$ for which the functional $J$ has a unique global minimizer.   
\end{proposition}
\begin{proof}
Let us first notice that problem \eqref{eq:minPblm} is an unconstrained minimization problem.
In this case, 
	if the hessian matrix coming from a quadratic functional is positive definite, 
	then the critical point is a global minimum. 
Let us consider $\mathbold{v}$ and $\mathbold{w}$ to be basis functions of $V$ and $W$.
The hessian matrix coming from problem \eqref{eq:minPblm} is given by 
\begin{equation*}
	\begin{bmatrix}
	\ell_h\,a + c_p\,b & \ell_h\,e\\
	\ell_h\,e & \ell_h\,c + c_p\,d
	\end{bmatrix} ,
\end{equation*}
	where 
\begin{equation*} 
	\begin{aligned}
	a =&\,\, \langle \mu\,\partial_t\mathbold{v},\mu\,\partial_t\mathbold{v}\rangle+\langle\nabla\times\mathbold{v},\nabla\times\mathbold{v}\rangle,\\
	b =&\,\, \langle \hat{\mathbold{n}}\times\mathbold{v},\hat{\mathbold{n}}\times\mathbold{v}\rangle_{\Gamma} + \langle\hat{\mathbold{n}}\cdot\mathbold{v},\hat{\mathbold{n}}\cdot\mathbold{v}\rangle_{\Gamma},\\
	c =&\,\, \langle\nabla\times\mathbold{w},\nabla\times\mathbold{w}\rangle + \langle\epsilon\,\partial_t \mathbold{w},\epsilon\,\partial_t \mathbold{w}\rangle,\\
	d =&\,\, \langle \hat{\mathbold{n}}\times\mathbold{w},\hat{\mathbold{n}}\times\mathbold{w}\rangle_{\Gamma} + \langle\hat{\mathbold{n}}\cdot\mathbold{w},\hat{\mathbold{n}}\cdot\mathbold{w}\rangle_{\Gamma},\\
	e = &\,\, \langle \mu\,\partial_t \mathbold{v},\nabla\times\mathbold{w}\rangle - \langle \epsilon\,\partial_t \mathbold{w},\nabla\times\mathbold{v}\rangle.
	\end{aligned}
\end{equation*}
Requiring the hessian matrix to be positive definite, 
	one finds the following conditions:
\begin{align}
	\ell_h\,a + c_p\,b >&\,\, 0, \label{eq:minCdn1}\\ 
	(\ell_h\,a + c_p\,b)\,(\ell_h\,c + c_p\,d) -  (\ell_h\,e)^2 >&\,\, 0. \label{eq:minCdn2}
\end{align}
Since $\mathbold{v}$ and $\mathbold{w}$ are basis functions of $V$ and $W$, 
	we have $\mathbold{v}\neq 0$ and $\mathbold{w}\neq0$.
Noticing that $\mathbold{v}$ and $\mathbold{w}$ cannot be both orthogonal and collinear 
	to $\hat{\mathbold{n}}$ except for the zero element, 
	we also have $b\neq0$ and $d\neq0$. 
Hence, 
	condition \eqref{eq:minCdn1} is always satisfied.	
However, 
	condition \eqref{eq:minCdn2} leads to the following criterion 
\begin{equation} \label{eq:cdnHessian}
	\tilde{c} > \frac{e^2-a\,c}{\tilde{c}\,b\,d} - \frac{a\,d+b\,c}{b\,d},
\end{equation}
	where 
	$\tilde{c} = \frac{c_p}{\ell_h}$.
For a sufficiently large $\tilde{c}$, 
	criterion \eqref{eq:cdnHessian} is satisfied.
\end{proof}


Let us now investigate the impact of the CFM based on minimization problem \eqref{eq:minPblm} 
	on a given FD scheme.
Assume a spatial finite difference operator noted $L$, 
	such that 
\begin{equation} \label{eq:withoutCFM}
	\partial_t \mathbold{U}(t) + L\,\mathbold{U}(t) = 0,
\end{equation}
	where $\mathbold{U}$ 
	is a vector containing $m$ unknowns that estimate electromagnetic fields on the grid points in space.
We consider that a correction is needed at $r$ nodes in the vicinity of the embedded boundary, 
	system \eqref{eq:withoutCFM} then becomes 
\begin{equation} \label{eq:withCFM}
	\partial_t \mathbold{U}(t) + L\,\mathbold{U}(t) + L\,A\,\mathbold{D}(t) = 0,
\end{equation}
	where $\mathbold{D}$ 
	is a vector containing $r$ values of the correction function coming from problem \eqref{eq:minPblm} and $A$ is a rectangular matrix of dimension $m\times r$ with either $0$ or $\pm 1$ as components depending where 
	the correction is needed.
Hence, 
	the correction function can therefore be considered as a time-dependent force term.
From Proposition~\ref{lem:minAnalysis}, 
	we have the existence and unicity of the coefficients of polynomial approximations of the correction function
	for an appropriate $c_p$. 
Since $\mathbold{D}_H(\mathbold{x}, t)$ and $\mathbold{D}_E(\mathbold{x},t)$ are polynomial functions, 
	and $\Omega_{\Gamma}^h\times I_{\Gamma}^h$ is a compact domain, 
	these approximations and their derivatives are bounded in the infinity norm. 
It is therefore sufficient to investigate the stability of the original FDTD scheme \eqref{eq:withoutCFM} to 
	identify any time-step criteria of the corresponding CFM-FDTD scheme (c.f. Theorem 5.1.1 in \cite{Gustafsson1995}).
As for the consistency, 
	it can be shown that the order in space of a CFM-FDTD scheme is $\min\{k,n\}$,
	where $k$ is the degree of polynomial approximations of the correction function and $n$ is the order of 
	the FD scheme in space (c.f. Proposition~\ref{lem:errorAnalysis}).

\subsection{PEC wall boundary conditions} \label{sec:PECwall}
Let us now focus on PEC wall boundary conditions.
By Proposition~\ref{lem:minAnalysis}, 
	one cannot just neglect interface conditions \eqref{eq:tangentHInterf} 
	and \eqref{eq:normalEInterf} for PEC problems for which $\mathbold{J}_s$ 
	and $\rho_s$ are unknown because 
	this leads to an ill-posed minimization problem.
To circumvent this issue, 
	we propose to use fictitious interface conditions, 
	given by
\begin{equation} \label{eq:fictInterfCdn}
\begin{aligned}
\hat{\mathbold{n}}_{1,i}\times(\mathbold{E}^+-\mathbold{E}^*) =&\,\, 0 \quad \text{on} \quad \Gamma_{1,i} \times I  \quad \text{for} \quad i=1,\ldots,N_1,\\
\hat{\mathbold{n}}_{2,i}\times(\mathbold{H}^+-\mathbold{H}^*) =&\,\, 0  \quad \text{on} \quad \Gamma_{2,i} \times I \quad \text{for} \quad i=1,\ldots,N_2,\\
\hat{\mathbold{n}}_{3,i}\cdot(\mathbold{E}^+-\mathbold{E}^*) =&\,\, 0 \quad \text{on} \quad \Gamma_{3,i} \times I \quad \text{for} \quad i=1,\ldots,N_3,\\
\hat{\mathbold{n}}_{4,i}\cdot(\mathbold{H}^+-\mathbold{H}^*) =&\,\, 0 \quad \text{on} \quad \Gamma_{4,i} \times I \quad \text{for} \quad i=1,\ldots,N_4,
\end{aligned}
\end{equation}
where $N_k$ is the number of fictitious interfaces $\Gamma_{k,i} \subset \Omega^+\cap\Omega_{\Gamma}^h$ 
	for $k=1,\dots,4$,
    	$\hat{\mathbold{n}}_{k,i}$ is the unit normal associated with fictitious interface $\Gamma_{k,i}$,
	and $\mathbold{H}^*$ and $\mathbold{E}^*$ are respectively finite difference approximations 
	of $\mathbold{H}^+$ and $\mathbold{E}^+$ in $\Omega^+$.
In subsection~\ref{sec:fictInterfImplementation}, 
	we provide more details on the implementation of these fictitious interface conditions.
The functional to minimize is then given by
\small
\begin{equation} \label{eq:functionalPEC}
\begin{aligned}
\tilde{J}(\mathbold{D}_H,&\mathbold{D}_E) = \frac{\ell_h}{2} \, \big\langle\mu\,\partial_t \mathbold{D}_H + \nabla\times\mathbold{D}_E,\mu\,\partial_t \mathbold{D}_H + \nabla\times\mathbold{D}_E\big\rangle \\
+&\,\, 
\frac{\ell_h}{2} \, \langle \epsilon\,\partial_t \mathbold{D}_E - \nabla\times\mathbold{D}_H,\epsilon\,\partial_t \mathbold{D}_E - \nabla\times\mathbold{D}_H \big\rangle\\
+&\,\, \frac{c_p}{2} \, \big\langle\hat{\mathbold{n}}\times\mathbold{D}_E ,\hat{\mathbold{n}}\times\mathbold{D}_E\big\rangle_{\Gamma} + \frac{c_p}{2}\,\big\langle\hat{\mathbold{n}}\cdot \mathbold{D}_H, \hat{\mathbold{n}}\cdot \mathbold{D}_H\big\rangle_{\Gamma}\\
+&\,\, \frac{c_{f}}{2\,N_{E}} \, \sum_{i=1}^{N_1} \big\langle \hat{\mathbold{n}}_{1,i}\times (\mathbold{D}_E - \mathbold{E}^*),\hat{\mathbold{n}}_{1,i}\times(\mathbold{D}_E- \mathbold{E}^*)\big\rangle_{\Gamma_{1,i}} \\
+&\,\, \frac{c_{f}}{2\,N_{H}} \, \sum_{i=1}^{N_2} \big\langle \hat{\mathbold{n}}_{2,i}\times(\mathbold{D}_H - \mathbold{H}^*),\hat{\mathbold{n}}_{2,i}\times(\mathbold{D}_H- \mathbold{H}^*)\big\rangle_{\Gamma_{2,i}} \\
+&\,\, \frac{c_{f}}{2\,N_{E}} \, \sum_{i=1}^{N_3}  \big\langle\hat{\mathbold{n}}_{3,i}\cdot (\mathbold{D}_E - \mathbold{E}^*),\hat{\mathbold{n}}_{3,i}\cdot (\mathbold{D}_E -  \mathbold{E}^*) \big\rangle_{\Gamma_{3,i}} \\
+&\,\, \frac{c_{f}}{2\,N_{H}} \, \sum_{i=1}^{N_4}  \big\langle\hat{\mathbold{n}}_{4,i}\cdot (\mathbold{D}_H -  \mathbold{H}^*),\hat{\mathbold{n}}_{4,i}\cdot (\mathbold{D}_H - \mathbold{H}^*) \big\rangle_{\Gamma_{4,i}},
\end{aligned}
\normalsize
\end{equation}
\normalsize
	where $c_f>0$ and $c_p>0$ are penalization coefficient,
	and $N_{H}=N_2+N_4$ and $N_{E}=N_1+N_3$ are the total number of fictitious interfaces for 
	each electromagnetic field. 
The problem statement is then 
\begin{equation} \label{eq:minPblmPEC}
\text{Find } (\mathbold{D}_H,\mathbold{D}_E) \in V \times W \text{ such that } (\mathbold{D}_H,\mathbold{D}_E) \in  \underset{\mathbold{v}\in V, \mathbold{w} \in W}{\arg\min}\tilde{J}(\mathbold{v},\mathbold{w}),
\end{equation}
	where $W=V$.	
The following proposition guarantees that there is a global minimizer for an appropriate choice of the penalization 
	coefficient $c_f$ and fictitious interfaces.
\begin{proposition} \label{lem:minAnalysisFictInterf} 
Let us consider problem \eqref{eq:minPblmPEC},
	and assume $\mathbold{v}$ and $\mathbold{w}$ to be basis functions of $V$ and $W$.
Moreover, 
	assume that there is collinear and orthogonal fictitious interfaces to each plane defined by the axis 
	of the Cartesian coordinate system and for each type of fictitious interface conditions \eqref{eq:fictInterfCdn}.
There exists a positive constant $\tilde{c} = \frac{c_f}{\ell_h}$ for which the functional $\tilde{J}$ has a unique global minimizer.   
\end{proposition}
\begin{proof}
The demonstration is similar to the proof presented in Proposition~\ref{lem:minAnalysis}.
The Hessian matrix coming from problem \eqref{eq:minPblmPEC} is given by 
\begin{equation*}
	\begin{bmatrix}
	\ell_h\,a + c_p\,b + c_f\,\tilde{b} & \ell_h\,e\\
	\ell_h\,e & \ell_h\,c + c_p\,d + c_f\,\tilde{d}
	\end{bmatrix} ,
\end{equation*}
	where $a$, 
	$c$ and $e$ are the same as in Proposition~\ref{lem:minAnalysis}, 
	but 
\begin{equation*}
	\begin{aligned}
	b =&\,\, \langle\hat{\mathbold{n}}\cdot\mathbold{v},\hat{\mathbold{n}}\cdot\mathbold{v}\rangle_{\Gamma},\\
	d =&\,\, \langle \hat{\mathbold{n}}\times\mathbold{w},\hat{\mathbold{n}}\times\mathbold{w}\rangle_{\Gamma},\\
	\tilde{b} =&\,\,   \frac{1}{N_{H}} \, \sum_{i=1}^{N_2}\langle \hat{\mathbold{n}}_{2,i}\times\mathbold{v},\hat{\mathbold{n}}_{2,i}\times\mathbold{v}\rangle_{\Gamma_{2,i}}  +  \frac{1}{N_{H}} \, \sum_{i=1}^{N_4} \langle\hat{\mathbold{n}}_{4,i}\cdot\mathbold{v},\hat{\mathbold{n}}_{4,i}\cdot\mathbold{v}\rangle_{\Gamma_{4,i}},\\
	\tilde{d} =&\,\,  \frac{1}{N_{E}} \,\sum_{i=1}^{N_1} \langle \hat{\mathbold{n}}_{1,i}\times\mathbold{w},\hat{\mathbold{n}}_{1,i}\times\mathbold{w}\rangle_{\Gamma_{1,i}} +  \frac{1}{N_{E}} \,  \sum_{i=1}^{N_3} \langle\hat{\mathbold{n}}_{3,i}\cdot\mathbold{w},\hat{\mathbold{n}}_{3,i}\cdot\mathbold{w}\rangle_{\Gamma_{3,i}}.
	\end{aligned}
\end{equation*}
One can notice that $b\geq0$ and $d\geq0$.
However,
	since there is fictitious interfaces that are collinear and orthogonal to each plane defined 
	by the axis of the coordinate system and for each fictitious interface condition, 
	we have $\tilde{b}\neq0$ and $\tilde{d}\neq0$. 
Requiring the Hessian matrix to be positive definite, 
	one finds the following conditions:
\begin{align}
	\ell_h\,a + c_p\,b + c_f\,\tilde{b} >&\,\, 0, \label{eq:minCdn1FictInterf}\\ 
	(\ell_h\,a + c_p\,b + c_f\,\tilde{b} )\,(\ell_h\,c + c_p\,d + c_f\,\tilde{d}) -  (\ell_h\,e)^2 >&\,\, 0. \label{eq:minCdn2FictInterf}
\end{align}
Condition \eqref{eq:minCdn1FictInterf} is always satisfied and 
	condition \eqref{eq:minCdn2FictInterf} leads to the following criterion 
\begin{equation} \label{eq:cdnHessianFictInterf}
	\tilde{c} > \frac{e^2-a\,c}{\tilde{c}\,\tilde{b}\,\tilde{d}} - \frac{c_p\,(a\,d+b\,c)}{c_f\,\tilde{b}\,\tilde{d}} - \frac{a\,\tilde{d}+\tilde{b}\,c}{\tilde{b}\,\tilde{d}}
		- \frac{c_p\,(b\,\tilde{d}+\tilde{b}\,d)}{\ell_h\,\tilde{b}\,\tilde{d}}-\frac{c_p^2\,b\,d}{\ell_h\,c_f\,\tilde{b}\,\tilde{d}},
\end{equation}
	where 
	$\tilde{c} = \frac{c_f}{\ell_h}$.
For a sufficiently large $\tilde{c}$, 
	criterion \eqref{eq:cdnHessianFictInterf} is satisfied.
\end{proof}
Let us now investigate the impact of the CFM based on minimization problem \eqref{eq:minPblmPEC} 
	on the stability of the original FD scheme. 
As shown previously, 
	we have a system of the form \eqref{eq:withCFM} with $\mathbold{D}$ coming from 
	problem~\eqref{eq:minPblmPEC}.
However, 
	this is not completely accurate since fictitious interface conditions \eqref{eq:fictInterfCdn}
	depend on FD solutions in $\Omega^+$.
Computing Gateaux derivatives and using a necessary condition to obtain a minimum, 
	we obtain
$$M\,\mathbold{c} = c_f\,\mathbold{b}_f + c_p\,\mathbold{b}_{\Gamma},$$
	where 
	$\mathbold{c}$ contains coefficients of polynomial approximations of the correction function,
	and $\mathbold{b}_f$ and $\mathbold{b}_{\Gamma}$ are associated with terms using 
	respectively fictitious interfaces and embedded boundaries.
Moreover, 
	we can define a linear operator $B$ that is such that 
	$\mathbold{b}_f = B\,\mathbold{U}$. 
From Proposition~\ref{lem:minAnalysisFictInterf},
	problem~\eqref{eq:minPblmPEC} is well-posed for appropriate fictitious interfaces, 
	$c_p$ and $c_f$,
	which leads to 
	$\mathbold{c} = c_f\,M^{-1}\, B\,\mathbold{U} + c_p\,M^{-1}\,\mathbold{b}_{\Gamma}$.
Hence,
	we have 
\begin{equation*} 
	\partial_t \mathbold{U}(t) + L\,(I + c_f\,A\,M^{-1}\, B)\,\mathbold{U}(t) + c_p\,L\,A\,M^{-1}\,\mathbold{b}_{\Gamma} = 0,
\end{equation*}
where $L$ is a finite difference operator, 
	$I$ is the identity operator and 
	$A$ is a linear operator that computes polynomial approximations of the correction function at nodes 
	where it is needed.
Since problem~\eqref{eq:minPblmPEC} is well-posed, 
	we assume that the term $L\,A\,M^{-1}\,\mathbold{b}_{\Gamma}$ can be bounded.
It is then sufficient to investigate the stability of 
\begin{equation*} \label{eq:withCFMPEC}
	\partial_t \mathbold{U}(t) + L\,(I + c_f\,A\,M^{-1}\, B)\,\mathbold{U}(t) = 0,
\end{equation*}
 to identify any time-step criteria of the corresponding CFM-FDTD scheme \cite{Gustafsson1995}. 
We remark that we recover the original FDTD scheme in the limit when $c_f \to 0$ 
	and therefore its properties.
We therefore assume that we should be close to the stability condition of the original FDTD scheme 
	for a sufficient small $c_f$.
This assumption is supported by numerical examples performed in Section~\ref{sec:numEx}. 

In order to choose an appropriate penalization coefficient $c_f$, 
	one should consider the following constraints. 
First, 
	the priority should be given to embedded boundary conditions 
	in problem \eqref{eq:minPblmPEC},
	that is $c_p>c_f$. 
Second, 
	the weight associated with fictitious interface conditions in the minimization 
	problem should diminish as the length of local patches $\ell_h$ goes to zero, 
	that is when the mesh grid size diminishes. 
This again enforces embedded boundary conditions and avoids any stability issues of  
	CFM-FDTD schemes as the time-step size is refined. 
However, 
	a too small value of $c_f$ could lead to poorly conditioned matrices 
	coming from the minimization problem~\eqref{eq:minPblmPEC}.
\subsection{Implementation of Fictitious Interface Conditions} \label{sec:fictInterfImplementation}
This short subsection focus on technical details concerning the implementation of fictitious interface conditions 
	\eqref{eq:fictInterfCdn}.
Since $\Omega^-$ is a PEC domain,
	it is common to consider $\mathbold{H}^-=0$ and $\mathbold{E}^-=0$ as explained in Section~\ref{sec:defPblm}.
The natural extension of these electromagnetic fields in the non-PEC domain,
	that is $\Omega^+$, 
	is then zero.
Hence, 
 	$\mathbold{D}_H = \mathbold{H}^+$ and $\mathbold{D}_E = \mathbold{E}^+$.
This allows us to enforce fictitious interface conditions in $\Omega^+$ 
	using finite difference approximations within it, 
	namely $\mathbold{H}^* \approx \mathbold{H}^+$ and 
	$\mathbold{E}^*\approx \mathbold{E}^+$.
	
To ease the implementation of fictitious interface conditions \eqref{eq:fictInterfCdn}, 
	we choose fictitious interfaces that are aligned with the mesh grid.
In other words, 
	fictitious interfaces $\Gamma_{k,i}$ are chosen in such a way that their normal $\hat{\mathbold{n}}_{i,k}$ 
	is an element of the standard basis in $\mathbb{R}^3$.
This facilitates the construction of space-time interpolating polynomials that use FD approximations. 
Let us consider the transverse magnetic (TM$_z$) mode (see Section~\ref{sec:TMz}),
	which is a 2-D simplification of Maxwell's equations,
	as an example.
In this case,
	the normal of a fictitious interface is either $\mathbold{n}_1 = (1,0)$ or $\mathbold{n}_2 = (0,1)$. 
Hence, 
	$\mathbold{n}_1\cdot \mathbold{H}^* = H_x^*$,
	$\mathbold{n}_2\cdot\mathbold{H}^* = H_y^*$, 
	$\mathbold{n}_1\times\mathbold{H}^* = H_y^*$, 
	$\mathbold{n}_2\times\mathbold{H}^* = -H_x^*$,
	$\mathbold{n}_1\times\mathbold{E}^* = (0,-E_z^*)$ and 
	$\mathbold{n}_2\times\mathbold{E}^* = (E_z^*,0)$.
Fig.~\ref{fig:fictInterfExample} illustrates fictitious interfaces that can be generated for a given local patch and 
	a staggered grid that is described in Section~\ref{sec:2D}.
\begin{figure}
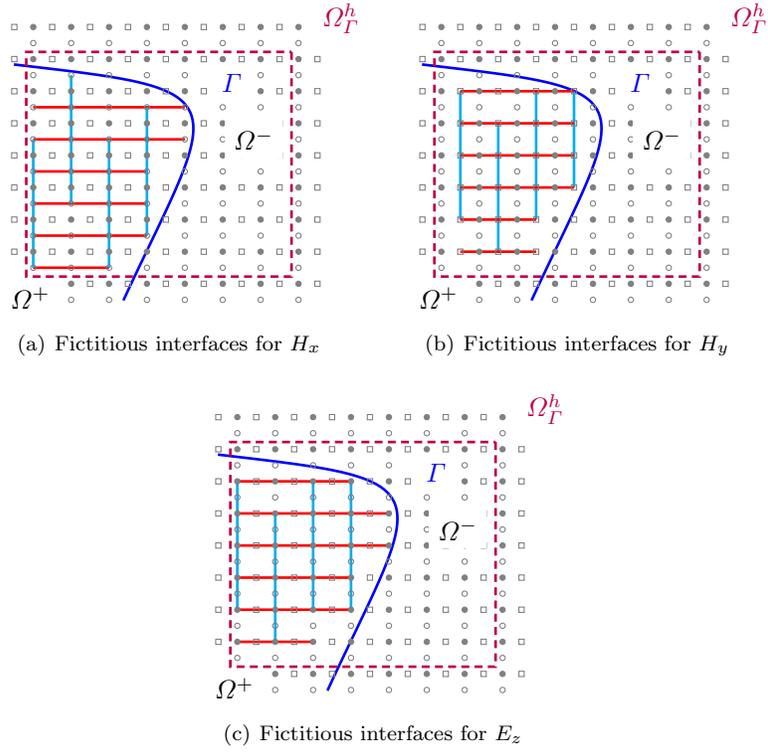

 \centering
  \subfigure[Fictitious interfaces for $H_x$]{\setlength\figureheight{0.3\linewidth} 
		\setlength\figurewidth{0.35\linewidth} 
		\tikzset{external/export next=false}
		\input{fictInterfForHx.tikz}}
  \subfigure[Fictitious interfaces for $H_y$]{\setlength\figureheight{0.3\linewidth} 
		\setlength\figurewidth{0.35\linewidth} 
		\tikzset{external/export next=false}
		\input{fictInterfForHy.tikz}}
  \subfigure[Fictitious interfaces for $E_z$]{\setlength\figureheight{0.3\linewidth} 
		\setlength\figurewidth{0.35\linewidth} 
		\tikzset{external/export next=false}
		\input{fictInterfForE.tikz}}	
  \caption{An example of a local patch $\Omega_{\Gamma}^h$ with fictitious interfaces. 
  		The $x$-component and the $y$-component of the magnetic field 
		are respectively represented by 
		{\color{gray}$\circ$} and {\color{gray}\tiny$\square$} while the $z$-component of the electric field is 
		represented by {\color{gray}$\bullet$}.
		Fictitious interfaces associated with $\mathbold{n}_1 = (1,0)$ and $\mathbold{n}_2=(0,1)$ are respectively
		represented by \textcolor{cyan}{\rule[0.35ex]{0.25cm}{1pt}} and \textcolor{red}{\rule[0.35ex]{0.25cm}{1pt}}.}
		 \label{fig:fictInterfExample}
\end{figure}

The functional \eqref{eq:functionalPEC} involves time integrals of 
	finite difference approximations $\mathbold{H}^*$ and $\mathbold{E}^*$ in 
	the vicinity of the boundary $\Gamma$.
Since the time interval associated with local patches is given by $I_{\Gamma}^h = [t_n-\Delta t_{\Gamma},  t_n]$,
	we can use previous computed finite difference solutions to construct 
	the space-time interpolant needed for fictitious interface conditions. 
However, 
	this makes difficult the initialization of a CFM-FDTD scheme that uses fictitious interface 
	conditions.
In Section~\ref{sec:Yee} and Section~\ref{sec:fourthOrderScheme}, 
	we propose an initialization strategy for the Yee scheme and a fourth-order FDTD scheme.
		
\subsection{Computation of local patches} \label{sec:compLocalPatch}
The computation of an appropriate local patch $\Omega_{\Gamma}^h$ is essential for the CFM.
The well-posedness of problem \eqref{eq:minPblm} and \eqref{eq:minPblmPEC} highly depends 
	on the representation of the embedded boundary 
	within the local patch $\Omega_{\Gamma}^{h}$. 
Hence, 
	an appropriate local patch is of foremost importance to obtain an accurate approximation
	of a correction function.
In previous CFM-FDTD schemes, 
	a \enquote{Node Centered} approach is used to compute local patches
		\cite{Marques2011,LawMarquesNave2020}. 
This approach consists to define a local patch and solve a minimization problem 
	for each node to be corrected.
Even though it is more expensive than other constructions 
	of local patches,   
	\enquote{Node Centered} approaches have the benefit to guarantee
	 the uniqueness of the correction function at a given node.
Hence, 
	a common discrete measure of the divergence for staggered FDTD schemes is conserved 
	for some nodes close to the embedded boundary \cite{LawMarquesNave2020}.	
	
In this work,
	as it is done for other immersed boundary methods \cite{Kallemov2016,Stein2016}, 
	we directly discretize the embedded boundary.  
Let us now summarize this approach.
For simplicity, 
	let us consider $\Omega \subset \mathbb{R}^2$.
Assume a given embedded boundary $\Gamma$ that can be parametrized with respect to the 
	parameter $s\in[s_a,s_b]$.  
The number of local patches is given by 
\begin{equation} \label{eq:nbLocPatch}
	N_s \approx \frac{L_\Gamma}{\alpha\,h} + 1,
\end{equation}
	where $L_\Gamma$ is the estimated arc length of $\Gamma$,
	$h$ is mesh grid size and $\alpha$ is positive constant.
In this work, 
	we use $\alpha = 2$.
Hence,
	centre points of local patches are given by $\mathbold{x}_{c,i} = (x(s_i),y(s_i))$, 
	where $s_i = s_a + i\,\Delta s$
 	for $i = 0,\dots,N_s-1$ and $\Delta s = \tfrac{s_b-s_a}{N_s-1}$.
For a given node to be corrected at $\mathbold{x}_D$,
	we find the closest $\mathbold{x}_{c,i}$ and associate the corresponding local patch to
	$\mathbold{x}_D$.
We therefore guarantee the uniqueness of the correction function to each node to be corrected
	while reducing the computational cost, 
	particularly for large stencils,
	when compared with \enquote{Node Centered} approaches.
The local patches are square and aligned with the computational grid. 
The length in space of local patches is $\ell_h = \beta\,h$, 
		where $h$ is the mesh grid size and $\beta$ is a positive constant. 
It is worth mentioning that the parameter $\beta$ is chosen in such a way 
	that enough fictitious interfaces 
	can be generated within $\Omega_{\Gamma}^h$ and that all nodes to be corrected are 
 	associated with a local patch.
	
As for the time interval $I_{\Gamma}^h = [t_n-\Delta t_{\Gamma},  t_n]$, 
	we choose $\Delta t_{\Gamma}$ in such a way that $I_{\Gamma}^h$ 
	include the number of time steps needed to construct space-time interpolants 
	associated with fictitious interface conditions. 
\section{Application of the CFM to the Yee Scheme} \label{sec:Yee}
In this section,
	we apply the CFM to the well-known Yee scheme \cite{Yee1966}, 
	which is a popular FDTD scheme in computational electromagnetics,
	with a particular attention on its initialization.
Let us recall that the Yee scheme uses a staggered grid in both space and time. 
We then need to adapt the functional \eqref{eq:functionalPEC}, 
	and more precisely 
	the interval of integration in time of fictitious interface conditions, 
	in order to consider a staggered grid in time.
Finally, 
	we conclude with pros and cons of such an approach.
In the following,
	we assume that the parameter $\beta$ has been chosen in such a way that enough fictitious interfaces have 
	been generated within $\Omega_{\Gamma}^h$ (see subsection~\ref{sec:compLocalPatch}) and 
	we therefore focus on the time component.

Let us first define a staggered grid in time. 
We consider a time interval $I = [0,T]$ subdivided into $N_t$ subintervals of length $\Delta t$. 
We then have $t_n := n\,\Delta t$ for $n = 0,\dots,N_t$ and $t_{n+1/2} := (n+1/2)\,\Delta t$ for $n = -1, \dots, N_t-1$.
The magnetic and electric field are respectively defined at $t_{n+1/2}$ and $t_n$.	

According to the Yee scheme, 
	we first compute $\mathbold{H}^{1/2}$ using 
	initial conditions $\mathbold{H}^{-1/2}$ and $\mathbold{E}^0$, 
		as illustrated in Fig.~\ref{fig:InitYeeScheme}. 
\begin{figure}[htbp]
 	\centering
 	\tdplotsetmaincoords{75}{105}
	\tikzset{external/export next=false}
  \begin{tikzpicture}[scale=0.75]
  	\draw[-latex,thick,black] (1,-5.8)--(8,-5.8);
	\draw[thick,black] (1,-5.7)--(1,-5.9);
	\draw[thick,black] (3,-5.7)--(3,-5.9);
	\draw[thick,black] (5,-5.7)--(5,-5.9);
	\draw[thick,black] (7,-5.7)--(7,-5.9);
  	\draw (1.45,-6.4) node {$t_{-1/2}$};
  	\draw (5.35,-6.4) node {$t_{1/2}$};
	\draw (3.1,-6.39) node {$t_0$};
	\draw (7.1,-6.39) node {$t_1$};
	\draw[thick,black] (1,0)--(3,0);
	\draw[dashed,thick,black] (3,0)--(5,0);
	\draw[-latex,thick,black,fill=black] (1,0) circle [radius=0.1];
	\draw[-latex,thick,black,fill=white] (3,0) circle [radius=0.1];
	\draw[-latex,thick,black,fill=black] (5,0) circle [radius=0.1];
  	\draw (1,1) node {\small$\mathbold{H}$, $\partial_t\mathbold{E}$};
  	\draw (3,1) node {\small$\mathbold{E}$, $\partial_t\mathbold{H}$};
	\draw[-latex,thick,black] (1,0.8)--(1,0.2);
	\draw[-latex,thick,black] (3,0.8)--(3,0.2);
	\draw[-latex,thick,black] (3,-1.45)--(3,-0.85);
  	\draw (3,-1.65) node {\small$\mathbold{D}_E$};
	\pattern[pattern=north east lines] (1,-0.5)--(3,-0.5)--(3,-0.75)--(1,-0.75)--cycle;
	\draw [black, thick] (1,-0.5)--(3,-0.5)--(3,-0.75)--(1,-0.75)--cycle;
	\draw[thick,black] (1,-3.5)--(5,-3.5);
	\draw[dashed,thick,black] (5,-3.5)--(7,-3.5);
	\draw[dashed,thick,black] (3,-3.5)--(5,-3.5);
	\draw[-latex,thick,black,fill=black] (1,-3.5) circle [radius=0.1];
	\draw[-latex,thick,black,fill=white] (3,-3.5) circle [radius=0.1];
	\draw[-latex,thick,black,fill=black] (5,-3.5) circle [radius=0.1];
	\draw[-latex,thick,black,fill=white] (7,-3.5) circle [radius=0.1];
  	\draw (1,-2.5) node {\small$\mathbold{H}$, $\partial_t\mathbold{E}$};
  	\draw (3,-2.45) node {\small$\mathbold{E}$};
  	\draw (5,-2.45) node {\small$\mathbold{H}$};
	\draw[-latex,thick,black] (1,-2.7)--(1,-3.3);
	\draw[-latex,thick,black] (3,-2.7)--(3,-3.3);
	\draw[-latex,thick,black] (5,-2.7)--(5,-3.3);
	\draw[-latex,thick,black] (5,-4.95)--(5,-4.35);
  	\draw (5,-5.15) node {\small$\mathbold{D}_H$};
	\pattern[pattern=north east lines] (1,-4)--(5,-4)--(5,-4.25)--(1,-4.25)--cycle;
	\draw [black, thick] (1,-4)--(5,-4)--(5,-4.25)--(1,-4.25)--cycle;
  \end{tikzpicture}
  \caption{Initialization strategy for the proposed CFM-Yee scheme. The black and white circle marker represent respectively the magnetic and electric field. The dashed box illustrates the time interval $I_{\Gamma}^h$ of local patches.}
   \label{fig:InitYeeScheme}
\end{figure}
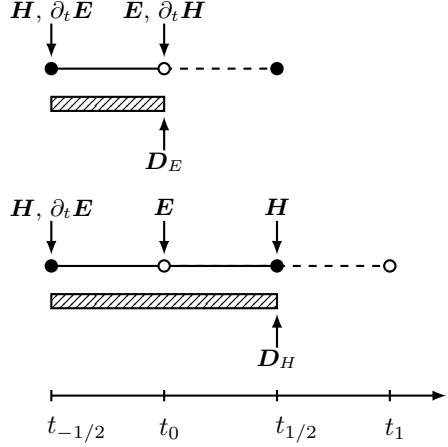
In this case, 
	the CFM-Yee scheme needs to provide corrections for the electric field at $t_0$, 
	that is $\mathbold{D}_E^0$.
The time interval of local patches is then $I^h_{\Gamma} = [t_{-1/2},t_0]$.
At first sight, 
	we do not have enough information in time within local patches to build accurate enough space-time interpolants 
		for fictitious interface conditions.
However, 
	by Faraday's law~\eqref{eq:Faraday} and Amp\`ere-Maxwell's law~\eqref{eq:AmpereMaxwell}, 
	we have $\partial_t \mathbold{H}^+ = -\mu^{-1}\,\nabla\times\mathbold{E}^+$ and 
	$\partial_t \mathbold{E}^+ = \epsilon^{-1}\,\nabla \times \mathbold{H}^+$ in $\Omega^+$. 
We can then compute the first-order time derivative of $\mathbold{H}$ at $t_0$ and $\mathbold{E}$ at $t_{-1/2}$ 
	using the curl of $\mathbold{E}^0$ and $\mathbold{H}^{-1/2}$. 
It is worth mentioning that one could estimate the curl operator using appropriate finite difference approximations.
First-degree polynomials in time can be constructed using $\mathbold{H}^{-1/2}$ and $\partial_t \mathbold{H}^0$, 
	and $\mathbold{E}^0$ and $\partial_t \mathbold{E}^{-1/2}$.
Hence, 
 	the interval of integration in time associated with all fictitious interface conditions in functional \eqref{eq:functionalPEC} 
	is also $I_{\Gamma}^h$.

For the computation of $\mathbold{E}^1$, 
	one needs correction functions for the magnetic field at $t_{1/2}$, 
	that is $\mathbold{D}_H^{1/2}$, 
	as illustrated in Fig.~\ref{fig:InitYeeScheme}.
The time interval of local patches is then $I^h_{\Gamma} = [t_{-1/2},t_{1/2}]$.
In this situation, 
	we construct a first-degree polynomial in time for the electric field using again $\mathbold{E}^0$ 
	and $\partial_t \mathbold{E}^{-1/2}$.
As for the magnetic field, 
	we use $\mathbold{H}^{-1/2}$ and $\mathbold{H}^{1/2}$ to compute a first-degree polynomial in time. 
The interval of integration in time associated with fictitious interface conditions involving the magnetic field
	is then $[t_{-1/2}, t_{1/2}]$ while 
	the one associated with the electric field is $[t_{-1/2},t_0]$.

Once the initialization of the proposed CFM-Yee scheme is done, 
	we only have two cases to consider, 
	as illustrated in Fig.~\ref{fig:strategyYeeScheme}.
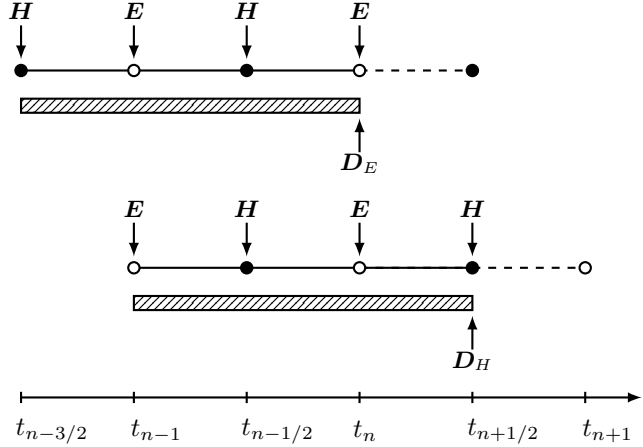
\begin{figure}[htbp]
 	\centering
 	\tdplotsetmaincoords{75}{105}
	\tikzset{external/export next=false}
  \begin{tikzpicture}[scale=0.75]
  	\draw[-latex,thick,black] (-3,-5.8)--(8,-5.8);
	\draw[thick,black] (-3,-5.7)--(-3,-5.9);
	\draw[thick,black] (-1,-5.7)--(-1,-5.9);
	\draw[thick,black] (1,-5.7)--(1,-5.9);
	\draw[thick,black] (3,-5.7)--(3,-5.9);
	\draw[thick,black] (5,-5.7)--(5,-5.9);
	\draw[thick,black] (7,-5.7)--(7,-5.9);
  	\draw (-2.45,-6.4) node {$t_{n-3/2}$};
  	\draw (-0.6,-6.39) node {$t_{n-1}$};
  	\draw (1.55,-6.4) node {$t_{n-1/2}$};
  	\draw (5.55,-6.4) node {$t_{n+1/2}$};
	\draw (3.1,-6.39) node {$t_n$};
	\draw (7.4,-6.39) node {$t_{n+1}$};
	\draw[thick,black] (-3,0)--(3,0);
	\draw[dashed,thick,black] (3,0)--(5,0);
	\draw[-latex,thick,black,fill=black] (-3,0) circle [radius=0.1];
	\draw[-latex,thick,black,fill=white] (-1,0) circle [radius=0.1];
	\draw[-latex,thick,black,fill=black] (1,0) circle [radius=0.1];
	\draw[-latex,thick,black,fill=white] (3,0) circle [radius=0.1];
	\draw[-latex,thick,black,fill=black] (5,0) circle [radius=0.1];
  	\draw (-3,1.05) node {\small$\mathbold{H}$};
  	\draw (-1,1.05) node {\small$\mathbold{E}$};
  	\draw (1,1.05) node {\small$\mathbold{H}$};
  	\draw (3,1.05) node {\small$\mathbold{E}$};
	\draw[-latex,thick,black] (-3,0.8)--(-3,0.2);
	\draw[-latex,thick,black] (-1,0.8)--(-1,0.2);
	\draw[-latex,thick,black] (1,0.8)--(1,0.2);
	\draw[-latex,thick,black] (3,0.8)--(3,0.2);
	\draw[-latex,thick,black] (3,-1.45)--(3,-0.85);
  	\draw (3,-1.65) node {\small$\mathbold{D}_E$};
	\pattern[pattern=north east lines] (-3,-0.5)--(3,-0.5)--(3,-0.75)--(-3,-0.75)--cycle;
	\draw [black, thick] (-3,-0.5)--(3,-0.5)--(3,-0.75)--(-3,-0.75)--cycle;
	\draw[thick,black] (-1,-3.5)--(5,-3.5);
	\draw[dashed,thick,black] (5,-3.5)--(7,-3.5);
	\draw[dashed,thick,black] (3,-3.5)--(5,-3.5);
	\draw[-latex,thick,black,fill=white] (-1,-3.5) circle [radius=0.1];
	\draw[-latex,thick,black,fill=black] (1,-3.5) circle [radius=0.1];
	\draw[-latex,thick,black,fill=white] (3,-3.5) circle [radius=0.1];
	\draw[-latex,thick,black,fill=black] (5,-3.5) circle [radius=0.1];
	\draw[-latex,thick,black,fill=white] (7,-3.5) circle [radius=0.1];
  	\draw (-1,-2.45) node {\small$\mathbold{E}$};
  	\draw (1,-2.45) node {\small$\mathbold{H}$};
  	\draw (3,-2.45) node {\small$\mathbold{E}$};
  	\draw (5,-2.45) node {\small$\mathbold{H}$};
	\draw[-latex,thick,black] (-1,-2.7)--(-1,-3.3);
	\draw[-latex,thick,black] (1,-2.7)--(1,-3.3);
	\draw[-latex,thick,black] (3,-2.7)--(3,-3.3);
	\draw[-latex,thick,black] (5,-2.7)--(5,-3.3);
	\draw[-latex,thick,black] (5,-4.95)--(5,-4.35);
  	\draw (5,-5.15) node {\small$\mathbold{D}_H$};
	\pattern[pattern=north east lines] (-1,-4)--(5,-4)--(5,-4.25)--(-1,-4.25)--cycle;
	\draw [black, thick] (-1,-4)--(5,-4)--(5,-4.25)--(-1,-4.25)--cycle;
  \end{tikzpicture}
  \caption{Strategy for a CFM-Yee scheme. The black and white circle marker represent respectively the magnetic and electric field. The dashed box illustrates the time interval $I_{\Gamma}^h$ of local patches.}
   \label{fig:strategyYeeScheme}
\end{figure}
The first case involves the computation of $\mathbold{H}^{n+1/2}$ and therefore $\mathbold{D}_E^{n}$. 
The time interval of local patches is $I_{\Gamma}^h = [t_{n-3/2}, t_n]$.
Approximations of the magnetic field at $t_{n-3/2}$ and $t_{n-1/2}$ 
	are used to construct a first-degree polynomial interpolant in time.
This leads to an interval of integration in time associated with fictitious interface conditions involving the magnetic field of 
	$[t_{n-3/2}, t_{n-1/2}]$.
As for fictitious interface conditions of the electric field, 
	the interval of integration in time
	is $[t_{n-1}, t_{n}]$,
	and $\mathbold{E}^{n-1}$ and $\mathbold{E}^n$ are used to construct a first-degree polynomial in time.
The second case implies the computation of $\mathbold{E}^{n+1}$.
We then need to compute $\mathbold{D}_H^{n+1/2}$ and therefore $I_{\Gamma}^h = [t_{n-1}, t_{n+1/2}]$.
First-degree polynomial in time is constructed using $\mathbold{H}^{n-1/2}$ and $\mathbold{H}^{n+1/2}$,
	and $\mathbold{E}^{n-1}$ and $\mathbold{E}^n$.
This leads to intervals of integration in time of fictitious interfaces given by $[t_{n-1/2}, t_{n+1/2}]$ for the magnetic 
	field and $[t_{n-1}, t_n]$ for the electric field. 
	
Another avenue to initialize the proposed CFM-Yee scheme, 
	although it is very specific to some applications, 
	is to consider that $\mathbold{H}^+$ and $\mathbold{E}^+$ in the vicinity 
	of the embedded boundary remain unchanged for $t\leq t_0$.
Hence, 
	the numerical strategy described previously for Fig.~\ref{fig:strategyYeeScheme} can be directly used.
As an example, 
	this approach could be useful for scattering problems.
	
Finally, 
	the main disadvantage of this approach is the computation cost associated with 
	minimization problems of the CFM at each update of electromagnetic fields.
In fact, 
	the CFM consumes most of the computational time when compared with the finite-difference part.
However,    
	a parallel implementation of the computation of approximations of the correction function can be performed since 
	minimization problems associated with local patches are independent at a given time step. 
We refer to \cite{Abraham2017} for more details about the benefits of a parallel implementation of the CFM.
It is also worth mentioning that we do not have to keep whole previous solutions but only approximations 
	associated with fictitious interfaces.
Despite this drawback, 
	the proposed CFM-Yee scheme could achieve second-order convergence 
	for appropriate approximations of the correction function (see Proposition~\ref{lem:errorAnalysis})
	while treating various complex geometries of the embedded boundary without significantly 
	increasing the complexity of the numerical approach.
Moreover, 
	it can also be implemented as a black-box for existing softwares that use the Yee scheme.
	
\section{Application of the CFM on a Fourth-Order Staggered FDTD Scheme} \label{sec:fourthOrderScheme}

In this section, 
	we introduce a CFM-FDTD scheme based on a fourth-order staggered FDTD scheme. 
The staggered space and time grid are defined as in the Yee scheme. 
Spatial derivatives are estimated using the fourth-order centered approximation.
As for time derivatives, 
	many avenues can be chosen, 
	such as staggered Adams-Bashforth or staggered backward differentiation methods \cite{Ghrist2000}.
In this work, 
	we choose a fourth-order staggered free-parameter multistep method introduced in \cite{Ghrist2000}, 
	which has a maximum imaginary stability boundary close to the leapfrog method used in the Yee scheme.

In the following, 
	we first describe a fourth-order staggered free-parameter multistep method that is used to discretize time derivatives.
We assume that previous solutions needed for the initialization of the multistep method are given.
Afterward, 
	we introduce the associated CFM-FDTD scheme with a particular attention on the time component.
As in Section~\ref{sec:Yee},
 	we assume that the parameter $\beta$ has been chosen in such a way that 
	enough fictitious interfaces have been generated within $\Omega_{\Gamma}^h$.
	
Let us consider $\partial_t \mathbold{H} = \mathbold{f}_H(\mathbold{E})$ and 
	$\partial_t \mathbold{E} = \mathbold{f}_E(\mathbold{H})$.
The considered fourth-order free-parameter method is given by
\begin{equation} \label{eq:freeParamScheme}
\begin{aligned}
\mathbold{H}^{n+1/2} =&\,\, -\alpha_3\,\mathbold{H}^{n-1/2} - \alpha_2\,\mathbold{H}^{n-3/2} - \alpha_1\,\mathbold{H}^{n-5/2} -
				\alpha_0\,\mathbold{H}^{n-7/2} + \\
				&\,\, \Delta t \, \big(\beta_3\,\mathbold{f}_H(\mathbold{E}^n) + \beta_2\,\mathbold{f}_H(\mathbold{E}^{n-1}) 
				+ \beta_1\,\mathbold{f}_H(\mathbold{E}^{n-2})\big)\\
\mathbold{E}^{n+1}   =&\,\,  -\alpha_3\,\mathbold{E}^{n} - \alpha_2\,\mathbold{E}^{n-1} - \alpha_1\,\mathbold{E}^{n-2} -
				\alpha_0\,\mathbold{E}^{n-3} + \\
				&\,\, \Delta t \, \big(\beta_3\,\mathbold{f}_E(\mathbold{H}^{n+1/2}) + \beta_2\,\mathbold{f}_E(\mathbold{H}^{n-1/2}) 
				+ \beta_1\,\mathbold{f}_E(\mathbold{H}^{n-3/2})\big)
\end{aligned}
\end{equation}
where $\beta_1=t$, 
       $\beta_2 = s$,
       $\beta_3 = \frac{1}{22}\,s + \frac{12}{11}$,  
\begin{equation*}
\begin{aligned}
\alpha_0 =&\,\, -\frac{1}{22} - \frac{1}{528}\,s+\frac{1}{24}\,t, \\
\alpha_1 =&\,\, \frac{5}{22} + \frac{9}{176}\,s-\frac{9}{8}\,t, \\
\alpha_2 =&\,\, -\frac{9}{22} - \frac{201}{176}\,s+\frac{9}{8}\,t, \\
\alpha_3 =&\,\, -\frac{17}{22} + \frac{577}{528}\,s-\frac{1}{24}\,t,
\end{aligned}
\end{equation*}
	with $s=-1$ and $t=1.045$.

Let us now introduce the corresponding CFM-FDTD scheme. 
The time scheme \eqref{eq:freeParamScheme} needs to compute
	$\mathbold{f}_H(\mathbold{E})$ at $t_n$, $t_{n-1}$ and $t_{n-2}$, 
	and $\mathbold{f}_E(\mathbold{H})$ at $t_{n+1/2}$, $t_{n-1/2}$ and $t_{n-3/2}$.
Hence, 
	we need to keep previous corresponding correction functions.
For the computation of $\mathbold{H}^{n+1/2}$, 
	we set $I_{\Gamma}^h = [t_{n-7/2}, t_{n}]$ and compute $\mathbold{D}_E^{n}$.
It is worth mentioning that, 
	at the first update of the magnetic field to estimate $\mathbold{H}^{1/2}$,
	we compute the correction function $\mathbold{D}_E$ at $t_0$, 
	$t_{-1}$ and $t_{-2}$, 
	and $\mathbold{D}_H$ at $t_{-1/2}$ and $t_{-3/2}$. 
For the update of the electric field, 
	that is $\mathbold{E}^{n+1}$, 
	we set $I_{\Gamma}^h = [t_{n-3}, t_{n+1/2}]$ and compute $\mathbold{D}_H^{n+1/2}$.
As for the Yee scheme, 
	we need to adapt the interval of integration in time of fictitious interface conditions using
	a similar procedure as in Section~\ref{sec:Yee} for both cases.

\section{Numerical Examples} \label{sec:numEx}
In the following, 
	we name the CFM-FDTD scheme based on the Yee scheme as CFM-Yee scheme
	while the one based on a fourth-order staggered FDTD scheme 
	is named CFM-$4^{th}$ scheme.
The error of $\mathbold{U}=[H_x,H_y,E_z]^T$ at $t_n$ is computed using approximations and 
	analytic solutions of the magnetic field and the electric field at respectively $t_n-\tfrac{\Delta t}{2}$ and 
	$t_n$ because of the staggered grid in time.

\subsection{Transverse Magnetic Mode} \label{sec:TMz}
Let us consider the transverse magnetic (TM$_{\text{z}}$) mode. 
The unknowns are $H_x(x,y,t)$, 
	$H_y(x,y,t)$ and $E_z(x,y,t)$.
For a domain $\Omega \subset \mathbb{R}^2$ and constant physical parameters, 
	Maxwell's equations are then simplified to 
\begin{equation} \label{eq:TMzSyst}
	\begin{aligned}
	\mu\,\partial_t H_x + \partial_y E_z =&\,\, 0 \quad \text{in } \Omega \times I,\\
	\mu\,\partial_t H_y - \partial_x E_z  =&\,\, 0 \quad \text{in } \Omega \times I,\\
	\epsilon\,\partial_t E_z - \partial_x H_y + \partial_y H_x =&\,\, 0 \quad \text{in } \Omega \times I,\\
	\partial_x H_x + \partial_y H_y =&\,\, 0 \quad \text{in } \Omega \times I,\\
	\end{aligned}
\end{equation}
	with the associated boundary, embedded 
	boundary and initial conditions. 

\subsection{A $2$-D Staggered Space Discretization}
\label{sec:2D}
In this subsection, 
	we present a staggered grid in space.
Let us consider a rectangular domain $\Omega = [x_{\ell},x_r]\times[y_b, y_t] \subset \mathbb{R}^2$. 
The domain is divided in $N=N_x\,N_y$ square cells, 
	noted by $\Omega_{ij} = [x_{i-1/2},x_{i+1/2}]\times[y_{j-1/2},y_{j+1/2}]$ and centered at 
	$$(x_i,y_j) = (x_{\ell} + (i - \tfrac{1}{2})\,\Delta x , y_b + (j-\tfrac{1}{2})\,\Delta y)$$ for $i=1,\dots,N_x$ and $j=1,\dots,N_y$
	with $\Delta x := (x_r-x_{\ell})/N_x$ and $\Delta y := (y_t-y_b)/N_y$.
The $z$-component of the electric field is approximated at the center of the cell.
The $x$-component and $y$-component of the magnetic field are respectively approximated at 
	$$(x_i,y_{j+1/2}) = (x_{\ell}+(i-\tfrac{1}{2})\,\Delta x, y_b +j\,\Delta y)$$ 
	for $i = 1, \dots, N_x$ and for $j=0,\dots,N_y$, and 
	$$(x_{i+1/2},y_j) = (x_{\ell}+i\,\Delta x, y_b + (j-\tfrac{1}{2})\,\Delta y)$$
	for $i= 0, \dots, N_x$ and for $j=1,\dots,N_y$.
We use either the second or fourth order centered finite difference scheme in space.
As an example, 
	the $x$-derivative of $H_y$ is of the form 
\begin{equation*}
	\partial_x H_y(x_i,y_j,t_{n+1/2}) \approx \frac{H_{y,i+1/2,j}^{n+1/2}-H_{y,i-1/2,j}^{n+1/2}}{\Delta x}
\end{equation*}
	and
\begin{equation*} 
	\partial_x H_y(x_i,y_j,t_{n+1/2}) \approx \frac{H_{y,i-3/2,j}^{n+1/2}-27\,H_{y,i-1/2,j}^{n+1/2}+27\,H_{y,i+1/2,j}^{n+1/2}-H_{y,i-3/2,j}^{n+1/2}}{24\,\Delta x}
\end{equation*}
	for respectively the second and fourth order centered finite-difference.
Finally,
	as it is commonly used, 
	we impose $E_z=0$ and $H_x=H_y=0$ in PEC subdomains.

\subsection{Problems with an Analytic Solution} \label{sec:pblmsAnalyticSol}
In this subsection,
	we perform numerical examples with analytic solutions 
	to assess the impact of the penalization coefficient $c_f$ and 
	to verify the proposed numerical approach. 
The domain $\Omega$ is divided in two subdomains $\Omega^+$ and $\Omega^-$.
Periodic boundary conditions are used on all $\partial \Omega$.
We set $c_p = 1$.
The mesh grid size is such that $h=\Delta x = \Delta y$.
We use second and third degree polynomial approximations of the correction function for respectively 
	the CFM-Yee scheme and the CFM-$4^{th}$ scheme. 
Hence, 
 	this should lead to a second and third order convergence in $L^2$-norm 
	(see Proposition~\ref{lem:errorAnalysis}). 
We set $\ell_h = 7\,h$,
	and we construct $\mathbold{E}^*$ and $\mathbold{H}^*$ in $\Omega^+$ 
	using at least a 
	second degree interpolating polynomial in space for both schemes.

\subsubsection{Circular Cavity Problem} \label{sec:circularCavity}
Let us consider a holed PEC material.   
The domain is $\Omega = [-1.25,1.25]\times[-1.25,1.25]$. 
Since $\Omega^-$ is a PEC subdomain, 
	the embedded boundary $\Gamma$ then encloses subdomain $\Omega^+$. 
The embedded boundary is a circle centered at $(0,0)$ with unit radius. 
The physical parameters are $\epsilon = 1$ and $\mu = 1$. 
The time-step size is $\Delta t = \tfrac{h}{2}$ for both CFM-FDTD schemes.
In subdomain $\Omega^+$, 
	the solution in cylindrical coordinates is given by 
\begin{equation*}
	\begin{aligned}
		H_\rho^+(\rho,\phi,t) =&\,\, \frac{i}{\alpha_{i,j}\,\rho}\,J_i(\alpha_{i,j}\,\rho)\,\sin(i\,\phi)\,\sin(\alpha_{i,j}\,t), \\
		H_\phi^+(\rho,\phi,t) =&\,\,\frac{1}{2}\,\big(J_{i-1}(\alpha_{i,j}\,\rho)-J_{i+1}(\alpha_{i,j}\,\rho)\big)\,\cos(i\,\phi)\,\sin(\alpha_{i,j}\,t), \\
		E_z^+(\rho,\phi,t) =&\,\, J_i(\alpha_{i,j}\,\rho)\,\cos(i\,\phi)\,\cos(\alpha_{i,j}\,t),
	\end{aligned}
\end{equation*}
	where $\alpha_{i,j}$ is the $j$-th positive real root of the $i$-order Bessel function of first kind $J_i$. 
In this numerical example, 
	we choose $i=6$ and $j=2$.
	
Let us assess the impact of the penalization coefficient $c_f$ on the proposed CFM-FDTD schemes.
Fig.~\ref{fig:convPlotCircCavityInterface} illustrates convergence plots of $\mathbold{U}=[H_x,H_y,E_z]^T$ 
	in $L^2$-norm for both CFM-FDTD schemes
	using different values of $c_f$, 
	that is $\Delta t$, 
	$\tfrac{\Delta t}{2}$ and $\tfrac{\Delta t}{4}$. 
The time interval is $I = [0,0.5]$.
We observe a clear second-order convergence for the CFM-Yee scheme.
For the CFM-$4^{th}$ scheme, 
	a fourth-order convergence is observed, 
	which is better than expected. 
The obtained convergence orders are in agreement with the theory for all values of $c_f$.
However, 
	one can notice that the error slightly increases as the value of $c_f$ diminishes. 
Let us now perform long time simulations. 
Fig.~\ref{fig:longSimulationCavity} illustrates the evolution of the error of electromagnetic fields
	in $L^2$-norm 
	as a function of the number of periods for different values of $c_f$.
The time interval is $I=[0,10]$ and the mesh grid size is $h=\tfrac{1}{160}$.
Numerical results suggest that the CFM-Yee scheme is stable for the considered values of 
	the penalization coefficient $c_f$. 
However,
	the CFM-$4^{th}$ scheme seems more sensitive than the CFM-Yee scheme to the value of $c_f$.
Stability issues appear after a dozen of periods for the largest considered value of the penalization coefficient, 
	that is $c_f = \Delta t$. 
The penalization coefficient $c_f$ must be therefore chosen small enough to avoid stability issues 
	as the mesh grid size diminishes.
Based on these numerical results,  
	we choose $c_f = \Delta t$ and $c_f = \tfrac{\Delta t}{4}$ for respectively the CFM-Yee scheme and 
	the CFM-$4^{th}$ scheme to avoid any stability issues in all other numerical examples.
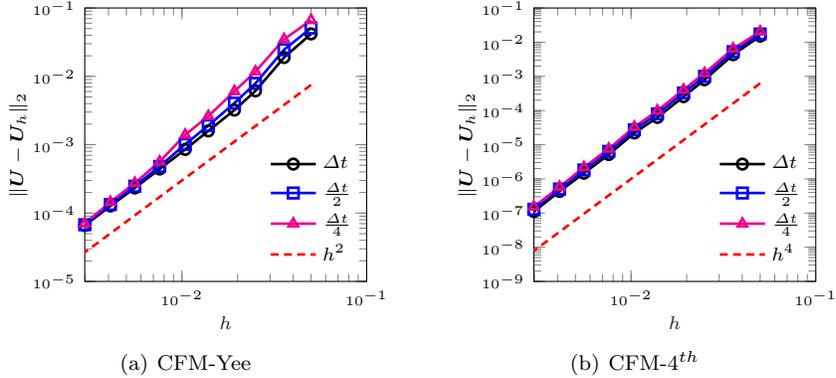
\begin{figure}
\begin{adjustbox}{max width=2.5\textwidth,center}
 \centering
  \subfigure[CFM-Yee]{
		\setlength\figureheight{0.3\linewidth} 
		\setlength\figurewidth{0.325\linewidth} 
		\tikzset{external/export next=false}
%
%
\begin{tikzpicture}

\begin{axis}[%
width=0.951\figurewidth,
height=\figureheight,
at={(0\figurewidth,0\figureheight)},
scale only axis,
xmode=log,
xmin=0.00297619047619048,
xmax=0.1,
xminorticks=true,
xlabel style={font=\color{white!15!black}},
xlabel={\scriptsize$h$},
ymode=log,
ymin=1e-05,
ymax=0.1,
yminorticks=true,
ytick = {1e-8,1e-7,1e-6,1e-5, 1e-4 ,1e-3, 1e-2, 1e-1,1},
ylabel={\scriptsize$\|\mathbold{U}-\mathbold{U}_h\|_{2}$},
axis background/.style={fill=white},
legend style={at={(0.625,0.5)},anchor=north west,legend cell align=left,align=left,draw=white!15!black,draw=none,fill=none},
legend style={font=\scriptsize},
ylabel style={yshift=-5pt},xlabel style={yshift=2.5pt},tick label style={font=\tiny} 
]
\addplot [color=black,line width=1pt,solid,mark=o,mark options={solid}]
  table[row sep=crcr]{%
0.05	0.0417830886647181\\
0.0357142857142857	0.018878788775868\\
0.025	0.00612843687760537\\
0.0192307692307692	0.00320678662285056\\
0.0138888888888889	0.00159055295965156\\
0.0104166666666667	0.000853238245548917\\
0.00757575757575758	0.000440977855381957\\
0.00555555555555556	0.000233001170595547\\
0.00409836065573771	0.000126987009188324\\
0.00297619047619048	6.57029940065244e-05\\
};
\addlegendentry{$\Delta t$}

\addplot [color=blue,line width=1pt,solid,mark=square,mark options={solid}]
  table[row sep=crcr]{%
0.05	0.0509442596986033\\
0.0357142857142857	0.0245448015723878\\
0.025	0.0078999604493456\\
0.0192307692307692	0.00406070840604229\\
0.0138888888888889	0.00190564387065083\\
0.0104166666666667	0.0010135151262299\\
0.00757575757575758	0.000481389030732908\\
0.00555555555555556	0.000247208637001538\\
0.00409836065573771	0.000133368526309654\\
0.00297619047619048	6.74423983653384e-05\\
};
\addlegendentry{$\tfrac{\Delta t}{2}$}

\addplot [color=magenta,line width=1pt,solid,mark=triangle,mark options={solid}]
  table[row sep=crcr]{%
0.05	0.0665382362383158\\
0.0357142857142857	0.034899587643591\\
0.025	0.0117356875230989\\
0.0192307692307692	0.00602646677776558\\
0.0138888888888889	0.00261250618615322\\
0.0104166666666667	0.00139534320041976\\
0.00757575757575758	0.000563609783072611\\
0.00555555555555556	0.000274730114211762\\
0.00409836065573771	0.000144896029832171\\
0.00297619047619048	7.02052383382268e-05\\
};
\addlegendentry{$\tfrac{\Delta t}{4}$}

\addplot [color=red,line width=1pt,densely dashed]
  table[row sep=crcr]{%
0.05	0.0075\\
0.0357142857142857	0.0038265306122449\\
0.025	0.001875\\
0.0192307692307692	0.0011094674556213\\
0.0138888888888889	0.000578703703703704\\
0.0104166666666667	0.000325520833333333\\
0.00757575757575758	0.000172176308539945\\
0.00555555555555556	9.25925925925926e-05\\
0.00409836065573771	5.03896801934964e-05\\
0.00297619047619048	2.65731292517007e-05\\
};
\addlegendentry{$h^2$}

\end{axis}

\end{tikzpicture}%
		}
  \subfigure[CFM-$4^{th}$]{
		\setlength\figureheight{0.3\linewidth} 
		\setlength\figurewidth{0.325\linewidth} 
		\tikzset{external/export next=false}
%
%
\begin{tikzpicture}

\begin{axis}[%
width=0.951\figurewidth,
height=\figureheight,
at={(0\figurewidth,0\figureheight)},
scale only axis,
xmode=log,
xmin=0.00297619047619048,
xmax=0.1,
xminorticks=true,
xlabel style={font=\color{white!15!black}},
xlabel={\scriptsize$h$},
ymode=log,
ymin=1e-09,
ymax=0.1,
yminorticks=true,
ytick = {1e-9,1e-8,1e-7,1e-6,1e-5, 1e-4 ,1e-3, 1e-2, 1e-1,1},
ylabel={\scriptsize$\|\mathbold{U}-\mathbold{U}_h\|_{2}$},
axis background/.style={fill=white},
legend style={at={(0.625,0.5)},anchor=north west,legend cell align=left,align=left,draw=white!15!black,draw=none,fill=none},
legend style={font=\scriptsize},
ylabel style={yshift=-5pt},xlabel style={yshift=2.5pt},tick label style={font=\tiny} 
]
\addplot [color=black,line width=1pt,solid,mark=o,mark options={solid}]
  table[row sep=crcr]{%
0.05	0.0149989545438022\\
0.0357142857142857	0.00430029635332894\\
0.025	0.000795097325440973\\
0.0192307692307692	0.000255248858152883\\
0.0138888888888889	6.28316066960807e-05\\
0.0104166666666667	2.20434675131876e-05\\
0.00757575757575758	5.17431597935875e-06\\
0.00555555555555556	1.46811900336065e-06\\
0.00409836065573771	4.30722685477289e-07\\
0.00297619047619048	1.10958586381605e-07\\
};
\addlegendentry{$\Delta t$}

\addplot [color=blue,line width=1pt,solid,mark=square,mark options={solid}]
  table[row sep=crcr]{%
0.05	0.0174016399467724\\
0.0357142857142857	0.00527386110206346\\
0.025	0.000999838718497958\\
0.0192307692307692	0.000325929190888266\\
0.0138888888888889	7.98245646335547e-05\\
0.0104166666666667	2.74037507244095e-05\\
0.00757575757575758	6.37047885174097e-06\\
0.00555555555555556	1.81930511080291e-06\\
0.00409836065573771	5.05016343911628e-07\\
0.00297619047619048	1.28507597588004e-07\\
};
\addlegendentry{$\tfrac{\Delta t}{2}$}

\addplot [color=magenta,line width=1pt,solid,mark=triangle,mark options={solid}]
  table[row sep=crcr]{%
0.05	0.0203590201522571\\
0.0357142857142857	0.00645383698223833\\
0.025	0.00123180376550425\\
0.0192307692307692	0.000406016961667216\\
0.0138888888888889	9.8812178087414e-05\\
0.0104166666666667	3.29343515101294e-05\\
0.00757575757575758	7.55883916148052e-06\\
0.00555555555555556	2.18612937058249e-06\\
0.00409836065573771	5.79585349001316e-07\\
0.00297619047619048	1.53641389574457e-07\\
};
\addlegendentry{$\tfrac{\Delta t}{4}$}

\addplot [color=red,line width=1pt,densely dashed]
  table[row sep=crcr]{%
0.05	0.000625\\
0.0357142857142857	0.000162692628071637\\
0.025	3.90625e-05\\
0.0192307692307692	1.36768670564756e-05\\
0.0138888888888889	3.72108862978204e-06\\
0.0104166666666667	1.17737569926698e-06\\
0.00757575757575758	3.29385346916026e-07\\
0.00555555555555556	9.52598689224204e-08\\
0.00409836065573771	2.82124430000316e-08\\
0.00297619047619048	7.84590220252878e-09\\
};
\addlegendentry{$h^4$}

\end{axis}

\end{tikzpicture}%
		}
\end{adjustbox}
  \caption{Convergence plots in $L^2$-norm for a circular cavity problem using the proposed CFM-FDTD schemes for different values of $c_f$, that is $\Delta t$, $\tfrac{\Delta t}{2}$ and $\tfrac{\Delta t}{4}$.}
   \label{fig:convPlotCircCavityInterface}
\end{figure}
\begin{figure}
 \centering
 \begin{adjustbox}{max width=2.5\textwidth,center}
		\setlength\figureheight{0.35\linewidth} 
		\setlength\figurewidth{0.85\linewidth} 
		\tikzset{external/export next=false}
    		\input{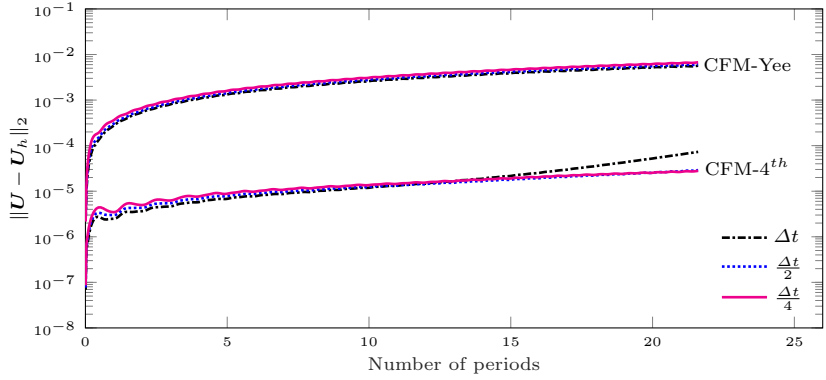}
\end{adjustbox}		
  \caption{Evolution in time of the error of $\mathbold{U}$ in $L^2$-norm for a circular cavity problem 
  using CFM-FDTD schemes with $h=\tfrac{1}{160}$ and different values of $c_f$, that is $\Delta t$, $\tfrac{\Delta t}{2}$ and $\tfrac{\Delta t}{4}$.}
   \label{fig:longSimulationCavity}
\end{figure}

\subsubsection{Square Cavity Problem} \label{sec:squareCavity}
Let us consider a PEC material with square holes.   
The domain is $\Omega = [-0.75,0.75]\times[-0.75,0.75]$ and the time interval is $I=[0,0.5]$.
Since $\Omega^-$ is a PEC subdomain, 
	the boundary $\Gamma$ then encloses subdomain $\Omega^+$. 
The boundary $\Gamma$ is a square of unit length centered at $(0,0)$ . 
The physical parameters are $\epsilon = 1$ and $\mu = 1$.
The time-step size is $\Delta t = \tfrac{h}{2}$ for both CFM-FDTD schemes.  
In $\Omega^+$, 
	the solution is given by 
\begin{equation*}
	\begin{aligned}
		H_x^+(x,y,t) =&\,\, -\frac{\pi\,n}{\omega}\,\sin(m\,\pi\,x)\,\cos(n\,\pi\,y)\,\sin(\omega\,t), \\
		H_y^+(x,y,t) =&\,\, \frac{\pi\,m}{\omega}\,\cos(m\,\pi\,x)\,\sin(n\,\pi\,y)\,\sin(\omega\,t), \\		
		E_z^+(x,y,t) =&\,\, \sin(m\,\pi\,x)\,\cos(n\,\pi\,y)\,\cos(\omega\,t),
	\end{aligned}
\end{equation*}
where $\omega = \pi\,\sqrt{m^2+n^2}$ with $m,n>0$ \cite{Hesthaven2008}.
In this numerical example, 
	we choose $m=n=4$. 
Convergence plots are illustrated in Fig.~\ref{fig:convPlotSquareCavityInterface} for both 
	proposed CFM-FDTD schemes.
A second and fourth order convergence are observed for respectively the CFM-Yee scheme and 
	the fourth-order CFM-FDTD scheme in $L^2$-norm.
The obtained convergence orders are in agreement with the theory.
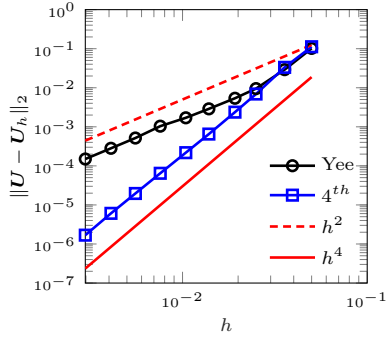
\begin{figure}
\begin{adjustbox}{max width=2.5\textwidth,center}
 \centering
		\setlength\figureheight{0.3\linewidth} 
		\setlength\figurewidth{0.325\linewidth} 
		\tikzset{external/export next=false}
%
%
\definecolor{mycolor1}{rgb}{1.00000,0.00000,1.00000}%
\begin{tikzpicture}

\begin{axis}[%
width=0.951\figurewidth,
height=\figureheight,
at={(0\figurewidth,0\figureheight)},
scale only axis,
xmode=log,
xmin=0.00297619047619048,
xmax=0.1,
xminorticks=true,
xlabel style={font=\color{white!15!black}},
xlabel={\scriptsize$h$},
ymode=log,
ymin=1e-07,
ymax=1,
yminorticks=true,
ytick = {1e-8,1e-7,1e-6,1e-5, 1e-4 ,1e-3, 1e-2, 1e-1,1},
ylabel={\scriptsize$\|\mathbold{U}-\mathbold{U}_h\|_{2}$},
axis background/.style={fill=white},
legend style={at={(0.62,0.5)},anchor=north west,legend cell align=left,align=left,draw=white!15!black,draw=none,fill=none},
legend style={font=\scriptsize},
ylabel style={yshift=-5pt},xlabel style={yshift=2.5pt},tick label style={font=\tiny} 
]
\addplot [color=black,line width=1pt,solid,mark=o,mark options={solid}]
  table[row sep=crcr]{%
0.05	0.101889433325655\\
0.0357142857142857	0.0290912298653827\\
0.025	0.00951069071949489\\
0.0192307692307692	0.00538190670556881\\
0.0138888888888889	0.00288479707645909\\
0.0104166666666667	0.0017010449540874\\
0.00757575757575758	0.00103811046874095\\
0.00555555555555556	0.000516183947720759\\
0.00409836065573771	0.000281869079001157\\
0.00297619047619048	0.000149853573387037\\
};
\addlegendentry{Yee}

\addplot [color=blue,line width=1pt,solid,mark=square,mark options={solid}]
  table[row sep=crcr]{%
0.05	0.112169332154982\\
0.0357142857142857	0.033308277194673\\
0.025	0.00691138397726731\\
0.0192307692307692	0.00237291915063484\\
0.0138888888888889	0.00065201321477437\\
0.0104166666666667	0.000217703172748734\\
0.00757575757575758	6.46051471842936e-05\\
0.00555555555555556	1.96445781495523e-05\\
0.00409836065573771	6.08909339231341e-06\\
0.00297619047619048	1.68287125427049e-06\\
};
\addlegendentry{$4^{th}$}

\addplot [color=red,line width=1pt,densely dashed]
  table[row sep=crcr]{%
0.05	0.125\\
0.0357142857142857	0.0637755102040816\\
0.025	0.03125\\
0.0192307692307692	0.018491124260355\\
0.0138888888888889	0.00964506172839506\\
0.0104166666666667	0.00542534722222222\\
0.00757575757575758	0.00286960514233242\\
0.00555555555555556	0.00154320987654321\\
0.00409836065573771	0.00083982800322494\\
0.00297619047619048	0.000442885487528345\\
};
\addlegendentry{$h^2$}

\addplot [color=red,line width=1pt,solid]
  table[row sep=crcr]{%
0.05	0.01875\\
0.0357142857142857	0.0048807788421491\\
0.025	0.001171875\\
0.0192307692307692	0.000410306011694268\\
0.0138888888888889	0.000111632658893461\\
0.0104166666666667	3.53212709780093e-05\\
0.00757575757575758	9.88156040748077e-06\\
0.00555555555555556	2.85779606767261e-06\\
0.00409836065573771	8.46373290000947e-07\\
0.00297619047619048	2.35377066075863e-07\\
};
\addlegendentry{$h^4$}

\end{axis}

\end{tikzpicture}%
\end{adjustbox}
  \caption{Convergence plots in $L^2$-norm for a square cavity problem using the proposed CFM-FDTD schemes.}
   \label{fig:convPlotSquareCavityInterface}
\end{figure}

\subsubsection{Two Concentric PEC Cylinders Problem} \label{sec:concentricPEC}
This problem considers a holed PEC containing a PEC cylinder 
	as illustrated in Figure~\ref{fig:concPEC}. 
It is recalled that $\Omega^-$ is a PEC subdomain.
We therefore have subdomain $\Omega^+$ enclosed by two PEC walls on $\Gamma_1$ and $\Gamma_2$.
There is two circular embedded boundaries centered at $(0,0)$ with radius $r_1=\tfrac{1}{3}$ and $r_2=1$ for respectively 
	$\Gamma_1$ and $\Gamma_2$. 
\begin{figure}[htbp]
 	\centering
 	\tdplotsetmaincoords{75}{105}
	\tikzset{external/export next=false}
  \begin{tikzpicture}[scale=0.75]
   	\draw[-latex,thick] (0,0)--(0,5)--(5,5)--(5,0)--cycle; 
	\draw[-latex,thick,blue] (2.5,2.5) circle [radius=2];
	\draw[-latex,thick,blue] (2.5,2.5) circle [radius=0.75];
  	\draw (3.25,3.3) node {$\color{blue}\Gamma_1$};
  	\draw (4.1,4.2) node {$\color{blue}\Gamma_2$};
  	\draw (0.75,4.5) node {$\Omega^-$};
   	\draw (1.6,3.5) node {$\Omega^+$};
   	\draw (2.5,2.5) node {$\Omega^-$};
  	\draw (0.5,0.3) node {$\partial \Omega$};
  \end{tikzpicture}
  \caption{Geometry of a two concentric PEC cylinders problem.}
\label{fig:concPEC}
\end{figure}
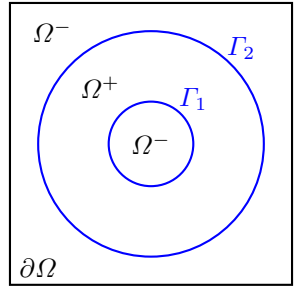
The domain is $\Omega = [-1.25,1.25]\times[-1.25,1.25]$ and the time interval is $I=[0,0.75]$.
The physical parameters are $\epsilon = \tfrac{1}{2}$ and $\mu = \tfrac{1}{2}$. 
The time-step size is $\Delta t = \tfrac{h}{4}$ for both CFM-FDTD schemes.  
In subdomain $\Omega^+$, 
	the solution in cylindrical coordinates is given by 
\begin{equation*}
	\begin{aligned}
		H_x^+(\rho,\phi,t) =&\,\, -\frac{1}{2}\,\sin(\omega\,t+\phi)\,\sin(\phi)\,\big(J_0(\tfrac{\omega\,\rho}{2})-J_2(\tfrac{\omega\,\rho}{2})+\alpha\,Y_0(\tfrac{\omega\,\rho}{2})-\alpha\,Y_2(\tfrac{\omega\,\rho}{2})\big)\\
		&\,\,-\frac{2\,\cos(\phi)}{\omega\,\rho}\,\cos(\omega\,t+\phi)\,\big(J_1(\tfrac{\omega\,\rho}{2})+\alpha\,Y_1(\tfrac{\omega\,\rho}{2})\big), \\
		H_y^+(\rho,\phi,t) =&\,\, \frac{1}{2}\,\sin(\omega\,t+\phi)\,\cos(\phi)\,\big(J_0(\tfrac{\omega\,\rho}{2})-J_2(\tfrac{\omega\,\rho}{2})+\alpha\,Y_0(\tfrac{\omega\,\rho}{2})-\alpha\,Y_2(\tfrac{\omega\,\rho}{2})\big)\\
		&\,\,-\frac{2\,\sin(\phi)}{\omega\,\rho}\,\cos(\omega\,t+\phi)\,\big(J_1(\tfrac{\omega\,\rho}{2})+\alpha\,Y_1(\tfrac{\omega\,\rho}{2})\big), \\		E_z^+(\rho,\phi,t) =&\,\, \cos(\omega\,t+\phi)\,\big(J_1(\tfrac{\omega\,\rho}{2})+\alpha\,Y_1(\tfrac{\omega\,\rho}{2})\big),
	\end{aligned}
\end{equation*}
where $\alpha \approx 1.76368380110927$, 
	$\omega \approx 9.813695999428405$, 
	and $J_i$ and $Y_i$ are the $i$-order Bessel function of respectively first and second kind \cite{Ditkowski2001}.	
Convergence plots are illustrated in Fig.~\ref{fig:convPlotTwoConCylInterface} for both 
	proposed CFM-FDTD schemes.
As expected, 
	we observe a second and third order convergence for respectively the CFM-Yee scheme and 
	the CFM-$4^{th}$ scheme in $L^2$-norm. 
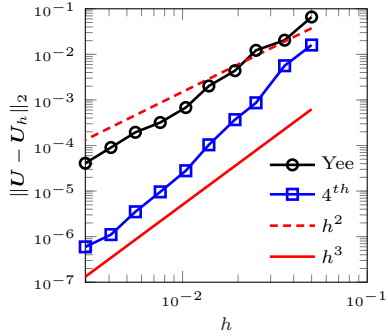
\begin{figure}
\begin{adjustbox}{max width=2.5\textwidth,center}
 \centering
		\setlength\figureheight{0.3\linewidth} 
		\setlength\figurewidth{0.325\linewidth} 
		\tikzset{external/export next=false}
%
%
\definecolor{mycolor1}{rgb}{1.00000,0.00000,1.00000}%
\begin{tikzpicture}

\begin{axis}[%
width=0.951\figurewidth,
height=\figureheight,
at={(0\figurewidth,0\figureheight)},
scale only axis,
xmode=log,
xmin=0.00297619047619048,
xmax=0.1,
xminorticks=true,
xlabel style={font=\color{white!15!black}},
xlabel={\scriptsize$h$},
ymode=log,
ymin=1e-07,
ymax=0.1,
yminorticks=true,
ytick = {1e-8,1e-7,1e-6,1e-5, 1e-4 ,1e-3, 1e-2, 1e-1,1},
ylabel={\scriptsize$\|\mathbold{U}-\mathbold{U}_h\|_{2}$},
axis background/.style={fill=white},
legend style={at={(0.62,0.5)},anchor=north west,legend cell align=left,align=left,draw=white!15!black,draw=none,fill=none},
legend style={font=\scriptsize},
ylabel style={yshift=-5pt},xlabel style={yshift=2.5pt},tick label style={font=\tiny} 
]
\addplot [color=black,line width=1pt,solid,mark=o,mark options={solid}]
  table[row sep=crcr]{%
0.05	0.0661289980983265\\
0.0357142857142857	0.0204395531746442\\
0.025	0.0122456902855226\\
0.0192307692307692	0.00437745234352367\\
0.0138888888888889	0.00202146896351383\\
0.0104166666666667	0.00068063782729368\\
0.00757575757575758	0.000319216281708078\\
0.00555555555555556	0.00019458516552045\\
0.00409836065573771	9.0314058389339e-05\\
0.00297619047619048	4.08007444672721e-05\\
};
\addlegendentry{Yee}

\addplot [color=blue,line width=1pt,solid,mark=square,mark options={solid}]
  table[row sep=crcr]{%
0.05	0.016024426340611\\
0.0357142857142857	0.00562698506964102\\
0.025	0.000864496044573546\\
0.0192307692307692	0.000369158572308801\\
0.0138888888888889	0.000103455936542195\\
0.0104166666666667	2.79376919941211e-05\\
0.00757575757575758	9.60308088891607e-06\\
0.00555555555555556	3.51275081224999e-06\\
0.00409836065573771	1.1081952194222e-06\\
0.00297619047619048	5.9540795681921e-07\\
};
\addlegendentry{$4^{th}$}

\addplot [color=red,line width=1pt,densely dashed]
  table[row sep=crcr]{%
0.05	0.0375\\
0.0357142857142857	0.0191326530612245\\
0.025	0.009375\\
0.0192307692307692	0.00554733727810651\\
0.0138888888888889	0.00289351851851852\\
0.0104166666666667	0.00162760416666667\\
0.00757575757575758	0.000860881542699725\\
0.00555555555555556	0.000462962962962963\\
0.00409836065573771	0.000251948400967482\\
0.00297619047619048	0.000132865646258503\\
};
\addlegendentry{$h^2$}

\addplot [color=red,line width=1pt,solid]
  table[row sep=crcr]{%
0.05	0.000625\\
0.0357142857142857	0.000227769679300292\\
0.025	7.8125e-05\\
0.0192307692307692	3.55598543468366e-05\\
0.0138888888888889	1.33959190672154e-05\\
0.0104166666666667	5.65140335648148e-06\\
0.00757575757575758	2.17394328964577e-06\\
0.00555555555555556	8.57338820301783e-07\\
0.00409836065573771	3.44191804600385e-07\\
0.00297619047619048	1.31811157002484e-07\\
};
\addlegendentry{$h^3$}

\end{axis}

\end{tikzpicture}%
\end{adjustbox}
  \caption{Convergence plots in $L^2$-norm for a two concentric PEC cylinders problem using the proposed CFM-FDTD schemes.}
   \label{fig:convPlotTwoConCylInterface}
\end{figure}

\subsection{Problems with a Manufactured Solution} \label{sec:scatteringPblms}
Let us now consider more complex embedded boundaries. 
To our knowledge, 
	there is no analytical solution for Maxwell's equations with arbitrary embedded boundaries.
We therefore consider that embedded boundary conditions are given by
\begin{equation*}
\begin{aligned}
	\hat{\mathbold{n}}\times\mathbold{E}^+=&\,\, \mathbold{a}(\mathbold{x},t) \quad \text{on } \Gamma \times I ,\\
	\hat{\mathbold{n}}\cdot (\mu\,\mathbold{H}^+) =&\,\, b(\mathbold{x},t) \quad \text{on } \Gamma \times I,
\end{aligned}
\end{equation*}
	where $\mathbold{a}(\mathbold{x},t)$ and $b(\mathbold{x},t)$ are known functions.
The physical parameters are $\epsilon = 1$ and $\mu = 2$. 
The time-step size is $\Delta t = \tfrac{h}{2}$ for both CFM-FDTD schemes.
We also consider the 5-star and 3-star embedded boundary illustrated in Fig.~\ref{fig:interfaceGeo}.
The solutions are 
\begin{equation*} 
	\begin{aligned}
		H_x^+ =&\,\, 0.5\,\sin(2\,\pi\,x)\,\sin(2\,\pi\,y)\,\sin(2\,\pi\,t), \\
		H_y^+ =&\,\, 0.5\,\cos(2\,\pi\,x)\,\cos(2\,\pi\,y)\,\sin(2\,\pi\,t), \\
		E_z^+ =&\,\, \sin(2\,\pi\,x)\,\cos(2\,\pi\,y)\,\cos(2\,\pi\,t) 
	\end{aligned}
\end{equation*}
	in $\Omega^+$ 
	while $H_x^-=H_y^-=0$ and $E_z^-=0$ in $\Omega^-$.
It is worth noting that manufactured solutions are at divergence-free in each subdomain, 
	but not in the whole domain because of embedded boundary conditions that we impose.
Nevertheless, 
	this allows us to assess the performance of the proposed CFM-FDTD schemes in a more 
	general framework.
The time interval is $I=[0,1]$. 
We set $\ell_h = 7\,h$ for both embedded boundaries.
The other parameters are the same as problems with analytical solutions in subsection~\ref{sec:pblmsAnalyticSol}.
Convergence plots for each geometry of the embedded boundary are illustrated in Fig.~\ref{fig:convPlotManuSol} using 
	both proposed CFM-FDTD schemes.
For the 5-star boundary, 
	a second-order convergence is observed for the CFM-Yee scheme.
We also observe a fourth order convergence for the CFM-$4^{th}$ scheme. 
As for the 3-star boundary, 
	we observe a second-order convergence for the CFM-Yee scheme
	while a third-order convergence is obtained for the CFM-$4^{th}$ scheme.
 \begin{figure} 
 \centering
\begin{adjustbox}{max width=2.5\textwidth,center}
  \subfigure[circular]{ \label{fig:circleInterfaceGeo}
		\setlength\figureheight{0.25\linewidth} 
		\setlength\figurewidth{0.25\linewidth} 
		\tikzset{external/export next=false}
%
%
\begin{tikzpicture}

\begin{axis}[%
width=\figurewidth,
height=\figureheight,
at={(0\figurewidth,0\figureheight)},
scale only axis,
xmin=0,
xmax=1,
xminorticks=true,
xtick = {0, 0.5 ,1},
ymin=0,
ymax=1,
yminorticks=true,
ytick = {0, 0.5 ,1},
axis background/.style={fill=white},
legend style={at={(0.01,0.99)},anchor=north west,legend cell align=left,align=left,draw=white!15!black,draw=none,fill=none},
legend style={font=\scriptsize},
ylabel style={yshift=-5pt},xlabel style={yshift=2.5pt},tick label style={font=\tiny} 
]
\addplot [color=black,line width=1pt,solid]
  table[row sep=crcr]{%
0.75	0.5\\
0.749496669117971	0.515855979914141\\
0.747988703207699	0.531648113393437\\
0.745482174315677	0.547312811090103\\
0.741987175349089	0.56278699679527\\
0.737517779435236	0.578008361424622\\
0.732091983254018	0.592915613915082\\
0.725731634571655	0.607448728022293\\
0.718462344267446	0.621549184025117\\
0.710313383207795	0.635160204363899\\
0.701317564382765	0.64822698226366\\
0.691511110779744	0.660696902421635\\
0.680933509526268	0.672519752870528\\
0.669627352889283	0.683647927164383\\
0.657638166771131	0.694036616072939\\
0.6450142273928	0.703643988012584\\
0.631806366902626	0.712431357487379\\
0.618067768693171	0.720363340861895\\
0.603853753250472	0.72740799883863\\
0.589221555397968	0.733536965066277\\
0.574230093832069	0.738725560361018\\
0.558939733877357	0.742952892080885\\
0.543412044416733	0.746201938253052\\
0.527709549975253	0.748459616115314\\
0.511895478955936	0.749716834795752\\
0.496033509041298	0.749968531918469\\
0.480187510785803	0.749213693987986\\
0.464421290431679	0.747455360470233\\
0.448798332983702	0.744700611553695\\
0.433381546577491	0.740960539639986\\
0.418233009170645	0.736250204678667\\
0.403413718576718	0.730588573526145\\
0.388983346848556	0.723998443572834\\
0.375	0.71650635094611\\
0.361519984033472	0.708142463658693\\
0.348597578215583	0.698940460132708\\
0.336284816513679	0.688937393588565\\
0.32463127807342	0.678173542844716\\
0.313683887581061	0.666692250129073\\
0.303486726314303	0.654539746555151\\
0.294080854642542	0.641764965965693\\
0.285504146691256	0.628419347893352\\
0.277791137836269	0.614556630431853\\
0.270972885641983	0.600232633851653\\
0.265076844803523	0.585505035831417\\
0.260126756596376	0.570433139210357\\
0.256142553278648	0.555077633196635\\
0.253140277830901	0.539500348993337\\
0.251132019356729	0.523764010826046\\
0.250125864404204	0.507931983374517\\
0.250125864404204	0.492068016625483\\
0.251132019356729	0.476235989173954\\
0.253140277830901	0.460499651006662\\
0.256142553278648	0.444922366803365\\
0.260126756596376	0.429566860789643\\
0.265076844803523	0.414494964168583\\
0.270972885641983	0.399767366148347\\
0.277791137836269	0.385443369568148\\
0.285504146691256	0.371580652106648\\
0.294080854642542	0.358235034034307\\
0.303486726314303	0.345460253444849\\
0.313683887581061	0.333307749870927\\
0.32463127807342	0.321826457155284\\
0.336284816513679	0.311062606411435\\
0.348597578215583	0.301059539867292\\
0.361519984033472	0.291857536341307\\
0.375	0.28349364905389\\
0.388983346848557	0.276001556427166\\
0.403413718576718	0.269411426473855\\
0.418233009170645	0.263749795321333\\
0.433381546577491	0.259039460360015\\
0.448798332983702	0.255299388446305\\
0.464421290431679	0.252544639529767\\
0.480187510785803	0.250786306012014\\
0.496033509041298	0.250031468081531\\
0.511895478955936	0.250283165204248\\
0.527709549975253	0.251540383884686\\
0.543412044416733	0.253798061746948\\
0.558939733877357	0.257047107919115\\
0.574230093832069	0.261274439638982\\
0.589221555397968	0.266463034933723\\
0.603853753250472	0.27259200116137\\
0.61806776869317	0.279636659138104\\
0.631806366902626	0.287568642512621\\
0.645014227392799	0.296356011987416\\
0.657638166771131	0.305963383927061\\
0.669627352889283	0.316352072835617\\
0.680933509526267	0.327480247129472\\
0.691511110779744	0.339303097578365\\
0.701317564382765	0.35177301773634\\
0.710313383207795	0.364839795636101\\
0.718462344267446	0.378450815974883\\
0.725731634571655	0.392551271977707\\
0.732091983254018	0.407084386084918\\
0.737517779435236	0.421991638575378\\
0.741987175349089	0.43721300320473\\
0.745482174315677	0.452687188909897\\
0.747988703207699	0.468351886606563\\
0.749496669117971	0.484144020085859\\
0.75	0.5\\
};
\end{axis}
\end{tikzpicture}%
		} 
  \subfigure[5-star]{\label{fig:starInterfaceGeo}
		\setlength\figureheight{0.25\linewidth} 
		\setlength\figurewidth{0.25\linewidth} 
		\tikzset{external/export next=false}
%
%
\begin{tikzpicture}

\begin{axis}[%
width=\figurewidth,
height=\figureheight,
at={(0\figurewidth,0\figureheight)},
scale only axis,
xmin=0,
xmax=1,
xminorticks=true,
xtick = {0, 0.5 ,1},
ymin=0,
ymax=1,
yminorticks=true,
ytick = {0, 0.5 ,1},
axis background/.style={fill=white},
legend style={at={(0.01,0.99)},anchor=north west,legend cell align=left,align=left,draw=white!15!black,draw=none,fill=none},
legend style={font=\scriptsize},
ylabel style={yshift=-5pt},xlabel style={yshift=2.5pt},tick label style={font=\tiny} 
]
\addplot [color=black,line width=1pt,solid]
  table[row sep=crcr]{%
0.75	0.5\\
0.765066930189006	0.516845499123648\\
0.777395596897263	0.535400996867555\\
0.785474949486592	0.555020786717682\\
0.788201994577402	0.574778085589939\\
0.785015355939205	0.593608069890764\\
0.77595740609245	0.610476680159308\\
0.76165735877021	0.624549447537036\\
0.743238589707334	0.6353343167743\\
0.722163808645616	0.64277601035563\\
0.700040086323819	0.647286395096484\\
0.678410979268275	0.649704586901987\\
0.658564374561849	0.651190825763033\\
0.641382028781861	0.653067981559738\\
0.627250504584707	0.656632481897858\\
0.616044232300894	0.662961322323235\\
0.607181039780345	0.672742897877369\\
0.599740192449566	0.686156495288\\
0.592624237639838	0.702818789561018\\
0.584739992587695	0.721806497324522\\
0.57517168653353	0.74175374251308\\
0.563320871123122	0.761012185776997\\
0.54899298926918	0.777853049331026\\
0.532418651828539	0.790684106941049\\
0.514207511767961	0.798252376567682\\
0.495242705950108	0.799804996894593\\
0.47653268388508	0.795186358147882\\
0.459043571509212	0.784858217139645\\
0.443538105232274	0.769840045925559\\
0.430446197184106	0.751577768613709\\
0.41978749849487	0.731758802741356\\
0.411158596481512	0.712098573477648\\
0.403787839421047	0.694127399905181\\
0.396650635094611	0.67900635094611\\
0.388628899709452	0.667396393140633\\
0.378691488291971	0.659397523849207\\
0.366068930318577	0.654564693923155\\
0.350396159890447	0.651996581389543\\
0.331801144729265	0.650483204539873\\
0.310924802629974	0.64869037868741\\
0.28886729267282	0.645354240940172\\
0.267066301407991	0.639458144361026\\
0.247122803447367	0.63036725568928\\
0.230597541587669	0.617902712298803\\
0.218806015883607	0.602346240273091\\
0.212640420665489	0.584376385492237\\
0.212443602463266	0.564947476486013\\
0.217953180803277	0.545130682692334\\
0.228324437528831	0.525941870830668\\
0.24222977195585	0.508182636263719\\
0.258021956852558	0.492318669514685\\
0.273939601184627	0.478413849178577\\
0.288327374858526	0.466129984705659\\
0.299841504094031	0.454792210092742\\
0.307613092527262	0.443510107071522\\
0.311347673723439	0.431336168610256\\
0.311348229696296	0.417437444595497\\
0.308459472225172	0.401253994825574\\
0.303941991974521	0.382619448574323\\
0.299294416612264	0.361824309008787\\
0.296048649998632	0.339610885577108\\
0.295566630432858	0.317098704281728\\
0.298866396256393	0.295649495700111\\
0.30650070270878	0.276689906746026\\
0.318503668139195	0.261516603583791\\
0.334411068357493	0.251111465823247\\
0.353349364905389	0.24599364905389\\
0.374178854276066	0.246130512759513\\
0.395668840671924	0.250921426425358\\
0.416678519846419	0.259258393384021\\
0.436316895970877	0.269656689333738\\
0.454058560735131	0.280438822818169\\
0.469799009354145	0.289947496199178\\
0.483842337686526	0.296758970171911\\
0.496824312132489	0.299867933057655\\
0.50958344614391	0.298818706976178\\
0.523000448121967	0.293764874710422\\
0.537831099564285	0.285449172824922\\
0.554558596631592	0.275106401615227\\
0.573288501130608	0.264302621791043\\
0.593703118208241	0.254732567191969\\
0.615083268861105	0.248002791883759\\
0.636395344936775	0.245429813564209\\
0.656431694024907	0.247880182902612\\
0.673984222484705	0.255673346298067\\
0.688025828957554	0.268559249751979\\
0.697872676996705	0.285772127230971\\
0.703302644490686	0.306151320021977\\
0.704611242291214	0.328310782058717\\
0.702595042441711	0.350832430569164\\
0.698462957769975	0.372455601627831\\
0.693686098827558	0.392235948724066\\
0.689805910373101	0.40965199149245\\
0.688226560415586	0.424645452329144\\
0.690020202931268	0.437591347041521\\
0.695772356120776	0.449204091999399\\
0.705489399144761	0.460395164537477\\
0.718581809518134	0.47210477008068\\
0.733926408046936	0.485133539295366\\
0.75	0.5\\
};
\end{axis}
\end{tikzpicture}%
		} 
  \subfigure[3-star]{\label{fig:triStarInterfaceGeo}
		\setlength\figureheight{0.25\linewidth} 
		\setlength\figurewidth{0.25\linewidth} 
		\tikzset{external/export next=false}
%
%
\begin{tikzpicture}

\begin{axis}[%
width=\figurewidth,
height=\figureheight,
at={(0\figurewidth,0\figureheight)},
scale only axis,
xmin=0,
xmax=1,
xminorticks=true,
xtick = {0, 0.5 ,1},
ymin=0,
ymax=1,
yminorticks=true,
ytick = {0, 0.5 ,1},
axis background/.style={fill=white},
legend style={at={(0.01,0.99)},anchor=north west,legend cell align=left,align=left,draw=white!15!black,draw=none,fill=none},
legend style={font=\scriptsize},
ylabel style={yshift=-5pt},xlabel style={yshift=2.5pt},tick label style={font=\tiny} 
]
\addplot [color=black,line width=1pt,solid]
  table[row sep=crcr]{%
0.75	0.55\\
0.768383691155712	0.567056285485863\\
0.784855939371776	0.586353079609766\\
0.798569247672799	0.607544505835653\\
0.808783283651588	0.630118192268234\\
0.814908287960172	0.653425855534761\\
0.816539300787274	0.676723220029148\\
0.813479080096982	0.699216694779779\\
0.805748304479489	0.720113788035931\\
0.793582461693234	0.738673991711258\\
0.777415669648107	0.754256830101566\\
0.757852505596638	0.766363942344277\\
0.735629678650627	0.774672444790369\\
0.711570021888213	0.779057373755364\\
0.686531762311539	0.779601705541283\\
0.661356319017503	0.776593244583852\\
0.636817963590651	0.770508511213973\\
0.61357854671968	0.761984593953246\\
0.592150159874138	0.75178070405503\\
0.57286808220072	0.740731832587438\\
0.555875694011952	0.729697420209933\\
0.541122262361211	0.719508274169579\\
0.528373671098689	0.710915085057808\\
0.517235331010121	0.704541799856712\\
0.507185718903124	0.700846803803962\\
0.497618308349223	0.700094370975715\\
0.48788911327206	0.702338193667119\\
0.477366703462387	0.707418035147949\\
0.465481391646248	0.714969718104235\\
0.4517703451712	0.724447815239603\\
0.435915638275791	0.735159603166056\\
0.417772716391469	0.74630812733257\\
0.397387362750012	0.757041649900452\\
0.375	0.76650635094611\\
0.351036977897187	0.773898951759353\\
0.326089339866257	0.778515940109955\\
0.300880372261704	0.779796300355625\\
0.276223968364629	0.777355071772135\\
0.252976437718672	0.771005649468394\\
0.231984821811756	0.760769465763369\\
0.214035011603472	0.746872499529074\\
0.199802985787138	0.729728904825291\\
0.189812299298019	0.709912874076284\\
0.184400561411509	0.68812060227235\\
0.183697076668585	0.66512484910402\\
0.18761311595587	0.641725065201822\\
0.195845484459602	0.61869632675674\\
0.207893209093245	0.596740392003832\\
0.223086334355692	0.576442049038377\\
0.240625045742688	0.558233576556572\\
0.259626683065719	0.542369609807538\\
0.279177704357766	0.528914027386286\\
0.298387346568558	0.517739694017157\\
0.316439622097695	0.50854106036347\\
0.332640397236881	0.500858786781107\\
0.34645661293846	0.494114777441185\\
0.357545209872456	0.487655334569043\\
0.365769976374519	0.480799613212578\\
0.371205307595373	0.472890209038588\\
0.374126697681612	0.463342567597689\\
0.374988630816851	0.451689972653066\\
0.374391337443451	0.437621149210248\\
0.37303858778221	0.421007986082703\\
0.371689260765654	0.401921513178496\\
0.37110581656491	0.380635019844539\\
0.372002990169758	0.357614024441967\\
0.375	0.33349364905389\\
0.380579330947101	0.309044762754784\\
0.389054720761967	0.285130980280279\\
0.400550380065498	0.262659193808722\\
0.414992747983782	0.242526735959632\\
0.432115274321156	0.225568494996845\\
0.45147587740097	0.212507314207483\\
0.472485908299546	0.203910805691148\\
0.494448709733373	0.200157307138778\\
0.516605239008747	0.201413134212458\\
0.538183768940384	0.207622567626085\\
0.558450417734776	0.218511208551704\\
0.576757205393503	0.233602490007808\\
0.592584493652186	0.252246299487896\\
0.605575028595216	0.273657902454885\\
0.615557346626805	0.29696470637777\\
0.622556990666661	0.321257912229455\\
0.6267947702146	0.345645796239216\\
0.628672135768096	0.369305268558684\\
0.628744571230722	0.391528473395405\\
0.627684683890353	0.411761519426597\\
0.626237340401908	0.429632939049313\\
0.625169715962851	0.444970137501007\\
0.625219459117422	0.457802865574245\\
0.627044304722356	0.468353582983459\\
0.631176384055404	0.477015419985697\\
0.637984189046328	0.484319238735193\\
0.647644665720762	0.490891992198984\\
0.660127270910301	0.497409132685517\\
0.67519106704659	0.504544198677694\\
0.692395100958555	0.512918883655448\\
0.711121467043621	0.523056852822891\\
0.730609647080229	0.535344325657581\\
0.75	0.55\\
};
\end{axis}
\end{tikzpicture}%
		} 
\end{adjustbox}
  \caption{Different geometries of an embedded PEC.}
  \label{fig:interfaceGeo}
\end{figure}
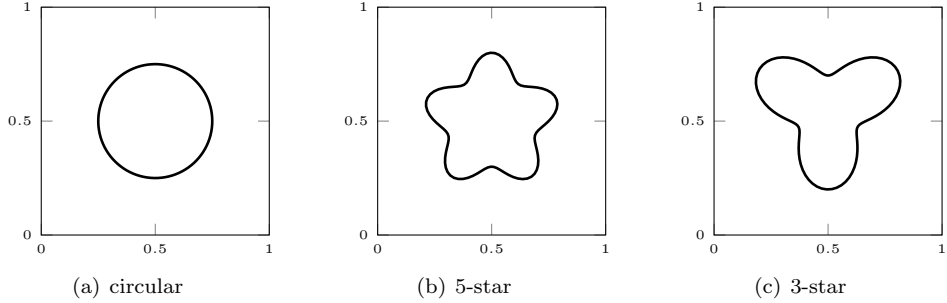
\begin{figure}
\begin{adjustbox}{max width=2.5\textwidth,center}
 \centering
  \subfigure[5-star]{
		\setlength\figureheight{0.3\linewidth} 
		\setlength\figurewidth{0.325\linewidth} 
		\tikzset{external/export next=false}
%
%
\definecolor{mycolor1}{rgb}{1.00000,0.00000,1.00000}%
\begin{tikzpicture}

\begin{axis}[%
width=0.951\figurewidth,
height=\figureheight,
at={(0\figurewidth,0\figureheight)},
scale only axis,
xmode=log,
xmin=0.00217391304347826,
xmax=0.1,
xminorticks=true,
xlabel style={font=\color{white!15!black}},
xlabel={\scriptsize$h$},
ymode=log,
ymin=1e-09,
ymax=0.1,
yminorticks=true,
ytick = {1e-8,1e-7,1e-6,1e-5, 1e-4 ,1e-3, 1e-2, 1e-1,1},
ylabel={\scriptsize$\|\mathbold{U}-\mathbold{U}_h\|_{2}$},
axis background/.style={fill=white},
legend style={at={(0.62,0.5)},anchor=north west,legend cell align=left,align=left,draw=white!15!black,draw=none,fill=none},
legend style={font=\scriptsize},
ylabel style={yshift=-5pt},xlabel style={yshift=2.5pt},tick label style={font=\tiny} 
]
\addplot [color=black,line width=1pt,solid,mark=o,mark options={solid}]
  table[row sep=crcr]{%
0.05	0.022338873305057\\
0.0357142857142857	0.0150064954701902\\
0.025	0.00330530455720858\\
0.0192307692307692	0.00127861118510874\\
0.0138888888888889	0.000607372536163846\\
0.0104166666666667	0.000339695901522835\\
0.00757575757575758	0.000194236046886483\\
0.00555555555555556	7.34003645093942e-05\\
0.00409836065573771	3.79060193301799e-05\\
0.00297619047619048	1.78540340792305e-05\\
0.00217391304347826	9.84032929900934e-06\\
};
\addlegendentry{Yee}

\addplot [color=blue,line width=1pt,solid,mark=square,mark options={solid}]
  table[row sep=crcr]{%
0.05	0.0216197198991167\\
0.0357142857142857	0.00137125493033694\\
0.025	0.000417758188580665\\
0.0192307692307692	9.60082320486095e-05\\
0.0138888888888889	3.10474923670031e-05\\
0.0104166666666667	1.76503100537608e-05\\
0.00757575757575758	2.97456355589466e-06\\
0.00555555555555556	4.67913752468585e-07\\
0.00409836065573771	1.78797247706408e-07\\
0.00297619047619048	5.3762367183194e-08\\
0.00217391304347826	2.19377827117848e-08\\
};
\addlegendentry{$4^{th}$}

\addplot [color=red,line width=1pt,densely dashed]
  table[row sep=crcr]{%
0.05	0.025\\
0.0357142857142857	0.0127551020408163\\
0.025	0.00625\\
0.0192307692307692	0.00369822485207101\\
0.0138888888888889	0.00192901234567901\\
0.0104166666666667	0.00108506944444444\\
0.00757575757575758	0.000573921028466483\\
0.00555555555555556	0.000308641975308642\\
0.00409836065573771	0.000167965600644988\\
0.00297619047619048	8.85770975056689e-05\\
0.00217391304347826	4.72589792060492e-05\\
};
\addlegendentry{$h^2$}

\addplot [color=red,line width=1pt,solid]
  table[row sep=crcr]{%
0.05	0.0009375\\
0.0357142857142857	0.000244038942107455\\
0.025	5.859375e-05\\
0.0192307692307692	2.05153005847134e-05\\
0.0138888888888889	5.58163294467307e-06\\
0.0104166666666667	1.76606354890046e-06\\
0.00757575757575758	4.94078020374038e-07\\
0.00555555555555556	1.42889803383631e-07\\
0.00409836065573771	4.23186645000474e-08\\
0.00297619047619048	1.17688533037932e-08\\
0.00217391304347826	3.350116673394747e-09\\
};
\addlegendentry{$h^4$}

\end{axis}

\end{tikzpicture}%
		}
  \subfigure[3-star]{
		\setlength\figureheight{0.3\linewidth} 
		\setlength\figurewidth{0.325\linewidth} 
		\tikzset{external/export next=false}
%
%
\definecolor{mycolor1}{rgb}{1.00000,0.00000,1.00000}%
\begin{tikzpicture}

\begin{axis}[%
width=0.951\figurewidth,
height=\figureheight,
at={(0\figurewidth,0\figureheight)},
scale only axis,
xmode=log,
xmin=0.00217391304347826,
xmax=0.1,
xminorticks=true,
xlabel style={font=\color{white!15!black}},
xlabel={\scriptsize$h$},
ymode=log,
ymin=1e-09,
ymax=0.1,
yminorticks=true,
ytick = {1e-9,1e-8,1e-7,1e-6,1e-5, 1e-4 ,1e-3, 1e-2, 1e-1,1},
ylabel={\scriptsize$\|\mathbold{U}-\mathbold{U}_h\|_{2}$},
axis background/.style={fill=white},
legend style={at={(0.62,0.5)},anchor=north west,legend cell align=left,align=left,draw=white!15!black,draw=none,fill=none},
legend style={font=\scriptsize},
ylabel style={yshift=-5pt},xlabel style={yshift=2.5pt},tick label style={font=\tiny} 
]
\addplot [color=black,line width=1pt,solid,mark=o,mark options={solid}]
  table[row sep=crcr]{%
0.05	0.0321539446483476\\
0.0357142857142857	0.00989285582168451\\
0.025	0.00406870942896455\\
0.0192307692307692	0.00185188935061162\\
0.0138888888888889	0.000479669390131065\\
0.0104166666666667	0.000299011000301145\\
0.00757575757575758	0.000121110105692554\\
0.00555555555555556	6.38356843408062e-05\\
0.00409836065573771	3.5464371502186e-05\\
0.00297619047619048	2.00484440095081e-05\\
0.00217391304347826	1.04476488045159e-05\\
};
\addlegendentry{Yee}

\addplot [color=blue,line width=1pt,solid,mark=square,mark options={solid}]
  table[row sep=crcr]{%
0.05	0.0094934656513759\\
0.0357142857142857	0.00175607060754307\\
0.025	0.000452719954365652\\
0.0192307692307692	7.5823962711877e-05\\
0.0138888888888889	4.94005908536265e-05\\
0.0104166666666667	2.85743061136718e-05\\
0.00757575757575758	1.65178488501363e-06\\
0.00555555555555556	6.50699419146073e-07\\
0.00409836065573771	2.17924306242816e-07\\
0.00297619047619048	8.66903823581193e-08\\
0.00217391304347826	2.72718910752209e-08\\
};
\addlegendentry{$4^{th}$}

\addplot [color=red,line width=1pt,densely dashed]
  table[row sep=crcr]{%
0.05	0.025\\
0.0357142857142857	0.0127551020408163\\
0.025	0.00625\\
0.0192307692307692	0.00369822485207101\\
0.0138888888888889	0.00192901234567901\\
0.0104166666666667	0.00108506944444444\\
0.00757575757575758	0.000573921028466483\\
0.00555555555555556	0.000308641975308642\\
0.00409836065573771	0.000167965600644988\\
0.00297619047619048	8.85770975056689e-05\\
0.00217391304347826	4.72589792060492e-05\\
};
\addlegendentry{$h^2$}


\addplot [color=red,line width=1pt,solid]
  table[row sep=crcr]{%
0.05	6.25e-05\\
0.0357142857142857	2.27769679300292e-05\\
0.025	7.8125e-06\\
0.0192307692307692	3.55598543468366e-06\\
0.0138888888888889	1.33959190672154e-06\\
0.0104166666666667	5.65140335648148e-07\\
0.00757575757575758	2.17394328964577e-07\\
0.00555555555555556	8.57338820301783e-08\\
0.00409836065573771	3.44191804600385e-08\\
0.00297619047619048	1.31811157002484e-08\\
0.00217391304347826	5.13684556587491e-09\\
};
\addlegendentry{$h^3$}
\end{axis}

\end{tikzpicture}%
		}
\end{adjustbox}
  \caption{Convergence plots in $L^2$-norm for problems with a manufactured solution using the proposed CFM-FDTD schemes and different geometries of the embedded boundary.}
   \label{fig:convPlotManuSol}
\end{figure}
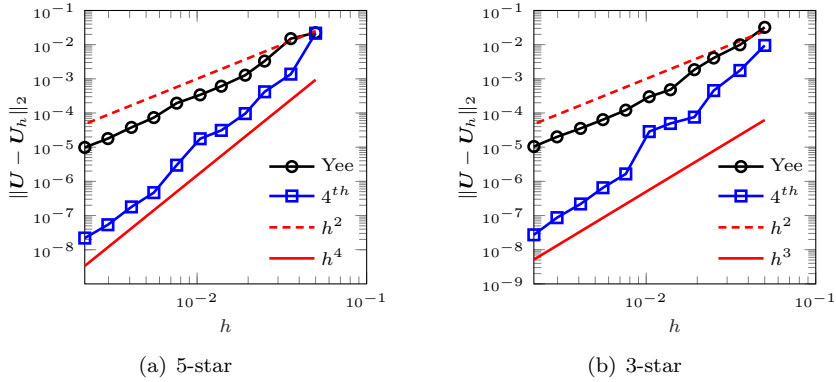
\subsection{Scattering Problems} \label{sec:scatteringPblms}
Let us now consider scattering problems involving various geometries of a PEC.   
To our knowledge, 
	there is no analytic solution for these problems with arbitrary geometries of the embedded boundary.
Hence,
	we estimate errors using approximate solutions coming from a very fine grid. 
The domain is $\Omega = [-1,1.5]\times[-0.75,1.75]$ and the time interval is $I=[0,1.5]$.
Periodic conditions are used on all $\partial \Omega$.
We consider embedded boundaries illustrated in Fig.~\ref{fig:interfaceGeo}.
The physical parameters are $\epsilon = 1$ and $\mu = 1$.
The mesh grid size is $h=\Delta x = \Delta y$. 
The time-step size is $\Delta t = \tfrac{h}{2}$.
The parameters for local patches are 
	$\ell_h = 6\,h$ for the circular PEC, 
	and $\ell_h = 7\,h$ for the 5-star PEC and 3-star PEC.
The other parameters are the same as problems with analytical solutions in subsection~\ref{sec:pblmsAnalyticSol}.

Let us consider a pulsed wave propagating in the positive $x$-direction, 
	given by
\begin{equation} \label{eq:pulsedWave}
	\begin{aligned}
		H_{x_{p}}(\mathbold{x},t) =&\,\, 0, \\
		H_{y_p}(\mathbold{x},t) =&\,\,-\frac{2}{\sigma^2}\,(x-\gamma-t)\,e^{-\big(\frac{x-\gamma-t}{\sigma}\big)^2}, \\
		E_{z_p}(\mathbold{x},t) =&\,\, \frac{2}{\sigma^2}\,(x-\gamma-t)\,e^{-\big(\frac{x-\gamma-t}{\sigma}\big)^2},
	\end{aligned}
\end{equation}
	where $\sigma = 0.1$ and $\gamma = -0.3$.
It is worth mentioning that we use electromagnetic fields given in \eqref{eq:pulsedWave} to compute all 
	previous solutions needed to initialize the time-stepping method presented in Section~\ref{sec:fourthOrderScheme}. 
It is recalled that we set $H_x^- = H_y^- =0$ and $E_z^-=0$ in the PEC subdomain, 
	that is $\Omega^-$.
The reference solution $\mathbold{U}^\star$
	is computed using $h=\tfrac{1}{1620}$ and the CFM-$4^{th}$ scheme.
All nodes used for $H_x$, 
	$H_y$ and $E_z$ in coarser grids with 
	$h\in\{\tfrac{1}{20},\tfrac{1}{60},\tfrac{1}{180},\tfrac{1}{540}\}$ are also part of the finer grid. 
Fig.~\ref{fig:convPlotScatteringPblms} illustrates convergence plots for each geometry of the embedded PEC using both CFM-FDTD schemes.
We observe a second and fourth order convergence for respectively the CFM-Yee scheme and 
	the CFM-$4^{th}$ scheme.
\begin{figure}
 \centering
  \subfigure[circular]{
		\setlength\figureheight{0.3\linewidth} 
		\setlength\figurewidth{0.325\linewidth} 
		\tikzset{external/export next=false}
%
%
\begin{tikzpicture}

\begin{axis}[%
width=0.951\figurewidth,
height=\figureheight,
at={(0\figurewidth,0\figureheight)},
scale only axis,
xmode=log,
xmin=0.001,
xmax=0.1,
xminorticks=true,
xlabel style={font=\color{white!15!black}},
xlabel={\scriptsize$h$},
ymode=log,
ymin=1e-06,
ymax=100,
yminorticks=true,
ytick = {1e-8,1e-7,1e-6,1e-5, 1e-4 ,1e-3, 1e-2, 1e-1,1,10,100,1000},
ylabel={\scriptsize$\|\mathbold{U}^\star-\mathbold{U}_h\|_{2}$},
axis background/.style={fill=white},
legend style={at={(0.62,0.5)},anchor=north west,legend cell align=left,align=left,draw=white!15!black,draw=none,fill=none},
legend style={font=\scriptsize},
ylabel style={yshift=-5pt},xlabel style={yshift=2.5pt},tick label style={font=\tiny} 
]
\addplot [color=black,line width=1pt,solid,mark=o,mark options={solid}]
  table[row sep=crcr]{%
0.05	6.66782119261563\\
0.0166666666666667	1.02340994612651\\
0.00555555555555556	0.113072547157417\\
0.00185185185185185	0.0125067195273335\\
};
\addlegendentry{Yee}

\addplot [color=blue,line width=1pt,solid,mark=square,mark options={solid}]
  table[row sep=crcr]{%
0.05	3.69104209587904\\
0.0166666666666667	0.0699891869554908\\
0.00555555555555556	0.000981919947682604\\
0.00185185185185185	1.48779378370159e-05\\
};
\addlegendentry{$4^{th}$}

\addplot [color=red,line width=1pt,densely dashed]
  table[row sep=crcr]{%
0.05	50\\
0.0166666666666667	5.55555555555556\\
0.00555555555555556	0.617283950617284\\
0.00185185185185185	0.0685871056241427\\
};
\addlegendentry{$h^2$}

\addplot [color=red,line width=1pt,solid]
  table[row sep=crcr]{%
0.05	0.75\\
0.0166666666666667	0.00925925925925926\\
0.00555555555555556	0.000114311842706904\\
0.00185185185185185	1.41125731736919e-06\\
};
\addlegendentry{$h^4$}

\end{axis}

\end{tikzpicture}%
		}
  \subfigure[5-star]{
		\setlength\figureheight{0.3\linewidth} 
		\setlength\figurewidth{0.325\linewidth} 
		\tikzset{external/export next=false}
%
%
\begin{tikzpicture}

\begin{axis}[%
width=0.951\figurewidth,
height=\figureheight,
at={(0\figurewidth,0\figureheight)},
scale only axis,
xmode=log,
xmin=0.001,
xmax=0.1,
xminorticks=true,
xlabel style={font=\color{white!15!black}},
xlabel={\scriptsize$h$},
ymode=log,
ymin=1e-06,
ymax=100,
yminorticks=true,
ytick = {1e-8,1e-7,1e-6,1e-5, 1e-4 ,1e-3, 1e-2, 1e-1,1,10,100,1000},
ylabel={\scriptsize$\|\mathbold{U}^\star-\mathbold{U}_h\|_{2}$},
axis background/.style={fill=white},
legend style={at={(0.62,0.5)},anchor=north west,legend cell align=left,align=left,draw=white!15!black,draw=none,fill=none},
legend style={font=\scriptsize},
ylabel style={yshift=-5pt},xlabel style={yshift=2.5pt},tick label style={font=\tiny} 
]
\addplot [color=black,line width=1pt,solid,mark=o,mark options={solid}]
  table[row sep=crcr]{%
0.05	7.89191458565988\\
0.0166666666666667	1.57265223271697\\
0.00555555555555556	0.120262026168658\\
0.00185185185185185	0.012550717571674\\
};
\addlegendentry{Yee}

\addplot [color=blue,line width=1pt,solid,mark=square,mark options={solid}]
  table[row sep=crcr]{%
0.05	5.72236605937382\\
0.0166666666666667	0.376006099674177\\
0.00555555555555556	0.0048716093586987\\
0.00185185185185185	3.40034762328334e-05\\
};
\addlegendentry{$4^{th}$}

\addplot [color=red,line width=1pt,densely dashed]
  table[row sep=crcr]{%
0.05	50\\
0.0166666666666667	5.55555555555556\\
0.00555555555555556	0.617283950617284\\
0.00185185185185185	0.0685871056241427\\
};
\addlegendentry{$h^2$}

\addplot [color=red,line width=1pt,solid]
  table[row sep=crcr]{%
0.05	3.125\\
0.0166666666666667	0.0385802469135802\\
0.00555555555555556	0.000476299344612102\\
0.00185185185185185	5.88023882237163e-06\\
};
\addlegendentry{$h^4$}

\end{axis}

\end{tikzpicture}%
		}
  \subfigure[3-star]{
		\setlength\figureheight{0.3\linewidth} 
		\setlength\figurewidth{0.325\linewidth} 
		\tikzset{external/export next=false}
%
%
\begin{tikzpicture}

\begin{axis}[%
width=0.951\figurewidth,
height=\figureheight,
at={(0\figurewidth,0\figureheight)},
scale only axis,
xmode=log,
xmin=0.001,
xmax=0.1,
xminorticks=true,
xlabel style={font=\color{white!15!black}},
xlabel={\scriptsize$h$},
ymode=log,
ymin=1e-06,
ymax=100,
yminorticks=true,
ytick = {1e-8,1e-7,1e-6,1e-5, 1e-4 ,1e-3, 1e-2, 1e-1,1,10,100,1000},
ylabel={\scriptsize$\|\mathbold{U}^\star-\mathbold{U}_h\|_{2}$},
axis background/.style={fill=white},
legend style={at={(0.62,0.5)},anchor=north west,legend cell align=left,align=left,draw=white!15!black,draw=none,fill=none},
legend style={font=\scriptsize},
ylabel style={yshift=-5pt},xlabel style={yshift=2.5pt},tick label style={font=\tiny} 
]
\addplot [color=black,line width=1pt,solid,mark=o,mark options={solid}]
  table[row sep=crcr]{%
0.05	7.79533813241798\\
0.0166666666666667	1.18136200023526\\
0.00555555555555556	0.118112140930244\\
0.00185185185185185	0.0124133134494936\\
};
\addlegendentry{Yee}

\addplot [color=blue,line width=1pt,solid,mark=square,mark options={solid}]
  table[row sep=crcr]{%
0.05	5.34554998542575\\
0.0166666666666667	0.150696971035038\\
0.00555555555555556	0.00197468590830428\\
0.00185185185185185	1.67403128345419e-05\\
};
\addlegendentry{$4^{th}$}

\addplot [color=red,line width=1pt,densely dashed]
  table[row sep=crcr]{%
0.05	50\\
0.0166666666666667	5.55555555555556\\
0.00555555555555556	0.617283950617284\\
0.00185185185185185	0.0685871056241427\\
};
\addlegendentry{$h^2$}

\addplot [color=red,line width=1pt,solid]
  table[row sep=crcr]{%
0.05	1.25\\
0.0166666666666667	0.0154320987654321\\
0.00555555555555556	0.000190519737844841\\
0.00185185185185185	2.35209552894865e-06\\
};
\addlegendentry{$h^4$}

\end{axis}

\end{tikzpicture}%
		}
  \caption{Convergence plots in $L^2$-norm for scattering problems using the proposed CFM-FDTD schemes and different geometries of the embedded PEC.}
   \label{fig:convPlotScatteringPblms}
\end{figure}
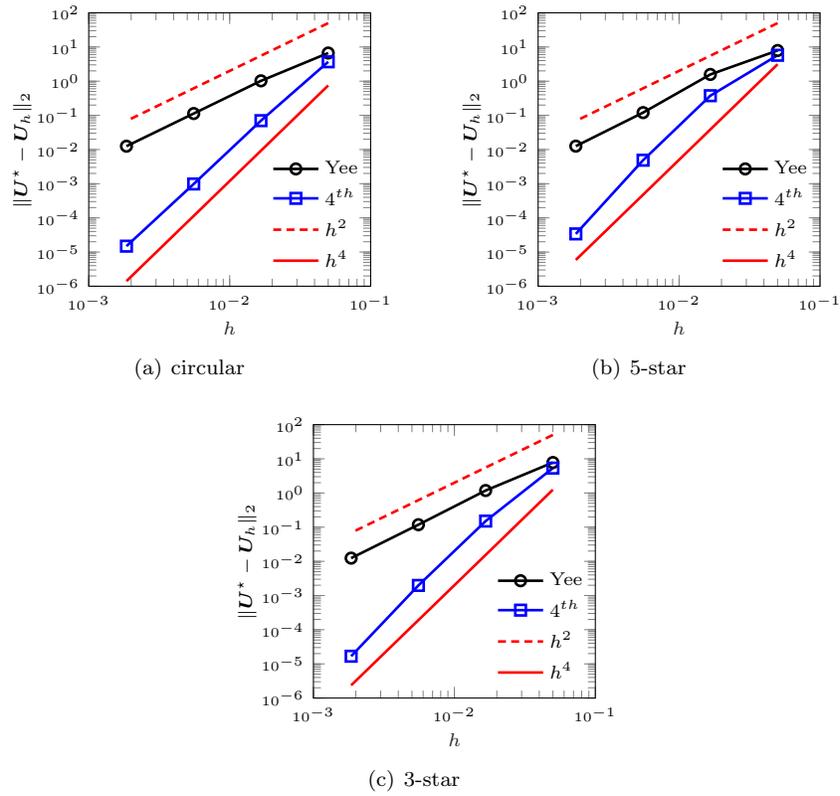
Fig.~\ref{fig:circularScattering}, 
	Fig.~\ref{fig:fiveStarScattering} and Fig.~\ref{fig:triStarScattering} illustrate the evolution of the magnitude of 
	each component of electromagnetic fields.
The numerical approach can handle various geometries of the embedded boundary without significantly increasing 
	the complexity of the method.
\begin{figure}
 \centering
\begin{adjustbox}{max width=1.25\textwidth,center}
\subfigure[$H_x$]{
 	{\includegraphics[width=2.25in]{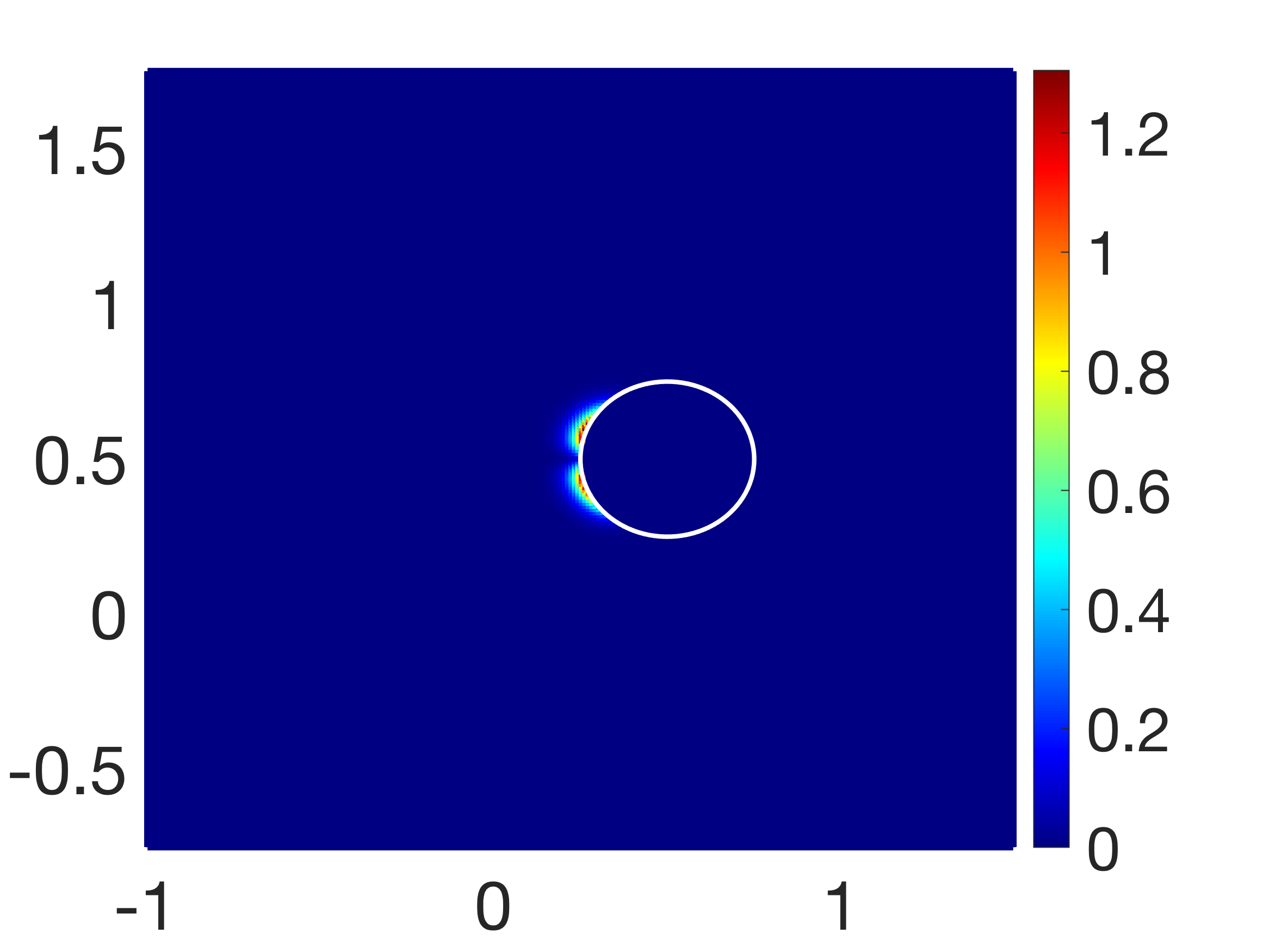}}
 	{\includegraphics[width=2.25in]{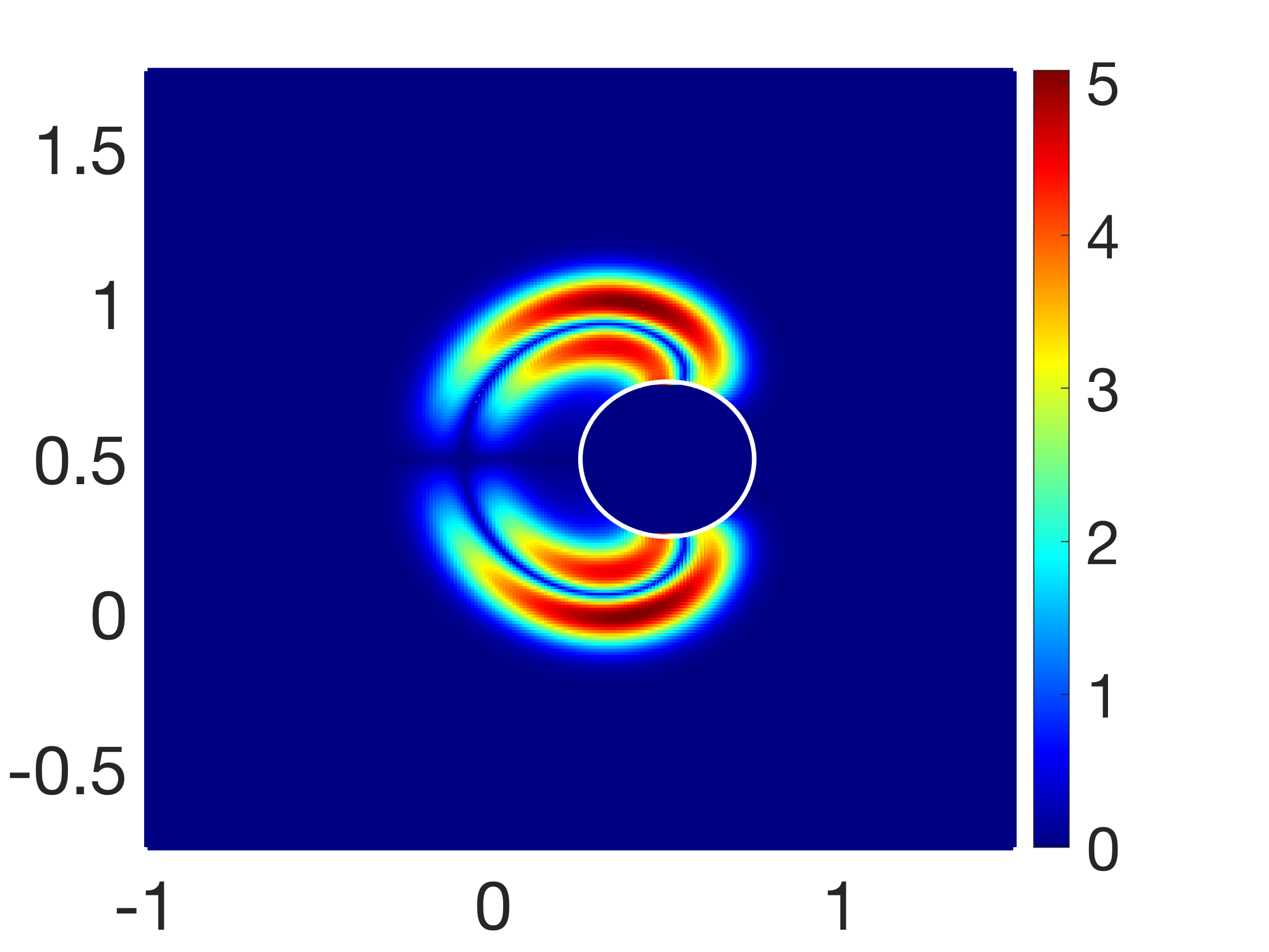}}
	{\includegraphics[width=2.25in]{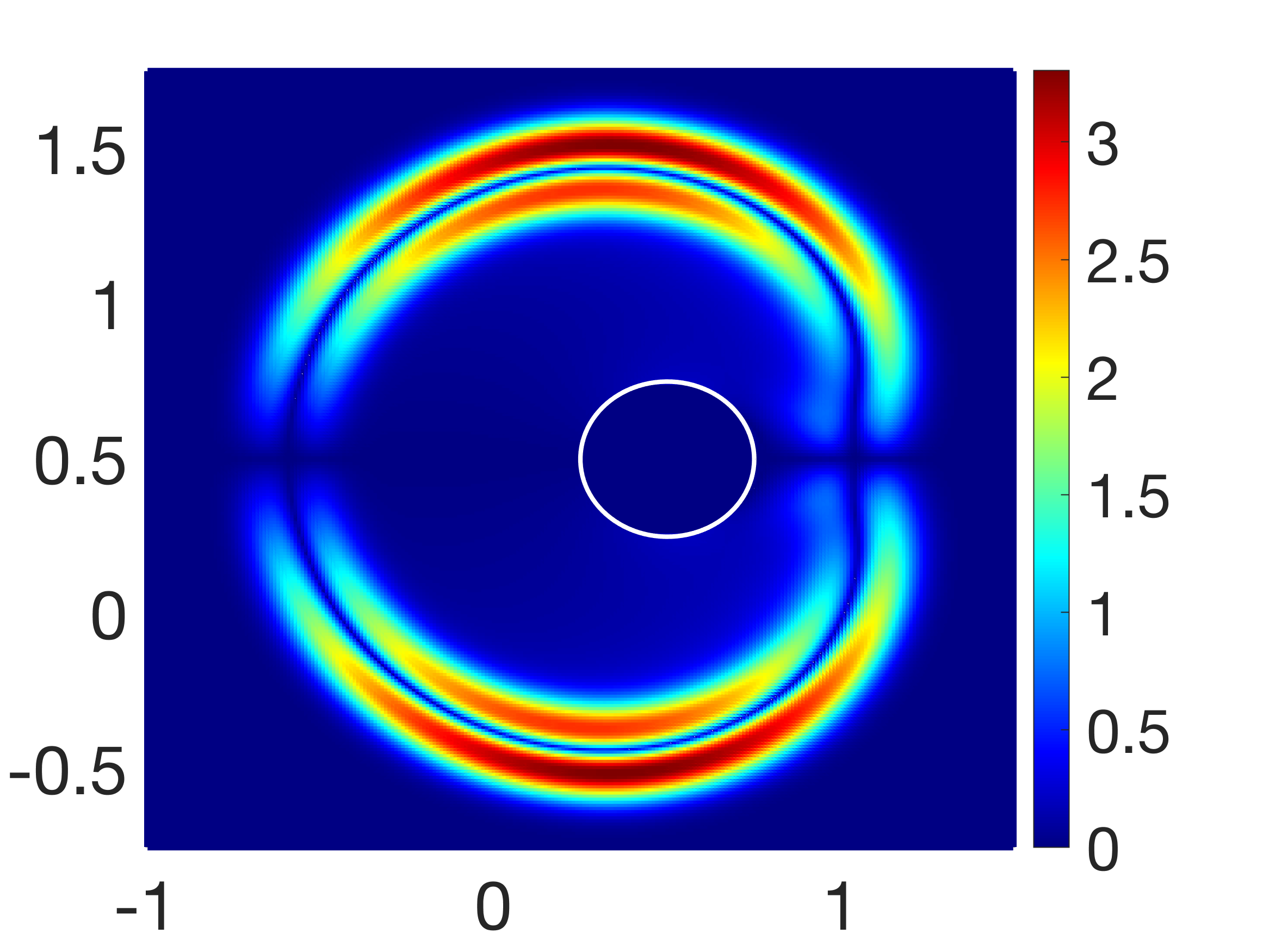}} 
			 }
\end{adjustbox}
\begin{adjustbox}{max width=1.25\textwidth,center}
\subfigure[$H_y$]{
 	{\includegraphics[width=2.25in]{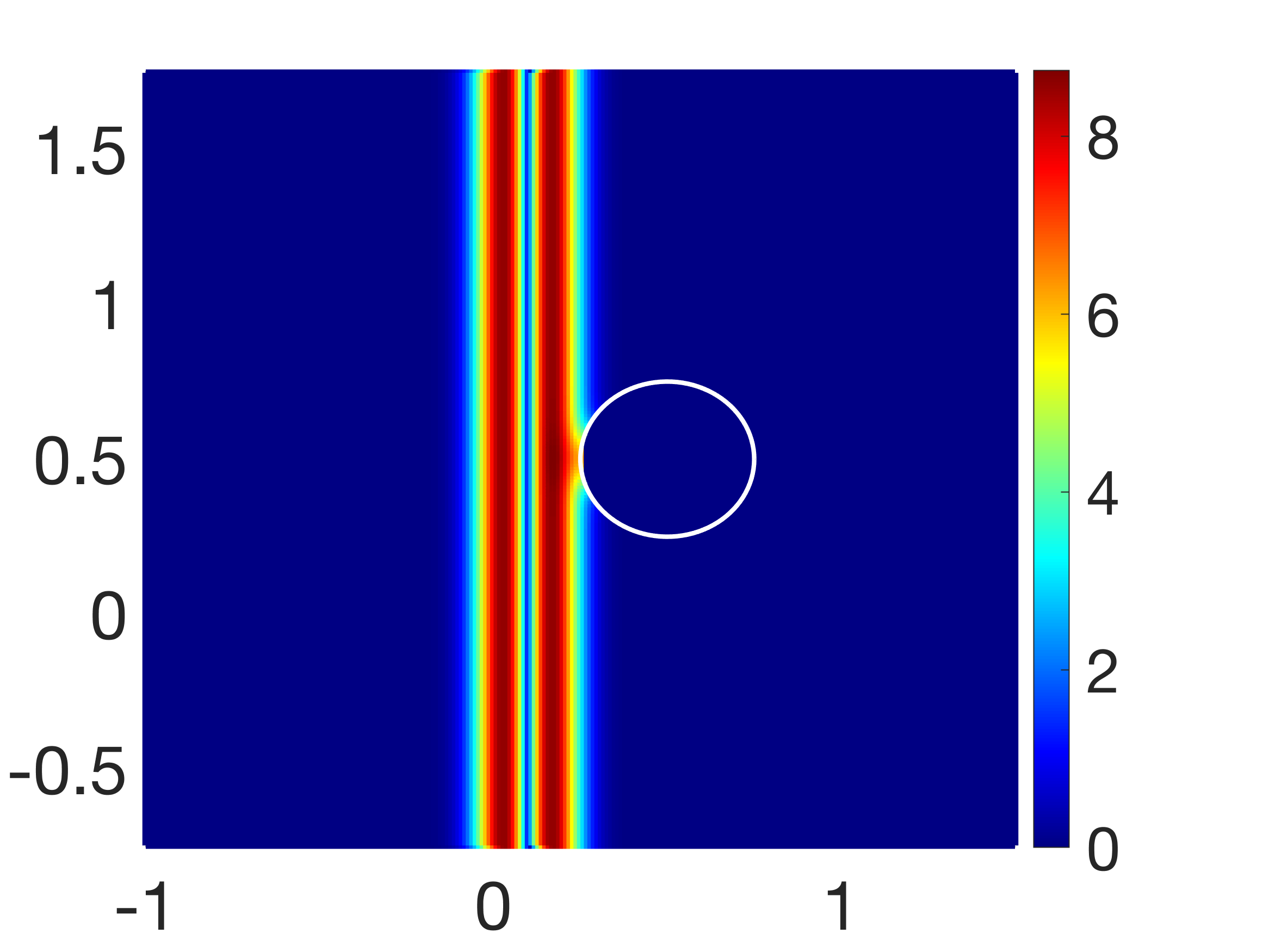}}
 	{\includegraphics[width=2.25in]{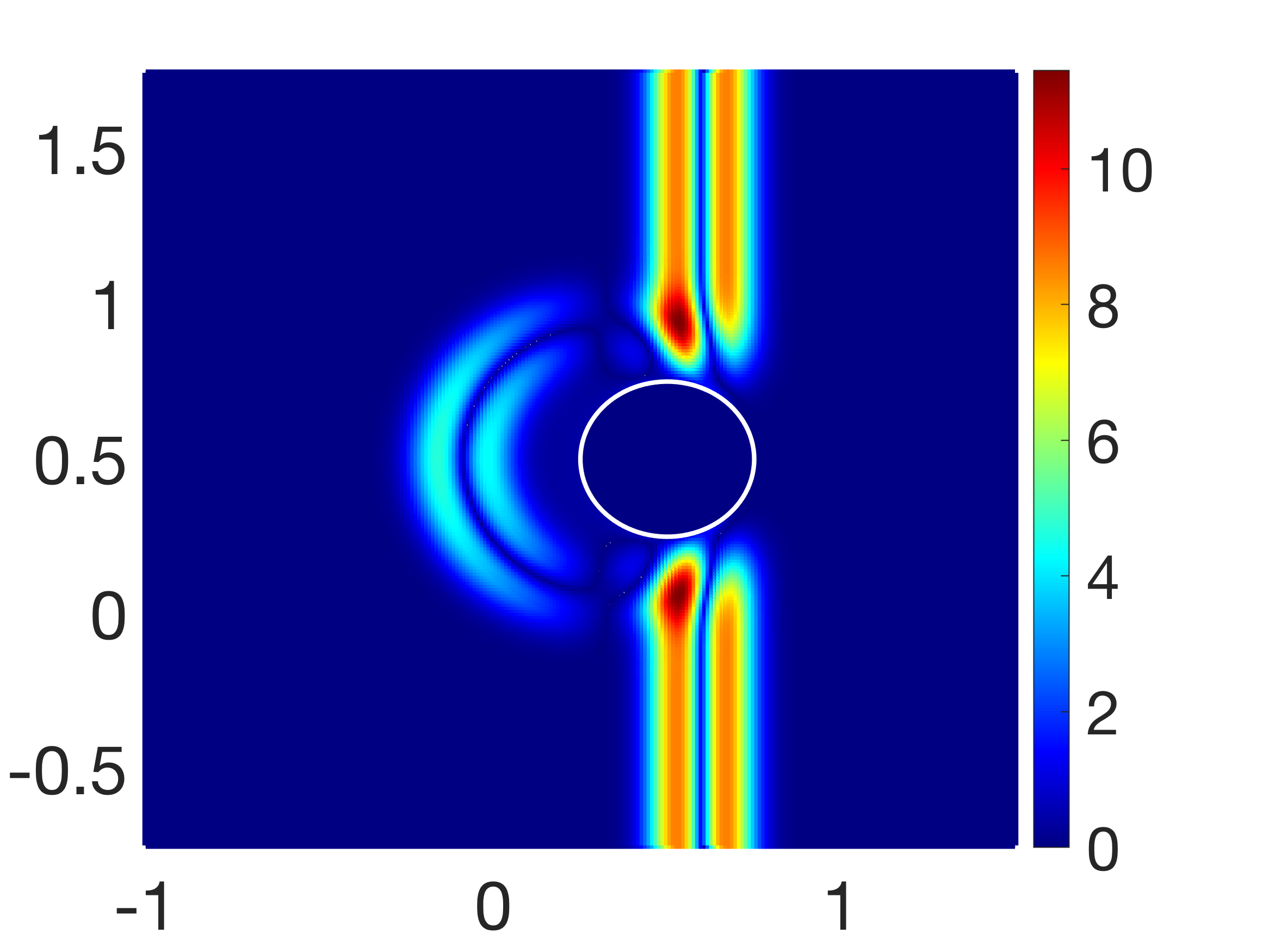}}
	{\includegraphics[width=2.25in]{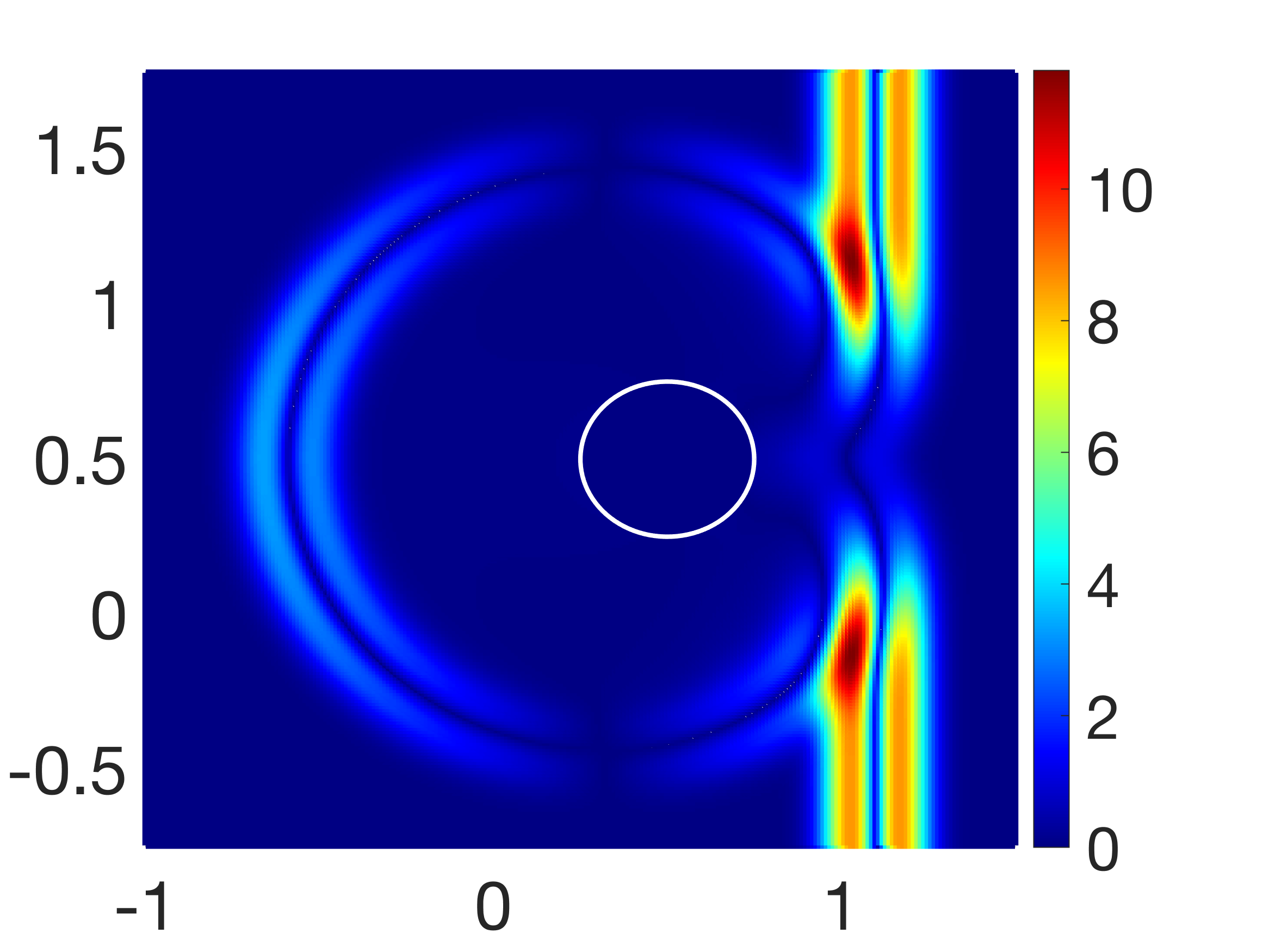}}  
			 }
\end{adjustbox}
\begin{adjustbox}{max width=1.25\textwidth,center}
\subfigure[$E_z$]{			 
	{\includegraphics[width=2.25in]{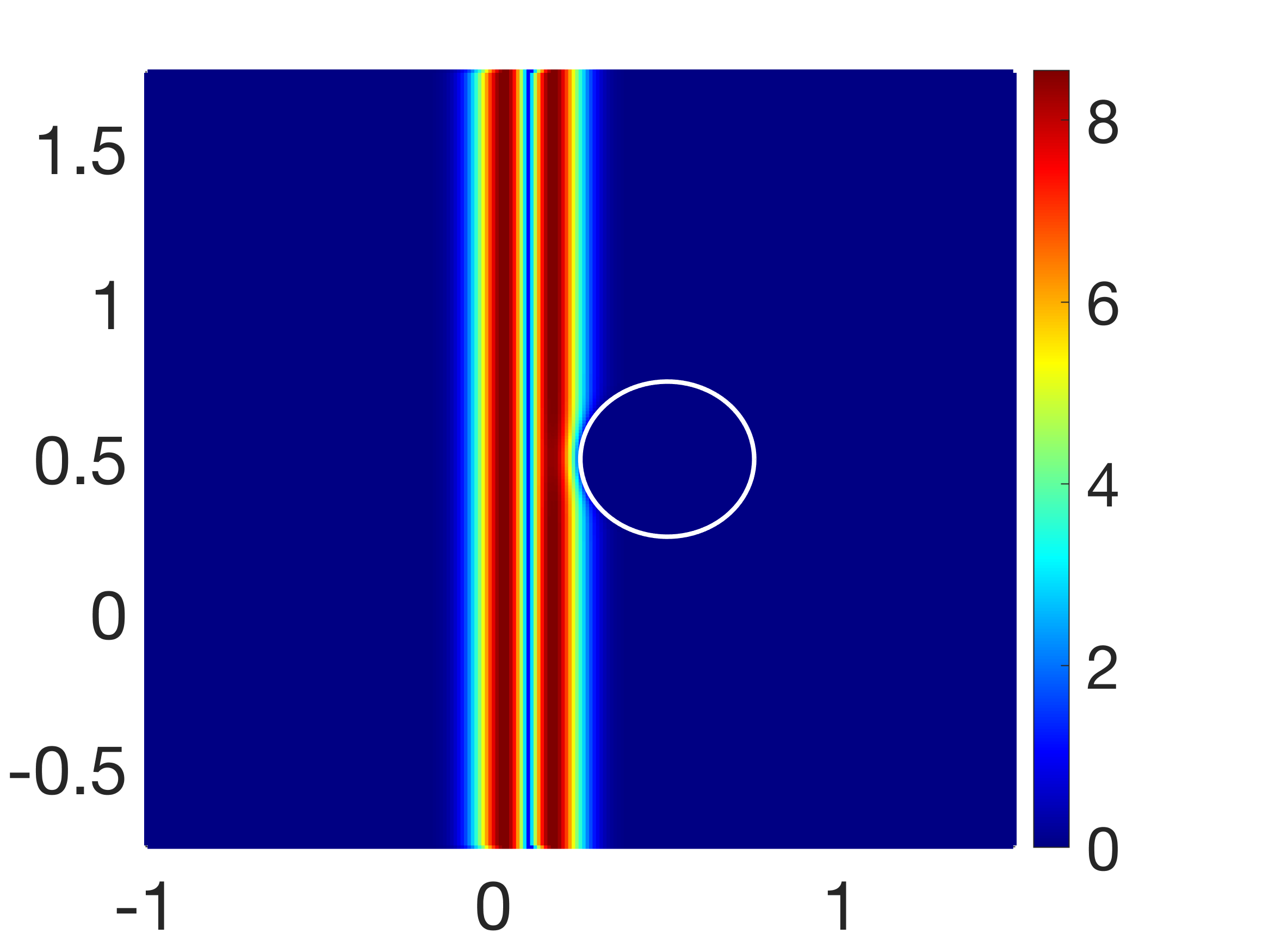}}
	{\includegraphics[width=2.25in]{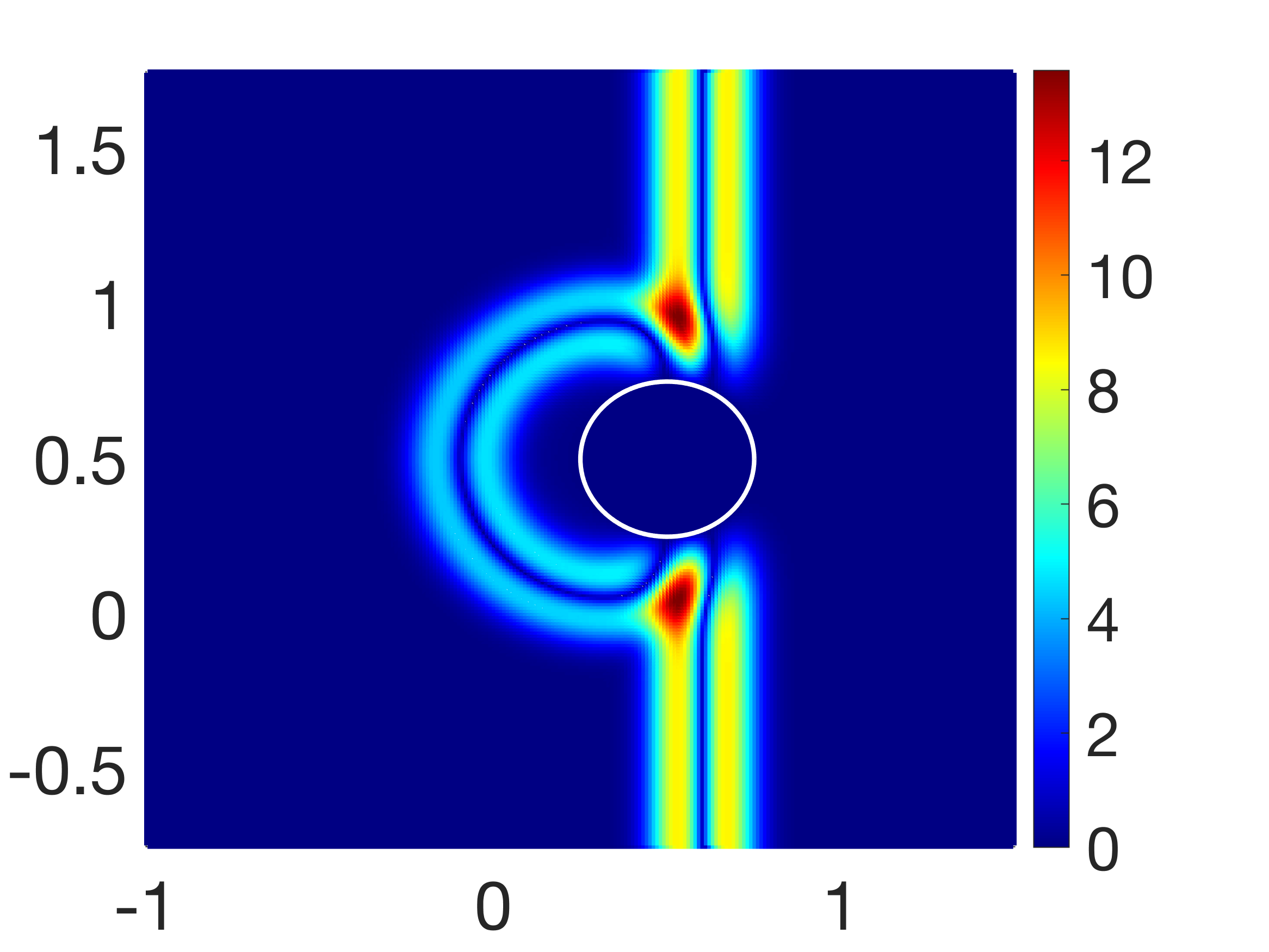}}
	{\includegraphics[width=2.25in]{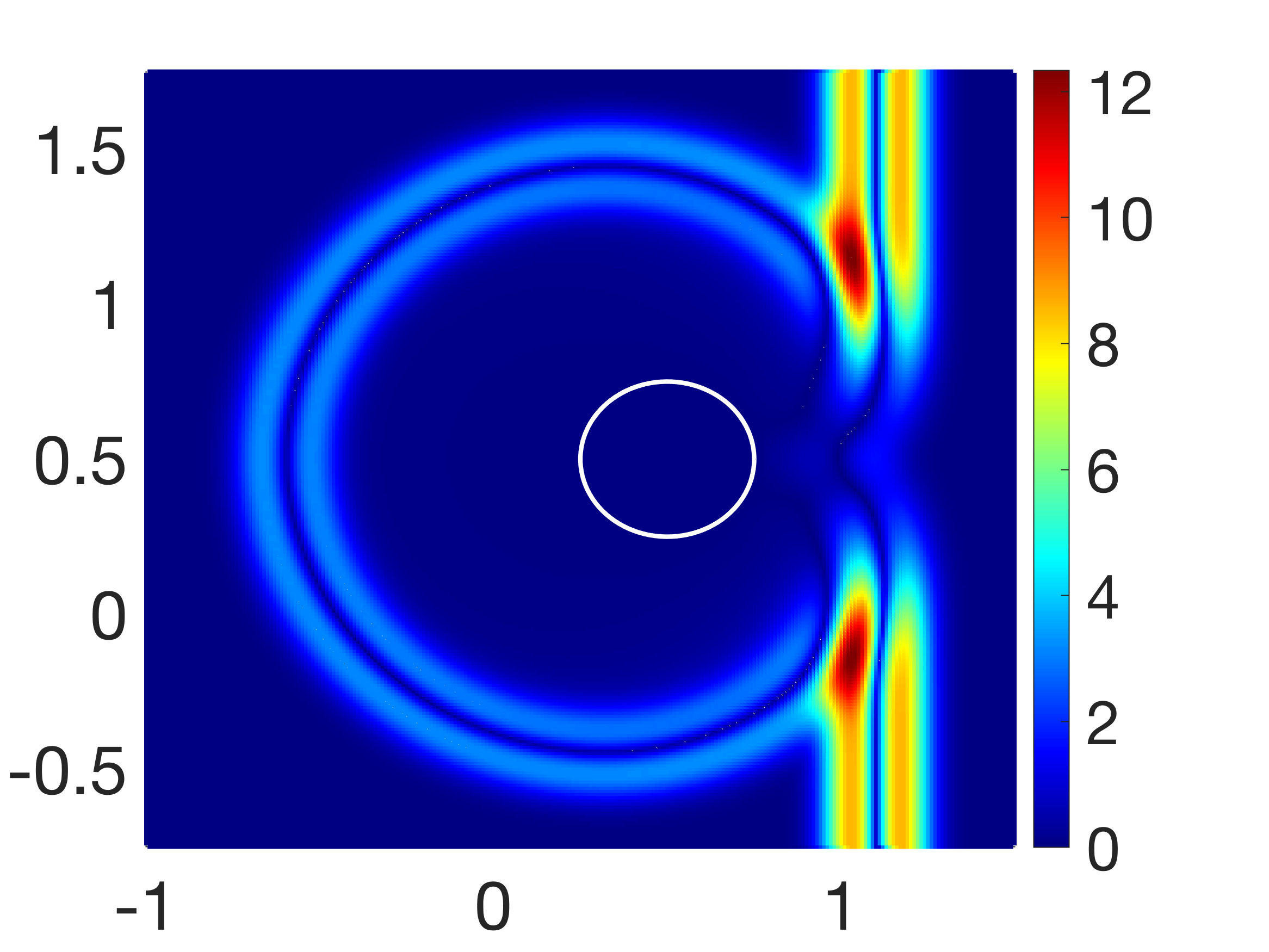}}
  }
\end{adjustbox}
  \caption{The evolution of the magnitude of components $H_x$, $H_y$ and $E_z$ with $h = \tfrac{1}{100}$ and $\Delta t = \tfrac{h}{2}$ using the CFM-$4^{th}$ scheme and the circular embedded PEC. From left to right, we show the computed electric field and 
  magnetic field at respectively $t \in \{ 0.4, 0.9, 1.4\}$ and $t-\tfrac{\Delta t}{2}$. The embedded boundary is represented by the white line.}
  \label{fig:circularScattering}  
\end{figure}
\begin{figure}
 \centering
\begin{adjustbox}{max width=1.25\textwidth,center}
\subfigure[$H_x$]{
 	{\includegraphics[width=2.25in]{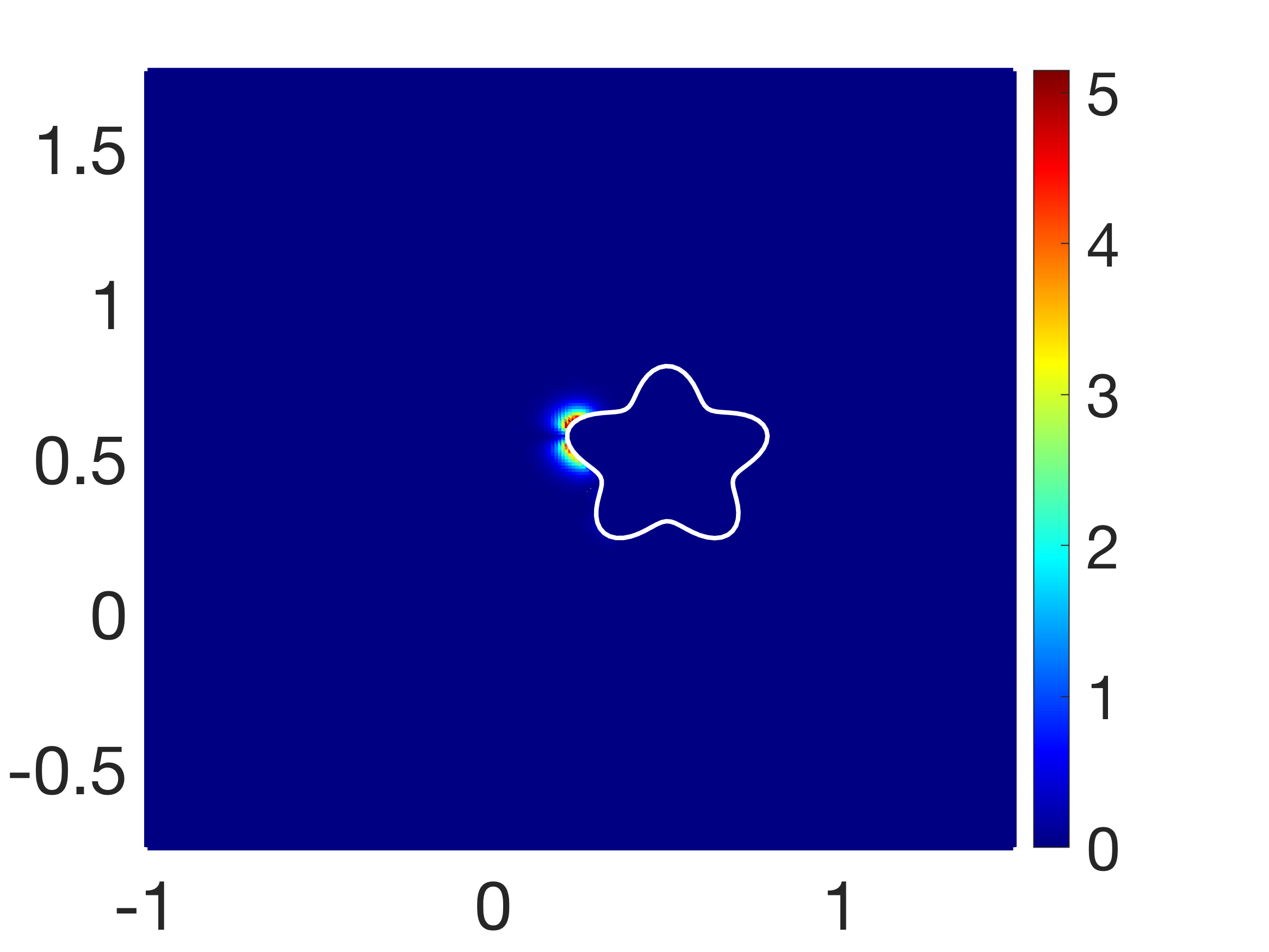}}
 	{\includegraphics[width=2.25in]{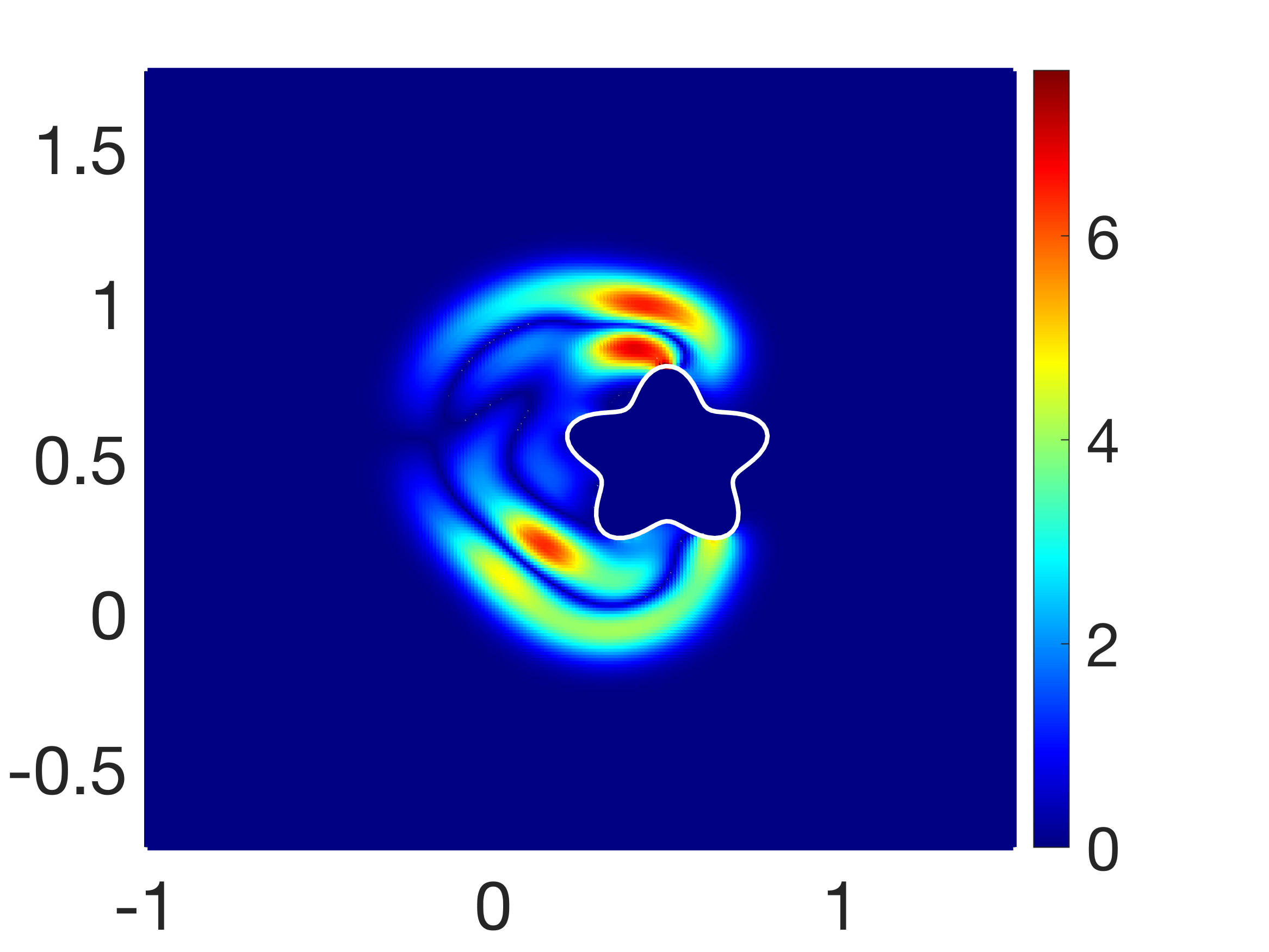}}
	{\includegraphics[width=2.25in]{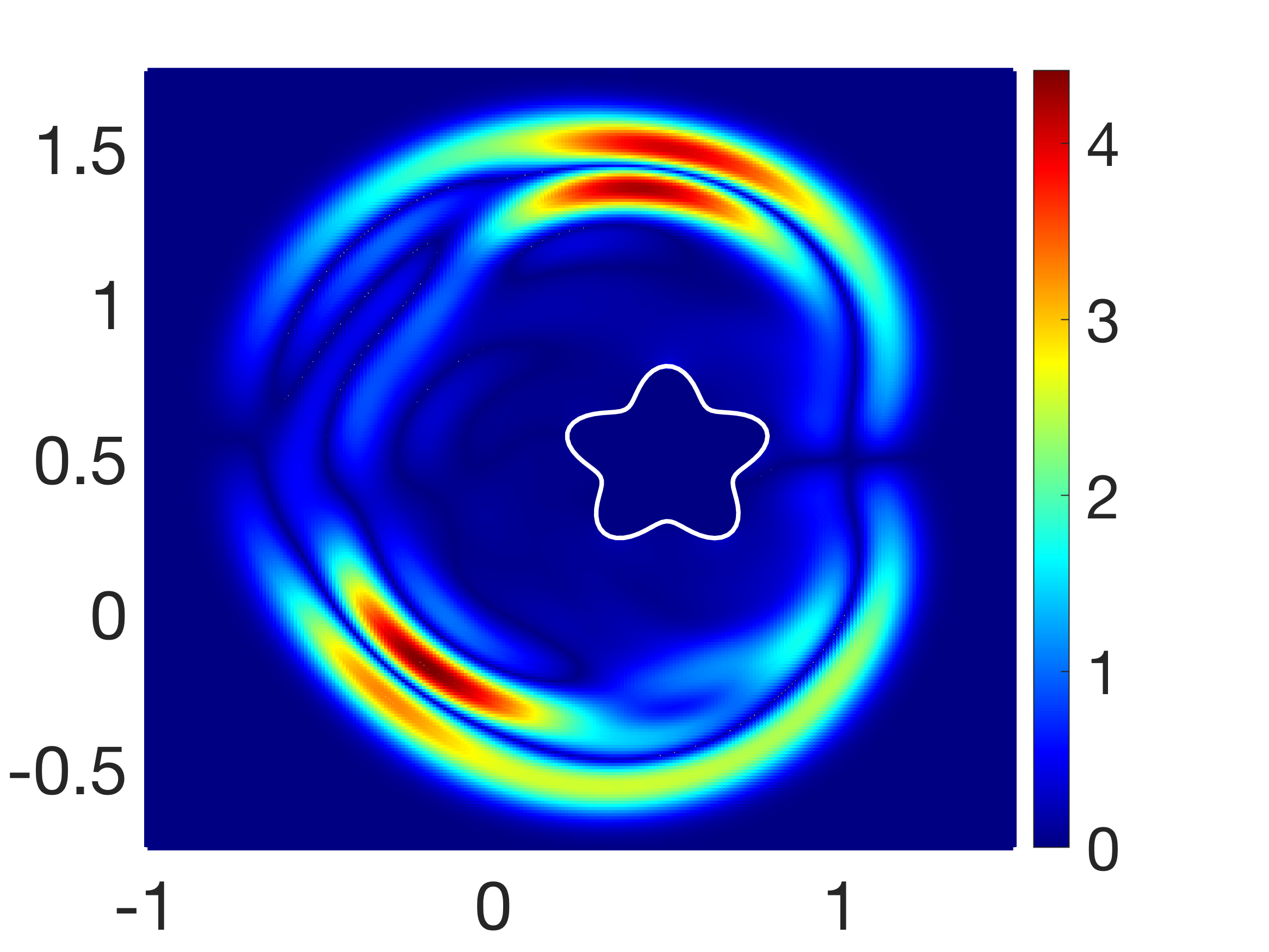}}  
			 }
\end{adjustbox}
\begin{adjustbox}{max width=1.25\textwidth,center}
\subfigure[$H_y$]{
 	{\includegraphics[width=2.25in]{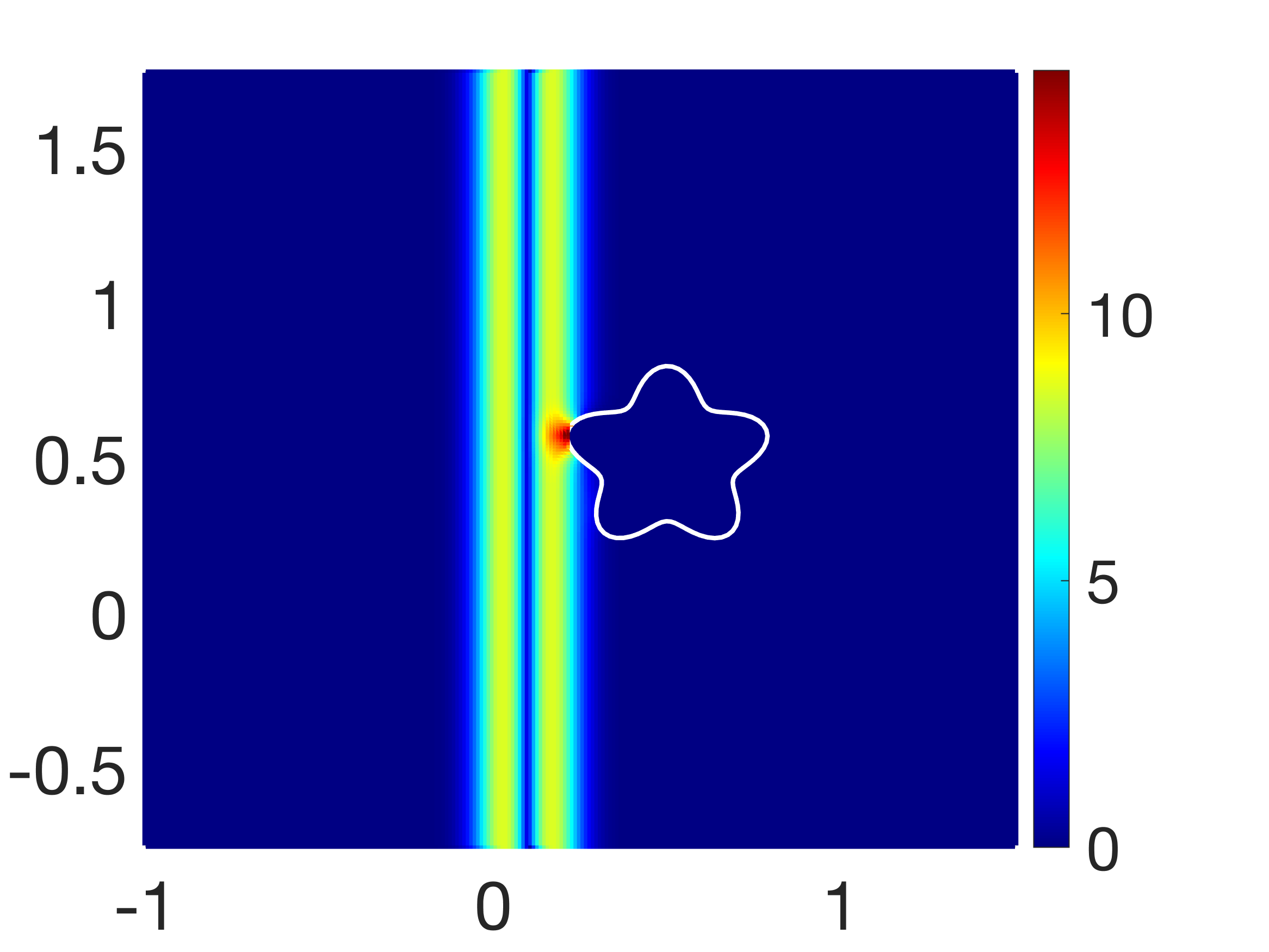}}
 	{\includegraphics[width=2.25in]{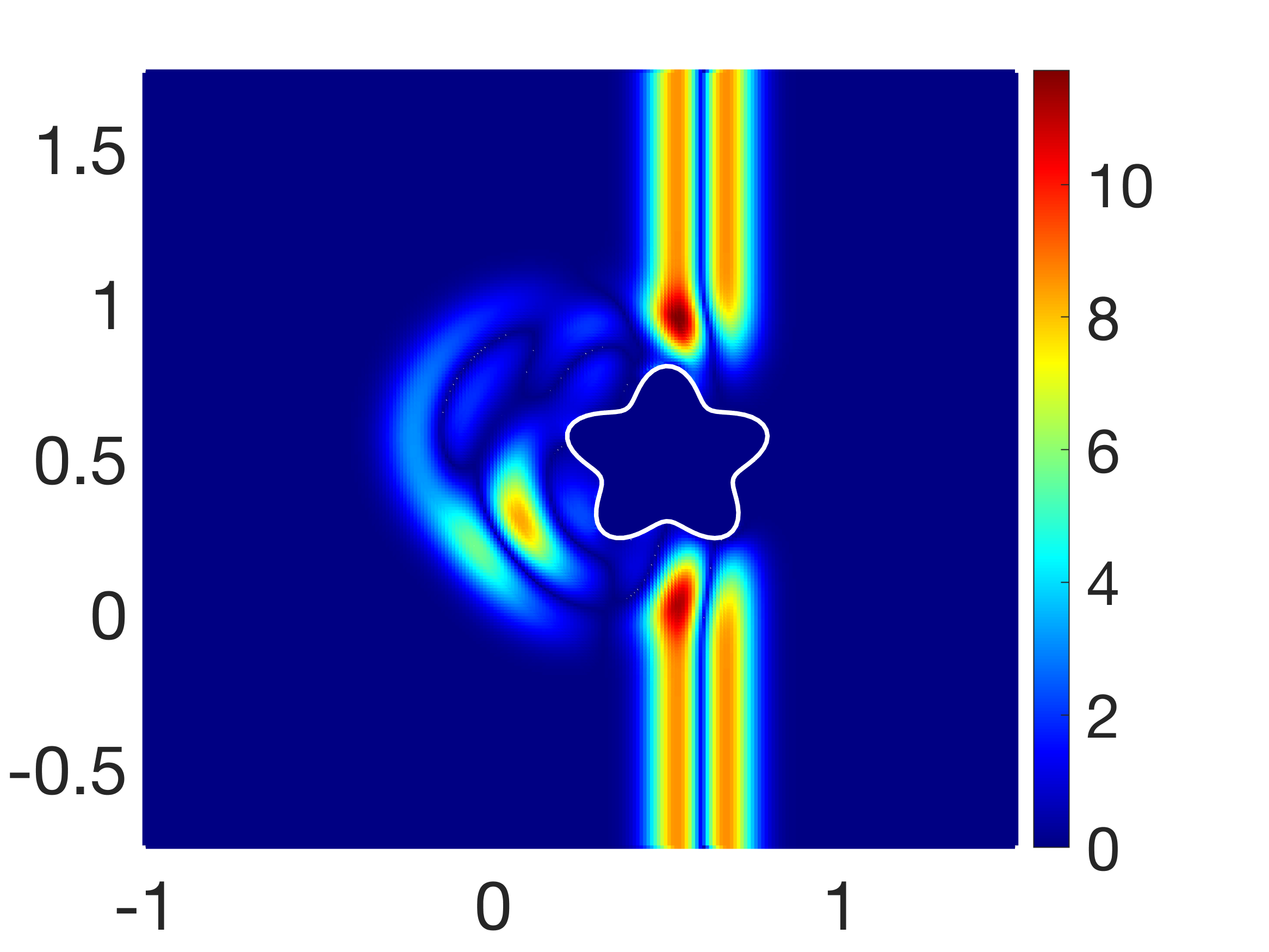}}
        {\includegraphics[width=2.25in]{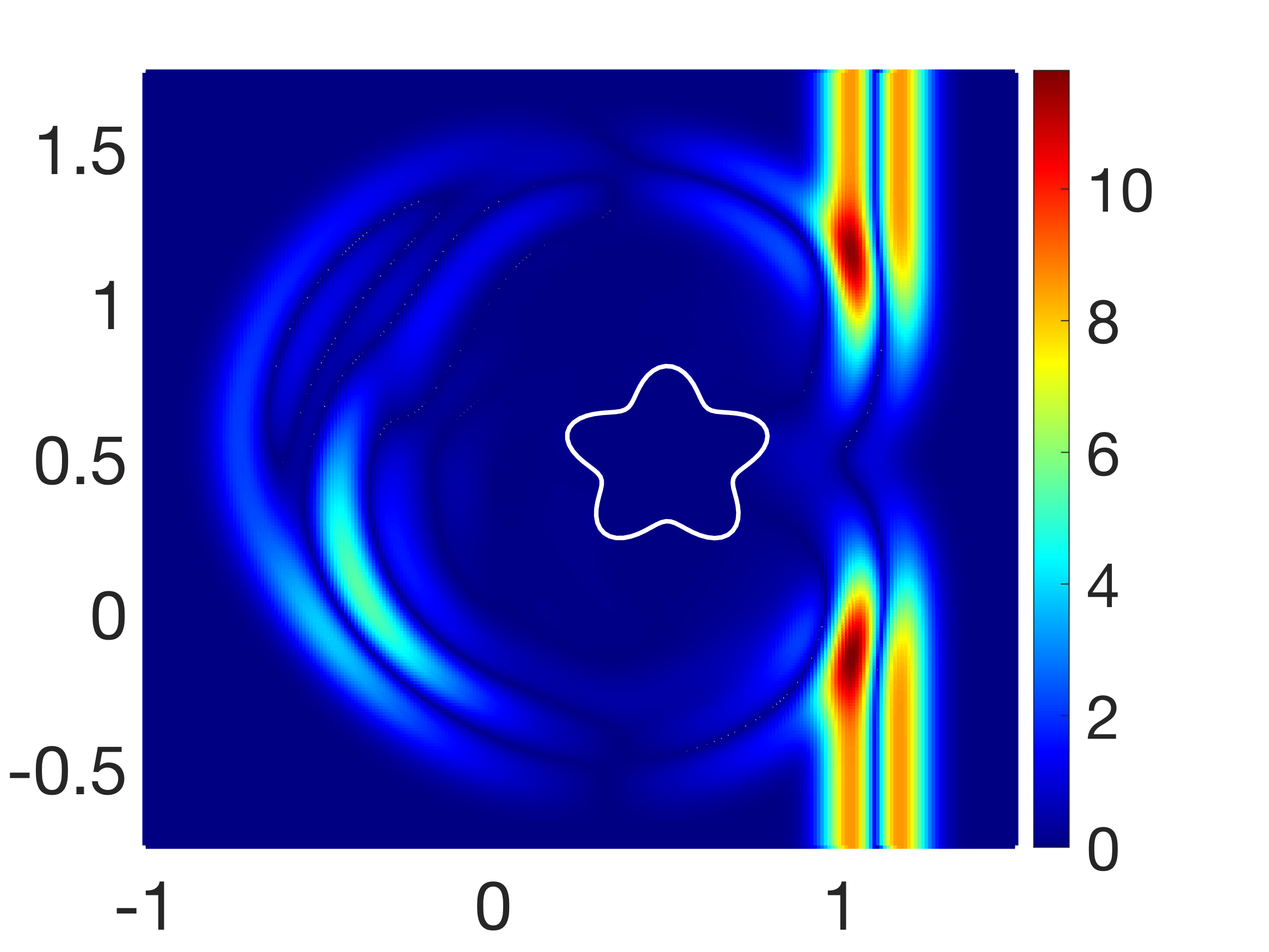}}  
			 }
\end{adjustbox}
\begin{adjustbox}{max width=1.25\textwidth,center}
\subfigure[$E_z$]{			 
	{\includegraphics[width=2.25in]{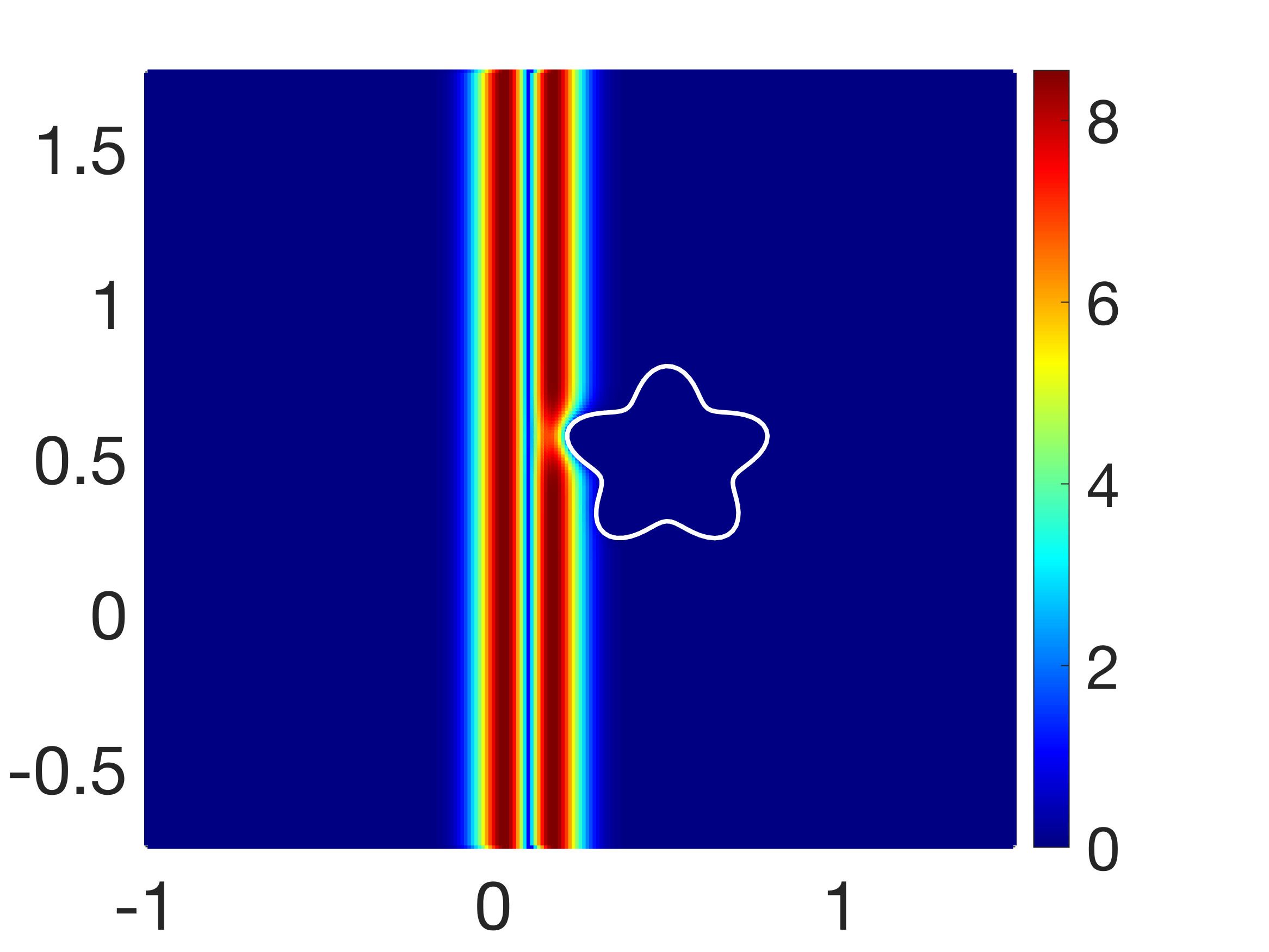}}
	{\includegraphics[width=2.25in]{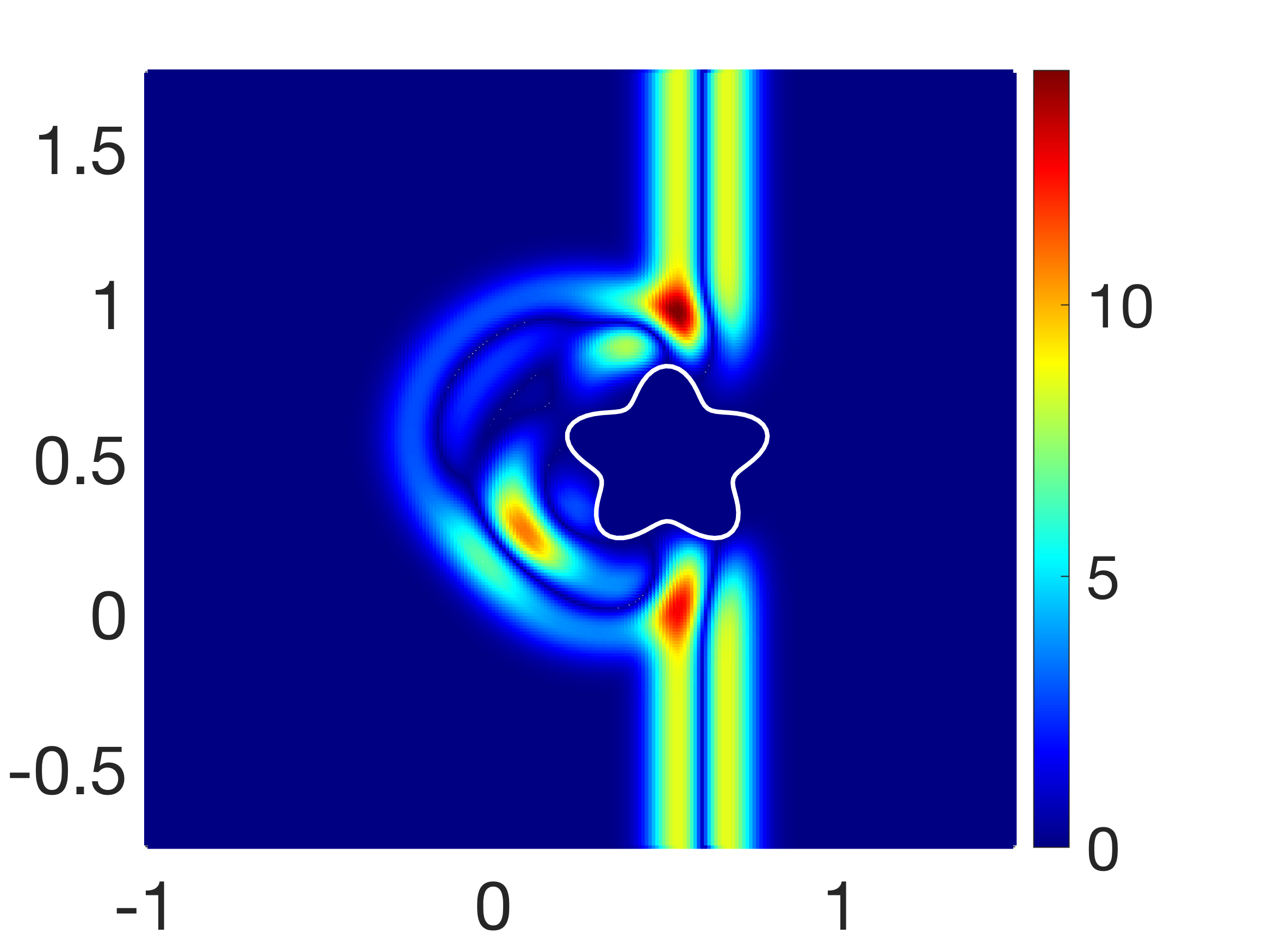}}
	{\includegraphics[width=2.25in]{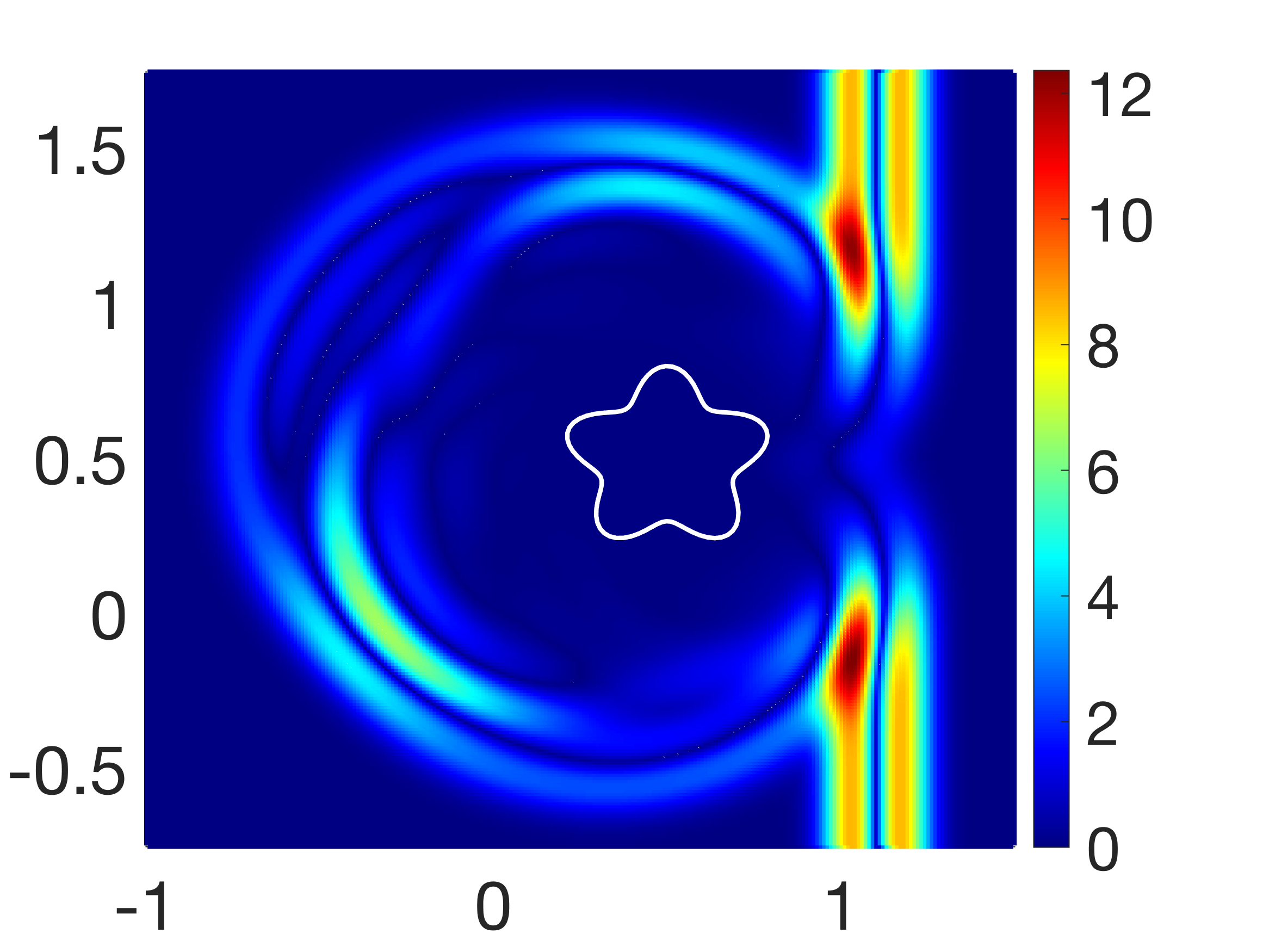}}
  }
\end{adjustbox}
  \caption{The evolution of the magnitude of components $H_x$, $H_y$ and $E_z$ with $h = \tfrac{1}{100}$ and $\Delta t = \tfrac{h}{2}$ using the CFM-$4^{th}$ scheme and the 5-star embedded PEC. From left to right, we show the computed electric field and 
  magnetic field at respectively $t\in \{ 0.4, 0.9, 1.4\}$ and $t-\tfrac{\Delta t}{2}$. The embedded boundary is represented by the white line.}
  \label{fig:fiveStarScattering}  
\end{figure}
\begin{figure}
\centering
\begin{adjustbox}{max width=1.25\textwidth,center}
\subfigure[$H_x$]{
 	{\includegraphics[width=2.25in]{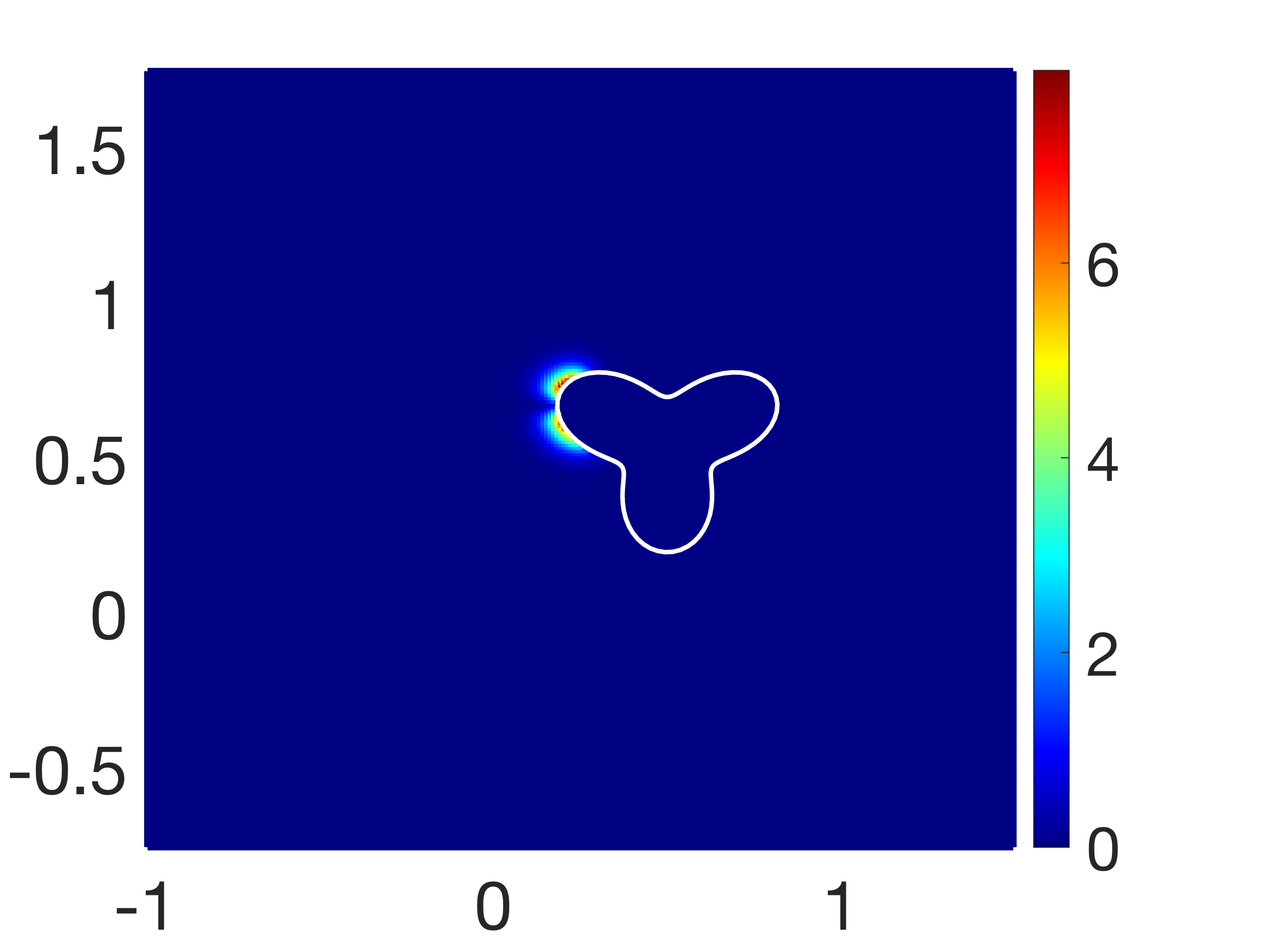}}
 	{\includegraphics[width=2.25in]{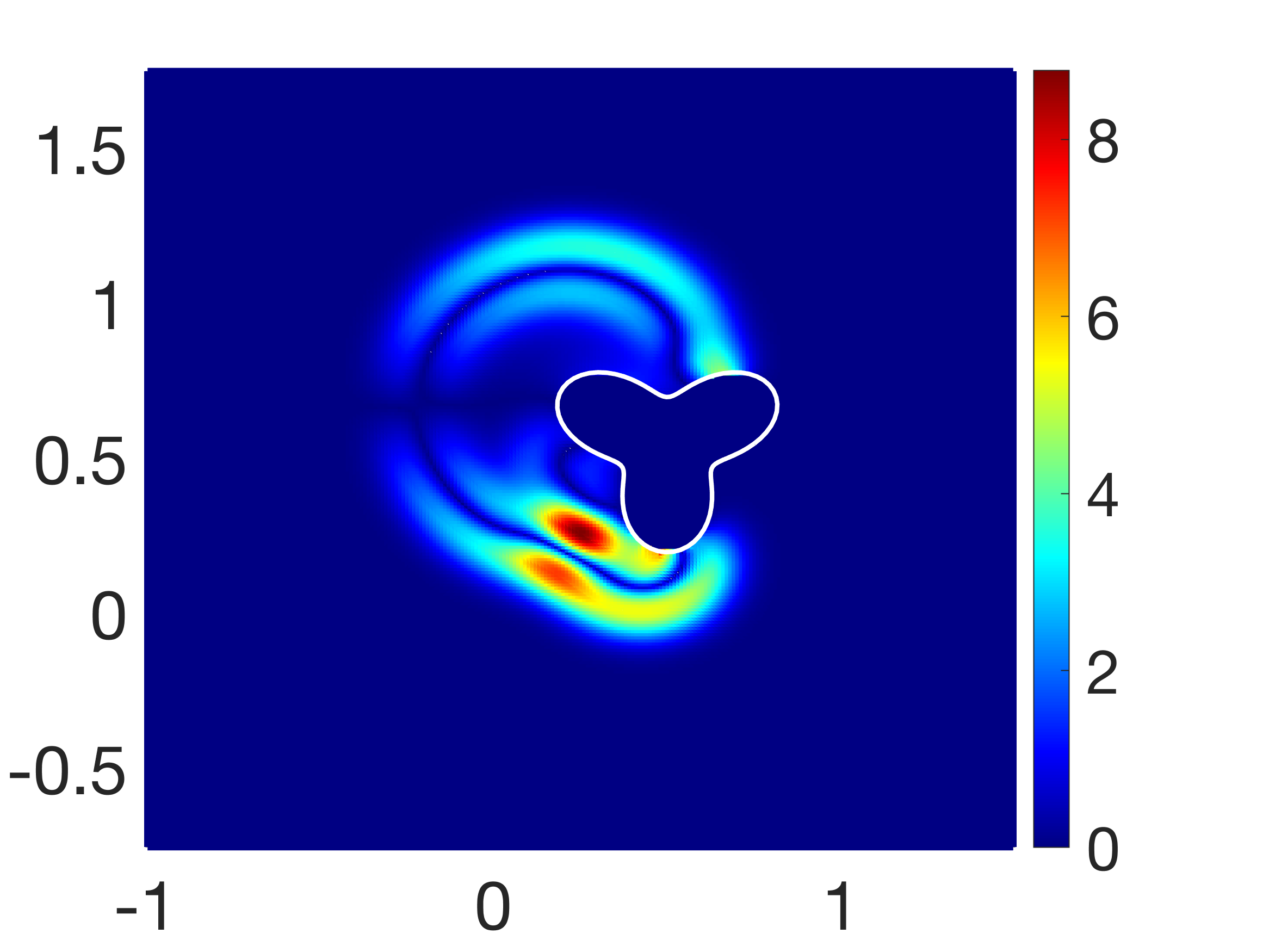}}
			 {\includegraphics[width=2.25in]{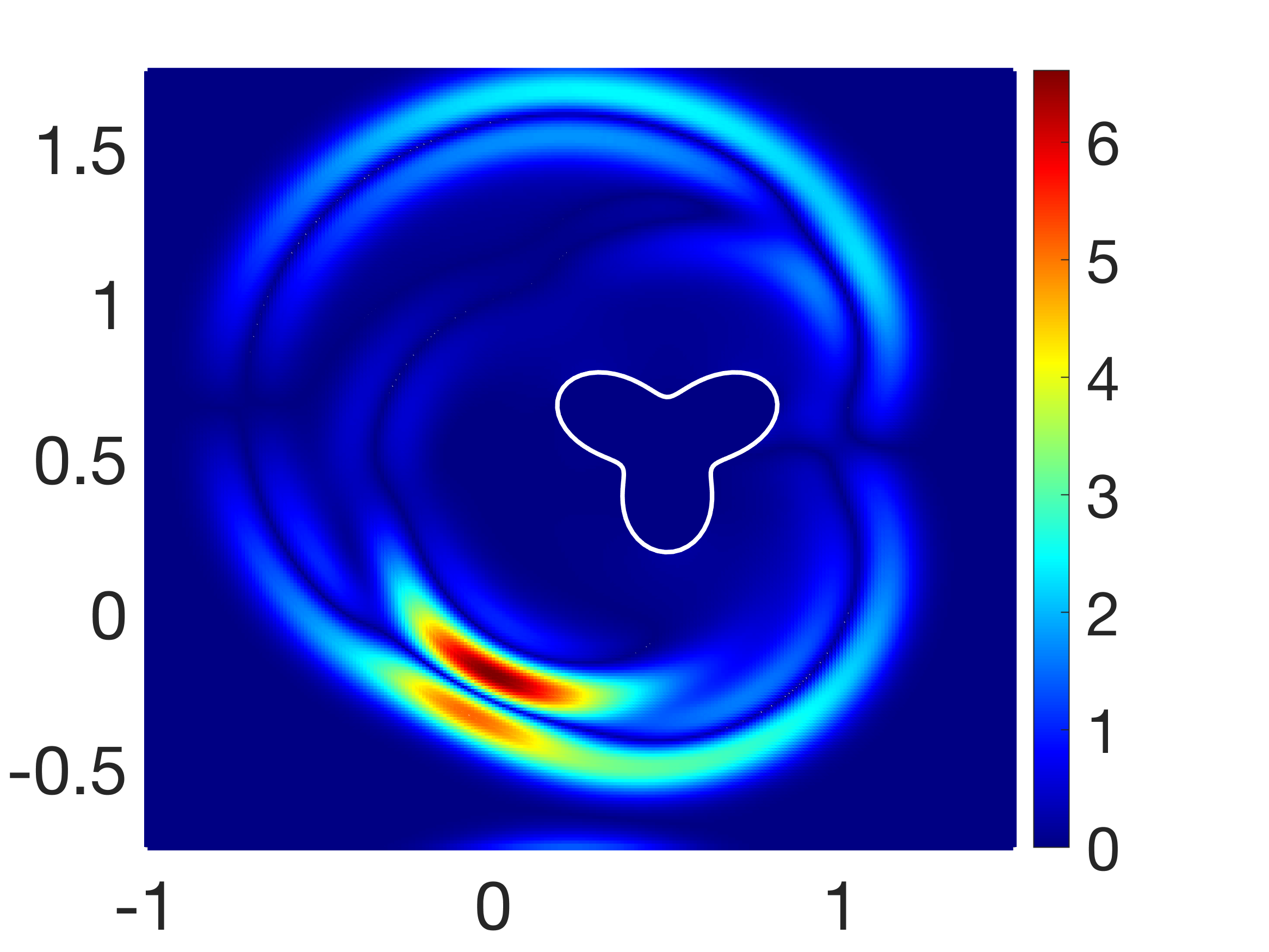}}  
			 }
\end{adjustbox}
\begin{adjustbox}{max width=1.25\textwidth,center}
\subfigure[$H_y$]{
 	{\includegraphics[width=2.25in]{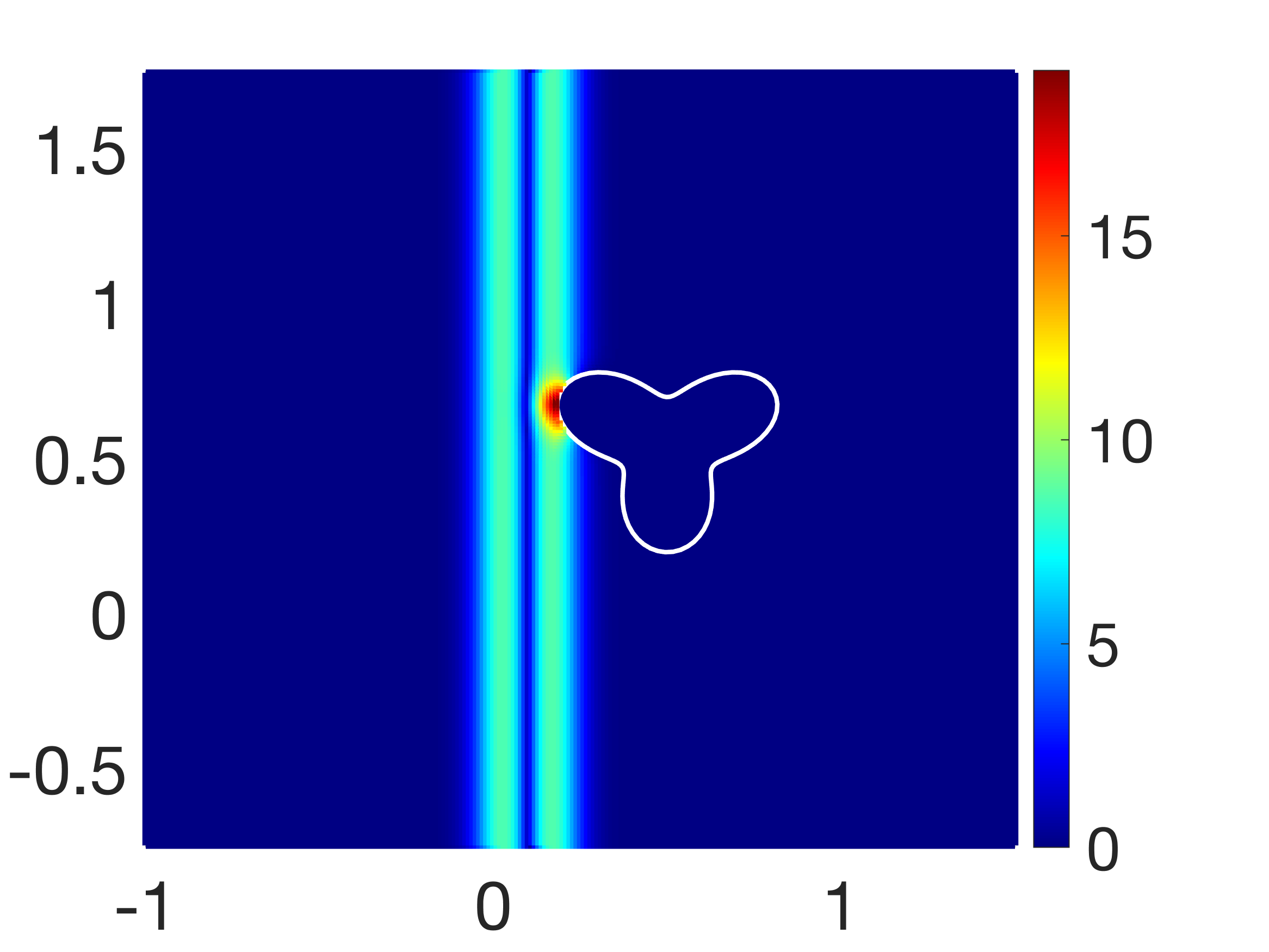}}
 	{\includegraphics[width=2.25in]{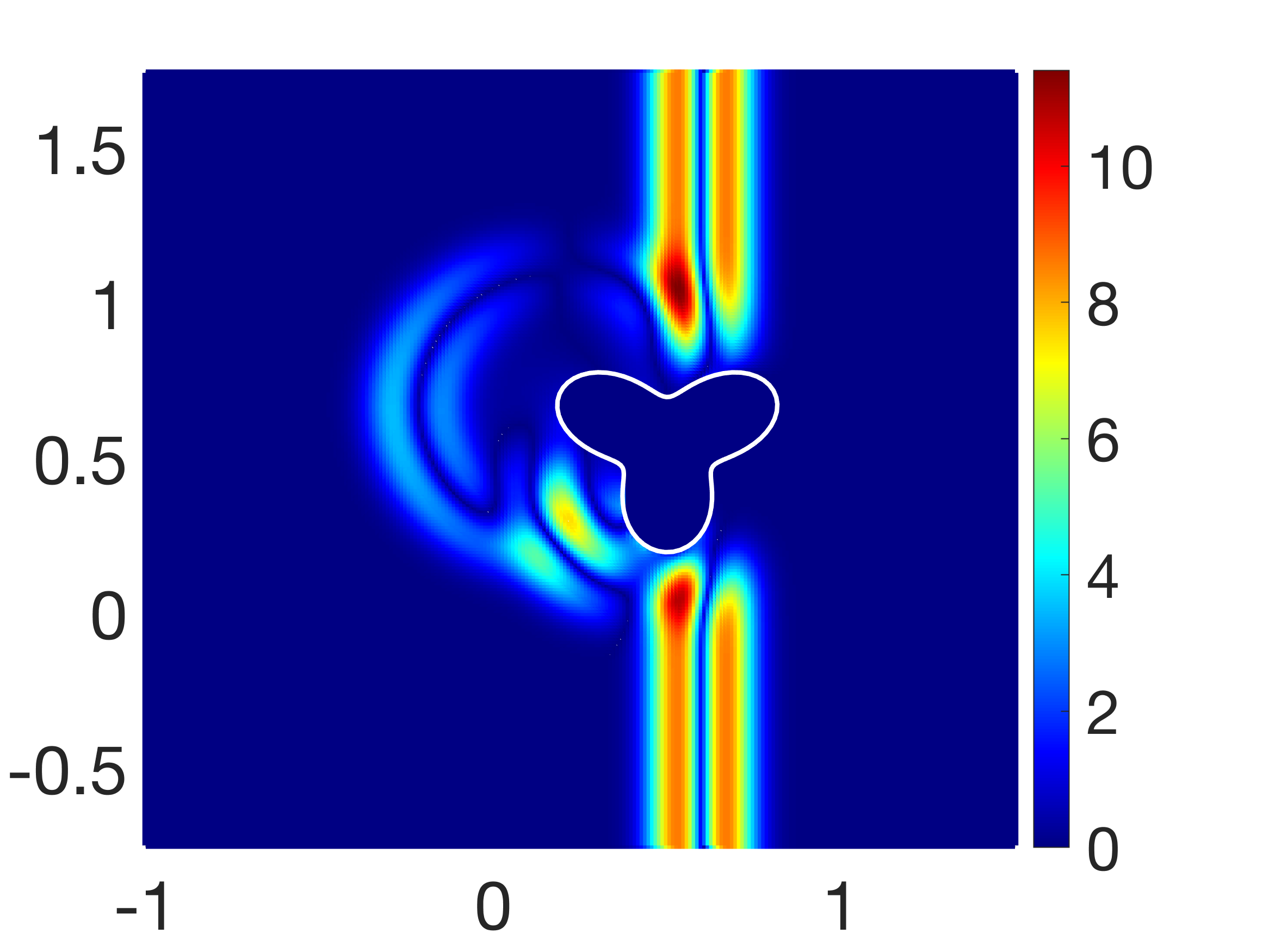}}
			 {\includegraphics[width=2.25in]{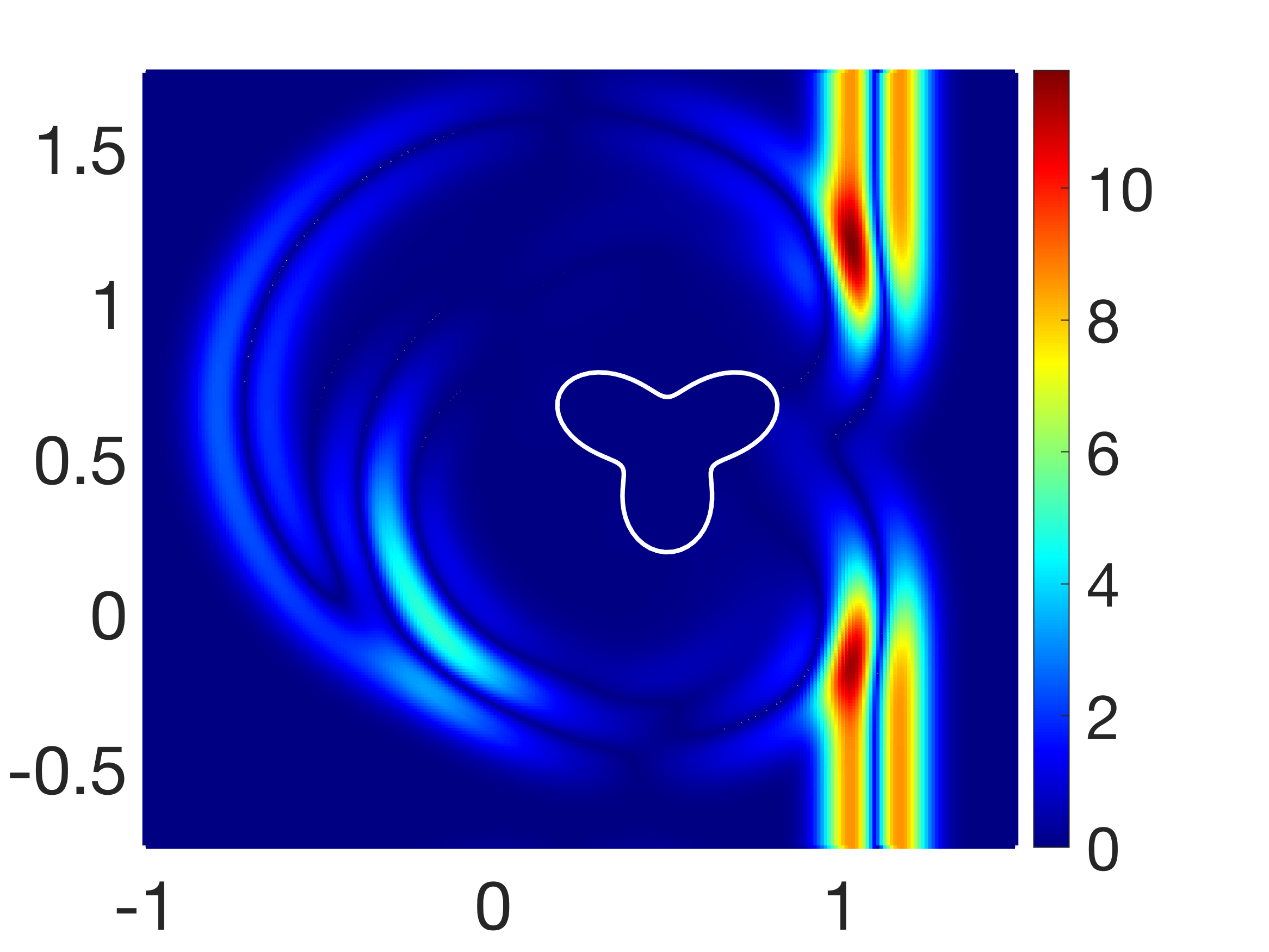}}  
			 }
\end{adjustbox}
\begin{adjustbox}{max width=1.25\textwidth,center}
\subfigure[$E_z$]{			 
	{\includegraphics[width=2.25in]{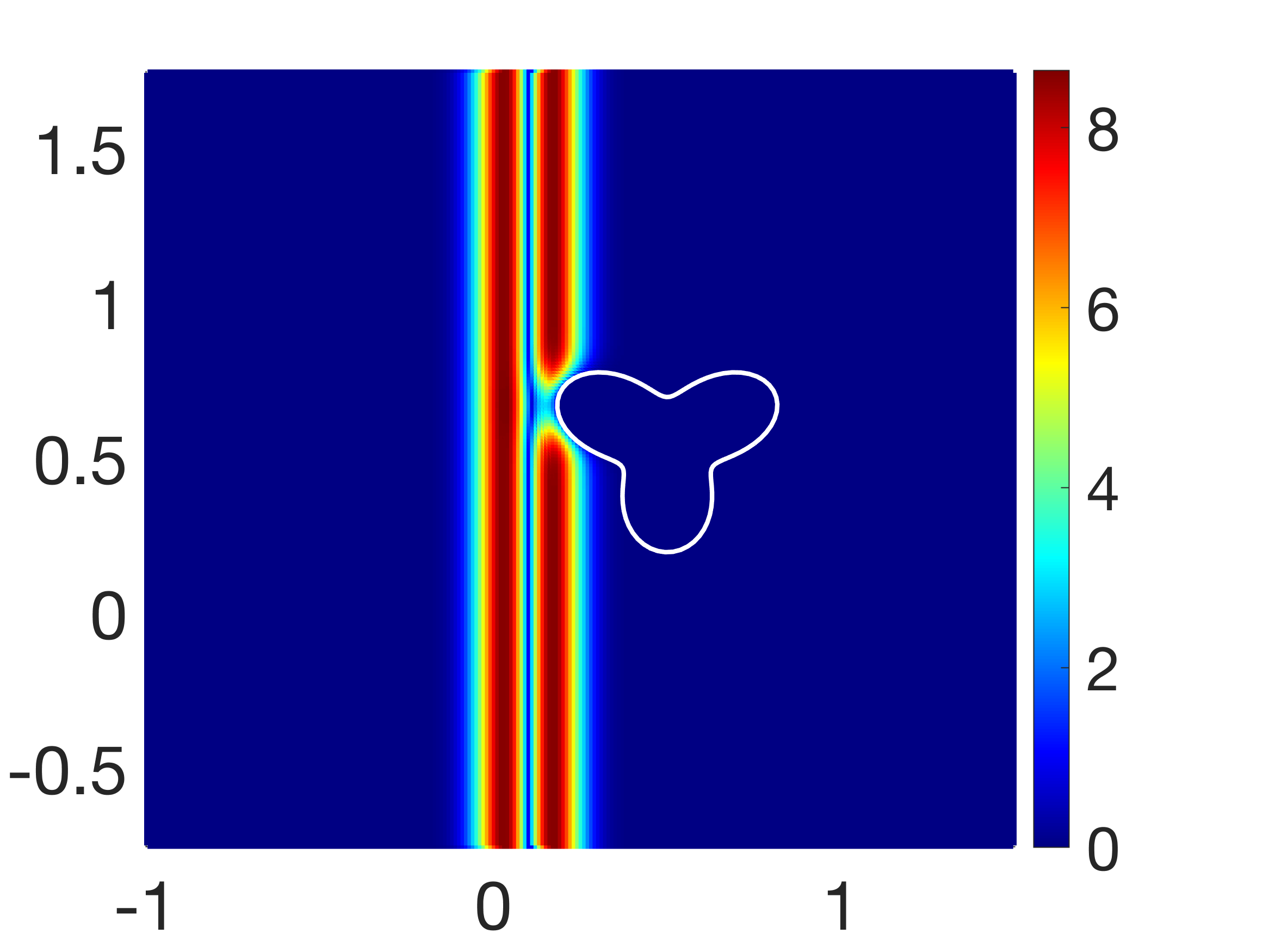}}
	{\includegraphics[width=2.25in]{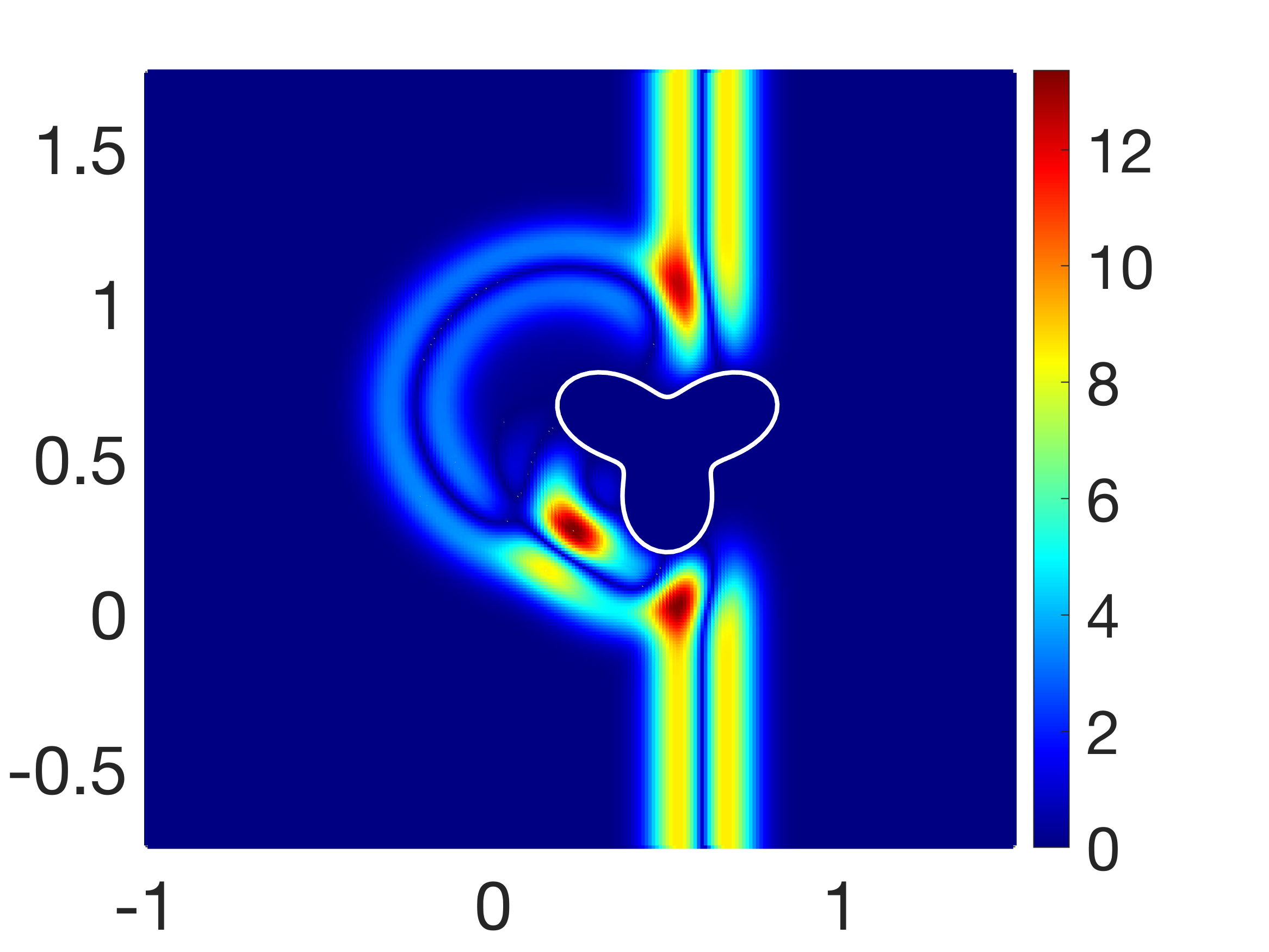}}
	{\includegraphics[width=2.25in]{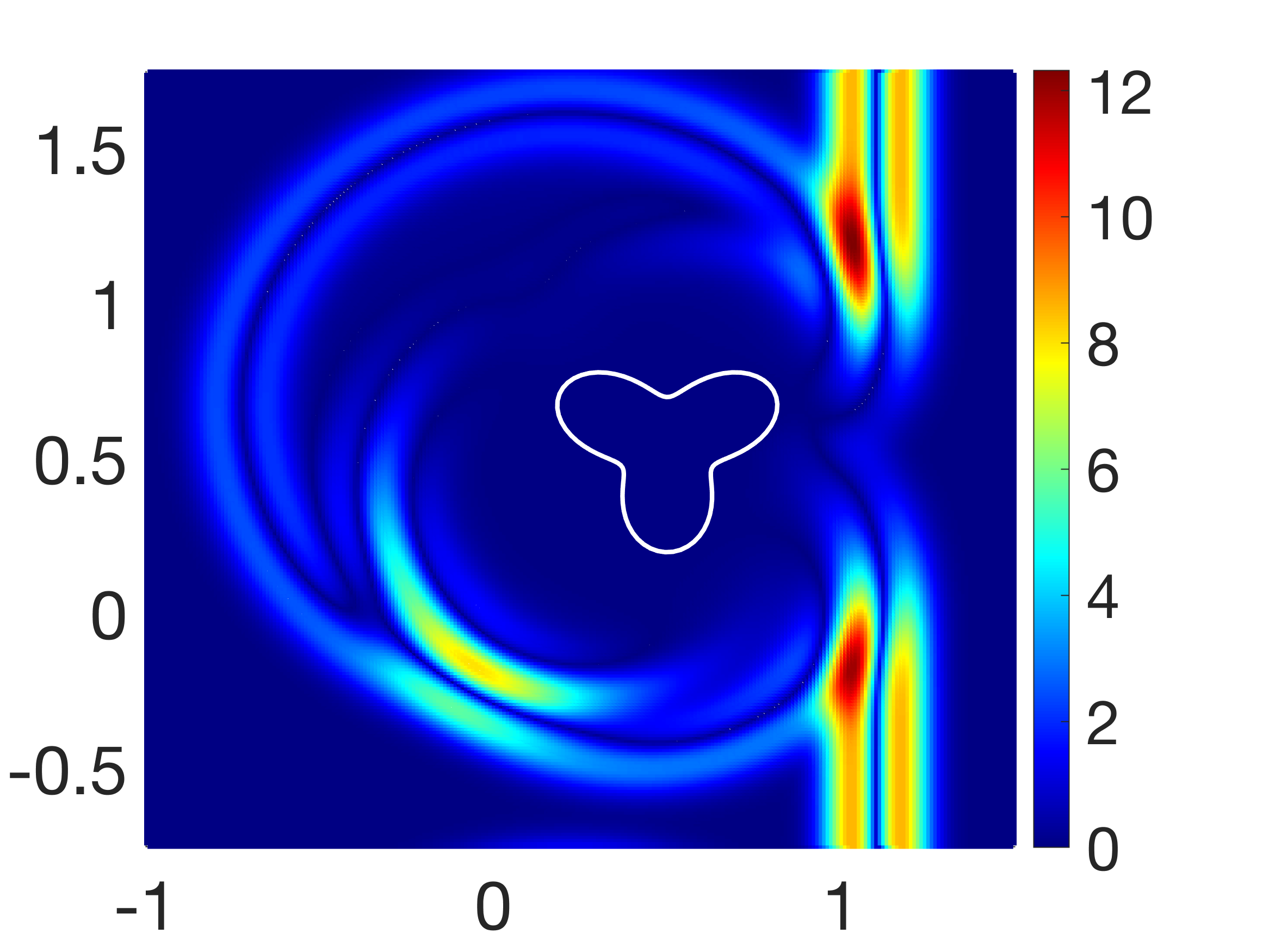}}
  }
\end{adjustbox}
  \caption{The evolution of the magnitude of components $H_x$, $H_y$ and $E_z$ with $h = \tfrac{1}{100}$ and $\Delta t = \tfrac{h}{2}$ using the CFM-$4^{th}$ scheme and the 3-star embedded PEC. From left to right, we show the computed electric field and 
  magnetic field at respectively $t\in \{ 0.4, 0.9, 1.4\}$ and $t-\tfrac{\Delta t}{2}$. The embedded boundary is represented by the white line.}
  \label{fig:triStarScattering} 
\end{figure}

\section{Conclusions}
This work proposes FDTD schemes based on the Correction Function Method for 
	Maxwell's equations with embedded PEC boundary conditions.
The associated minimization problems are well-posed for an appropriate representation 
	of the embedded boundary within 
	the local patch,
	and an appropriate choice of penalization coefficients and 
	fictitious interfaces.
Fictitious interfaces can induce some issues for long time simulations. 
The penalization coefficient associated with fictitious interface conditions 
	must be therefore chosen small enough to avoid any stability issues as the time-step size diminishes.
We have applied the CFM to the well-known Yee scheme and a fourth-order staggered FDTD scheme. 
It is worth mentioning that the proposed numerical strategy can be implemented as a black-box to existing 
	softwares. 
Based on numerical examples,
	it has been shown that the proposed CFM-FDTD schemes can handle embedded PEC 
	problems with
	various geometries of the boundary
	while retaining high-order convergence and 
	without significantly increasing the complexity of the proposed numerical approach.

\section*{Acknowledgments}
The authors are grateful to Dr. Marc Laforest and Dr. Serge Prudhomme of Polytechnique Montr{\'e}al for their support.
The authors also thank Damien Tageddine for helpful conversations. 
The research of JCN was partially supported by the NSERC Discovery Program.
This is a pre-print of an article published in Journal of Scientific Computing. The final authenticated version is available online at: https://doi.org/10.1007/s10915-021-01591-z.

\appendix 

\section{Truncation Error Analysis}
As shown in \cite{LawMarquesNave2020}, 
	the CFM can reduce the order in space of an original FD scheme for unsteady problems. 
Proposition~\ref{lem:errorAnalysis} provides a general result on the order in space of a corrected FD scheme.
\begin{proposition} \label{lem:errorAnalysis}
Let us consider a domain $\Omega$, 
	a time interval $I$ and 
	an interface $\Gamma \subset~\Omega$ on which there is interface jump conditions.  
Assume that the correction function coming from the CFM 
	is smooth enough and is such that 
\begin{equation} \label{eq:semiDiscreteForm}
	\partial_t \hat{\mathbold{U}} + L\,(\hat{\mathbold{U}}+A\,\hat{\mathbold{D}})  = \mathbold{F},
\end{equation}
	where $\hat{\mathbold{U}}$ is the vector of true solution values, 
	$A$ is a rectangular matrix with either $0$ or $\pm 1$ as components,
	$\hat{\mathbold{D}}$ is the vector of true correction function values, 
	$L$ is a spatial finite difference operator of order $n$ that approximates $q$-order derivatives 
	and $\mathbold{F}$ is a source term.
A $(k+1)$-order approximation of the correction function 	
	leads to a corrected FD scheme of order $\min\{n,k-q+1\}$ in space. 
\end{proposition}
\begin{proof}
A $(k+1)$-order approximation of $\hat{\mathbold{D}}$ leads to 
\begin{equation*}
	\mathbold{D} = \hat{\mathbold{D}} + \mathcal{O}(\ell_h^{k+1}),
\end{equation*}
	where $\ell_h = \beta\,h$ is the length of the space-time patch, 
	$\beta$ is a positive constant and $h$ is the mesh grid size. 
The discrete operator $L$,
	that approximates $q$-order derivatives, 
	involves components scaled by a factor $\tfrac{1}{h^q}$.
Hence, 
	$L\,A\,\mathbold{D} = L\,A\,\hat{\mathbold{D}} + \mathcal{O}(\ell_h^{k+1}\,h^{-q}) =  L\,A\,\hat{\mathbold{D}} + \mathcal{O}(h^{k-q+1})$.
\end{proof}
For problems that do not involve transient derivatives, 
	we have $$\hat{\mathbold{U}} = L^{-1}\,\mathbold{F} - A\,\hat{\mathbold{D}}$$ 
	and the order of the corrected FD scheme is then $\min\{n,k+1\}$.
\section*{References}

\bibliography{references}

\end{document}